\documentclass[10pt, leqno]{amsart}
\usepackage{array}
\usepackage{longtable}
\usepackage{url}
\usepackage{enumerate}
\usepackage{amsmath}
\usepackage{amssymb}
\usepackage{amsthm}
\usepackage{amsfonts}
\usepackage{tikz,pgf}
\usepackage{graphicx}
  \usepackage{ifthen}
 \newcommand{\mymarginpar}[1]{%
    \marginpar{\ifthenelse{\isodd{\arabic{page}}}{\flushleft 
#1}{\flushright #1}}}

\numberwithin{equation}{section}
 
 \newcommand{\pphi}{\varphi}
 \newcommand{\eps}{\varepsilon}           
 
 \newcommand{\IN}{\mathbb{N}}

 \newcommand{\IZ}{\mathbb{Z}}         
                  
\newcommand{\CA}{\mathcal{A}}
\newcommand{\CB}{\mathcal{B}}

\newcommand{\CD}{\mathcal{D}}
\newcommand{\CE}{\mathcal{E}}

\newcommand{\CS}{\mathcal{S}}

\newcommand{\CO}{\mathcal{O}}
\newcommand{\CP}{\mathcal{P}}

\newcommand{\CT}{\mathcal{T}}
\newcommand{\CR}{\mathcal{R}}

\newcommand{\CU}{\mathcal{U}}

\newcommand{\CZ}{\mathcal{Z}}

\newcommand{\cstar}{$C^*$}

 \theoremstyle{plain} 
 \newtheorem{Theorem}{Theorem}[section]
 \newtheorem{Lemma}[Theorem]{Lemma}
 \newtheorem{Proposition}[Theorem]{Proposition}
 
 \newtheorem{Corollary}[Theorem]{Corollary}
 
 \theoremstyle{definition} 
 \newtheorem{Definition}[Theorem]{Definition}
 \newtheorem{Remark}[Theorem]{Remark}
 \newtheorem{Example}[Theorem]{Example}

\begin{document}

\author{Jack Spielberg}
\title{Groupoids and \cstar-algebras for categories of paths}
\date{28 November 2011}
\address{School of Mathematical and Statistical Sciences \\ Arizona State University \\ P.O. Box 871804 \\ Tempe, AZ 85287-1804}
\email{jack.spielberg@asu.edu}
\subjclass[2010]{Primary 46L05; Secondary 20L05}
\keywords{Cuntz-Krieger algebra, Toeplitz Cuntz-Krieger algebra, groupoid, aperiodicity}

\dedicatory{Dedicated to the memory of Bill Arveson}

\begin{abstract}
In this paper we describe a new method of defining $C^*$-algebras from oriented combinatorial data, thereby generalizing the construction of algebras from directed graphs, higher-rank graphs, and ordered groups.  We show that only the most elementary notions of concatenation and cancellation of paths are required to define versions of Cuntz-Krieger and Toeplitz-Cuntz-Krieger algebras, and the presentation by generators and relations follows naturally.  We give sufficient conditions for the existence of an AF core, hence of the nuclearity of the $C^*$-algebras, and for aperiodicity, which is used to prove the standard uniqueness theorems.
\end{abstract}

\maketitle

\section{Introduction}
\label{s.intro}

In this paper we describe a new method of defining $C^*$-algebras from oriented combinatorial data, thereby generalizing the constructions of algebras from directed graphs, higher-rank graphs, and ordered groups.  The use of directed graphs to analyze $C^*$-algebras goes back to Bratteli's thesis introducing AF algebras (\cite{bra}).  The dual role played by graphs was apparent even then.  On the one hand, known $C^*$-algebras could be described by generators and relations based on suitable graphs.  On the other hand, any graph (in the class being considered) gives rise, by the same construction, to a $C^*$-algebra.  The resulting family of $C^*$-algebras can be studied as a whole via this class of graphs.  The flexibility inherent in the description of a graph yields a method of great power, especially if the process is \textit{functorial} (so that symmetries of a graph determine corresponding symmetries of the associated $C^*$-algebra).  The graphs are combinatorial objects, easy to ``turn around in one's hands.''  Yet their combinatorial properties control the behavior of complicated analytic objects:  the $C^*$-algebras.

The current use of directed graphs in $C^*$-theory originated in the work of Cuntz and Cuntz-Krieger (\cite{cun}, \cite{cunkri}).  Here the generators and relations were so clearly apparent that subsequent work took their imitation as primary focus (although it can be argued that the original motivation came from symbolic dynamics).  The treatment of arbitrary directed graphs was developed over a period of some 20 years (see, e.g., the historical remarks in \cite{rae}), with ad hoc devices to deal with each new aspect of a graph that was considered.  In a far-reaching generalization of the graph algebra construction, Kumjian and Pask introduced in \cite{kumpas} the notion of  \textit{higher-rank graphs}, based on work of Robertson and Steger on $C^*$-algebras associated to actions of groups on the boundary of an affine building.  Kumjian and Pask realized that for the purposes of the $C^*$-algebra, a graph may be replaced by its space of paths.  The key features of the space of paths are the concatenation of paths, and the unique factorization of a path as a concatenation of subpaths of prescribed lengths.  If ``length'' is taken to be an element in the positive cone  $(\IZ^+)^k \subseteq \IZ^k$ (instead of in $\IZ^+$), one obtains a \textit{$k$-graph}.  As in the case of 1-graphs, the work of finding the right generators and relations, and indeed, the right hypotheses, for higher-rank graphs passed through several phases.  The culmination of this difficult process was the notion of \textit{finite alignment}, introduced in \cite{raesimyee} (based on earlier work of Fowler).  In that paper the authors explicitly laid out their desire to find generators and relations that shared the most important characteristics of the traditional Cuntz-Krieger algebras.

In this paper we have quite the opposite motivation.  We wish to start with a suitable combinatorial object, and define a $C^*$-algebra directly from what might be termed the \textit{generalized symbolic dynamics} that it induces.  We give a functorial procedure that does not make any a priori assumptions about possible presentations of the algebra.  (This approach was used to give a natural derivation of the presentation of the $C^*$-algebra of a general directed graph in \cite{spi_graphalg}.  The current paper simplifies and generalizes that work.)  In this general context, finite alignment is a natural property, but is not needed to define the algebra.  The construction itself naturally gives rise to a presentation by generators and relations.  We then specialize to the finitely aligned case, where the considerations are much simpler.  In this case, the natural generators and relations turn out to be the same as those introduced in \cite{raesimyee} for higher-rank graphs.  Our treatment, however, applies in much more generality.  Following an idea proposed in \cite{exel}, we find that a degree functor is not at all necessary for defining the $C^*$-algebra.  Indeed, many different degrees can exist, giving different decompositions of the algebra (see section \ref{s.gaugeactions}).

We briefly describe the idea of our construction.  Cuntz and Krieger began with a shift matrix, i.e. a space of admissible sequences.  Equivalently, one can consider the directed graph having this as incidence matrix.  The finite admissible sequences correspond to finite directed paths in this graph.  The symbolic dynamics can be represented by the twin notions of \textit{concatenation} (right shift) and \textit{cancellation} (left shift):  the right shift by a path $\alpha$ is the map $\beta \mapsto \alpha\beta$, when $\beta$ is a path for which the concatenation is defined, while the left shift by $\alpha$ is the inverse of the right shift.  If we let $\alpha E^*$ denote the set of all paths that extend $\alpha$, then the collection of all such sets generates a Boolean ring of sets which is preserved by the shift maps.  The set of ultrafilters in this ring is a compact Hausdorff space on which the shifts then act as partial homeomorphisms (of course, this is just the space of one-sided infinite admissible sequences).  There are various ways of passing from this to a $C^*$-algebra; we choose the method of groupoids as being the most natural (\cite{arzren}, \cite{pat}).  This produces the usual Cuntz-Krieger algebra, as well as the Toeplitz version.  Our point of departure is the observation that the entire process requires only the operations of concatenation and cancellation; \textit{any} collection that behaves roughly like paths with regard to such operations will permit an analogous construction.  We find that all versions of generators and relations that have been used in previous work follow from elementary set theoretic considerations (see Theorem \ref{t.ringhomomorphisms.two}).  (Ultrafilters were used to define a boundary space, for directed graphs and their generalizations, in \cite{spi_graphalg}, \cite{exel3}, and \cite{brosimvit}.)

One of our primary motivations is to recover some of the flexibility inherent in the situation of directed graphs.  Any collection of ``dots and arrows'' is a directed graph, and so produces a graph algebra.  These algebras include some very important classes:  AF algebras, and Kirchberg algebras having free $K_1$ group, are the most notable.  Yet it remains a fairly limited class.  Already the 2-graphs exhibit much richer behavior (e.g. \cite{pasraerorsim}). Higher-rank graphs, on the other hand, are quite rigid objects, and difficult to construct.  (It is not clear how to build a $k$-graph from a collection of $k$ commuting graphs.  See \cite{kumpassim}, Example 5.15(ii).  The construction we give treats this example as easily as a finite $k$-graph.)  Although there are as yet no compelling examples of higher-rank graphs that are not finitely aligned, we believe that the difficulty of envisioning generators and relations for such has been a formidable obstacle.  Our methods provide a natural, albeit complicated, definition.  Moreover, in some instances the considerations are not so difficult, and our definition may lead to new and interesting examples.   Another of our goals is to give subcategories a natural place in the theory.  One of the successes of the approach taken in \cite{spi_graphalg} was that each subgraph of a given graph defines a subalgebra of its graph algebra.  For example, the algebra of a graph equals the direct limit of algebras associated to its finite subgraphs.  This result was obtained from the observation that the algebra of a subgraph depends on the ambient graph; a graph algebra is properly thought of as the algebra of a pair of nested graphs.  We adapt that idea here, although the considerations are much more difficult, and hence our conclusions more preliminary.  For example, even the relative category of paths defined by a subcategory of a finite category of paths need not be finitely aligned.  We plan to address some of these issues in a subsequent paper.

In section \ref{s.restriction} we consider the example of a countable ordered group; the positive cone is a category of paths.  With appropriate hypotheses, it is possible to define a \textit{directed boundary} of the group in this situation.  We prove that the corresponding crossed product $C^*$-algebra is Morita equivalent to the $C^*$-algebra of the category of paths.  One unexpected application is to the Wiener-Hopf algebras of Nica (\cite{nica}, \cite{lacrae}).  This work is concerned with $C^*$-algebras constructed from groups having a \textit{quasi-lattice ordering}.  In our context, this just means a particularly strict version of finite alignment.  In particular, our treatment relaxes this requirement, as well as clarifying the role of amenability for these algebras.  Our construction generalizes that of \cite{crilac}  in the quasi-lattice ordered case.  In particular, for quasi-lattice ordered groups, the \textit{boundary spectrum} of \cite{crilac}, Definition 3.4, is the same as our boundary.  In \cite{crilac} it is Nica's notion of amenability that is investigated, namely, the coincidence of full and reduced $C^*$-algebras.  We consider amenability in the sense of Renault (\cite{ren}).  In a separate paper (\cite{spi_baumsol}) we use our methods to describe the structure of the algebras associated to the Baumslag-Solitar groups.  These turn out to be identical to certain algebras studied by Katsura in his work on topological graphs (\cite{kat}).  Our method gives a new approach to the description of these algebras by generators and relations, and also gives the ideal structure of the Toeplitz versions.

Already in \cite{nica} it was observed that the algebras obtained from a quasi-lattice ordered group need not be amenable, in the sense that the spatial and abstract versions might not coincide.  From the point of view of amenability of the underlying groupoid, the problem is also of the nuclearity of the $C^*$-algebras.  In the case of graphs and higher-rank graphs, these properties are established by decomposing the algebra by means of a compact abelian group action, having an AF fixed-point algebra with a (partial) action of a free abelian group.  In general, there may be many such actions, with fixed-point algebras that are not AF.  In section \ref{s.gaugeactions} we give reasonable hypotheses guaranteeing the existence of an action having an AF fixed-point algebra.  This generalizes the argument for higher-rank graphs given in \cite{raesimyee}.  In section \ref{s.aperiodicity} we define \textit{aperiodicity} for categories of paths, and prove that it is equivalent to topological freeness of the corresponding groupoid.  We give a simple criterion for this property, generalizing recent results of Lewin and Sims \cite{lewsim}.  We also give analogs of the standard uniqueness theorems in the subject, as well as results on minimality and local contractivity.  In section \ref{s.amalgamation} we define the amalgamation of a collection of categories of paths.  This generalizes and simplifies the construction in \cite{kirchmodels} that has proved useful in several applications.

We briefly describe the contents of the first part of the paper.  In sections \ref{s.definition} and \ref{s.finitealignment} we define (relative) categories of paths, and introduce the basic facts about homomorphisms of Boolean rings of their subsets.  In section \ref{s.groupoid} we define the groupoid of a relative category of paths, and describe the simplifications that follow from finite alignment.  In section \ref{s.boolean} we prove that $*$-homomorphisms from the $C^*$-algebra of continuous functions vanishing at infinity on the unit space of the groupoid are characterized in a simple way by homomorphisms from the Boolean ring of subsets of the category.  This is particularly simple in the finitely aligned case, and this is the reason for the form of the usual presentation of the Toeplitz Cuntz-Krieger algebra.  We derive this presentation in section \ref{s.gensandrels}.  In section \ref{s.boundary} we restrict our attention to the finitely aligned case, and give a precise description of the elements of the unit space of the groupoid --- this might be termed the general \textit{infinite path space}, and consists of the directed hereditary subsets of the category.  It is natural to define the \textit{boundary} to be the closure of the maximal elements (see also \cite{nica}, \cite{lacrae}, and \cite{crilac}).  We characterize the elements of this closure using the finite exhaustive subsets of the category.  (We acknowledge our debt to \cite{raesimyee} for this.)  In the first part of section \ref{s.restriction} we use this to derive the usual presentation of the Cuntz-Krieger algebra.  This generalizes, and gives what we feel is a more natural motivation for, the presentation given in \cite{raesimyee}.

In the following we will always identify the objects of a category with the identity morphisms in that category, and we use juxtaposition to indicate composition of morphisms.  Morphisms will be referred to as \textit{paths}, and objects as \textit{vertices}.  We will use $s$ and $r$ to denote the \textit{source} and \textit{range} of morphisms in a category, and $\Lambda^0$ for the vertices in the category $\Lambda$.  It should be mentioned that the categorical framework puts us firmly in the \textit{Australian convention} for the case of graph algebras.  Thus when our method is used to construct a Cuntz-Krieger algebra, we obtain the algebra associated to the transpose of the matrix.

We thank the referee for giving many valuable suggestions.
We also thank Allan Donsig and David Milan for finding a gap in our earlier proof of Theorem \ref{t.gensandrels}.

\section{Definition of a category of paths}
\label{s.definition}

\begin{Definition}
\label{d.categoryofpaths}
A \textit{category of paths} is a small category satisfying
\begin{enumerate}
\item \label{d.categoryofpathsa}
$\alpha\beta = \alpha\gamma$ implies $\beta = \gamma$ \textit{(left-cancellation)}.
\item \label{d.categoryofpathsb} $\beta\alpha = \gamma\alpha$ implies $\beta = \gamma$ \textit{(right-cancellation)}.
\item \label{d.categoryofpathsc}
$\alpha\beta = s(\beta)$ implies $\alpha = \beta = s(\beta)$ \textit{(no inverses)}.
\end{enumerate}
\end{Definition}

\begin{Example}
\label{e.catpaths}
The following are examples of categories of paths:  
\begin{enumerate}
\item Higher-rank graphs.
\item Arbitrary subcategories of higher-rank graphs.
\item The categories of (finite) paths of the hybrid objects constructed in \cite{kirchmodels}, Definition 2.1.  The obvious generalizations of this construction are straightforward to define in the context of categories of paths (see section \ref{s.amalgamation}).
\item The category of paths in examples like \cite{kumpassim}, 5.15(ii).
\item The positive cone in a discrete ordered group (and in particular, the quasi-lattice ordered groups studied in \cite{nica}).
\item A small category equipped with a degree functor, satisfying unique factorization, taking values in the positive cone of an ordered group (\cite{kumpas}, Remarks 1.2).  The $P$-graphs of \cite{brosimvit} are examples of this kind.  (In the case that the group is totally ordered abelian, one obtains the natural definition of  \textit{directed $\Lambda$-graphs}, generalizing (directed) $\Lambda$-trees (\cite{chiswell}, \cite{ephrem}).)  See Remark \ref{r.uniquefactorization}.
\label{e.catpathsuniquefactorization}
\end{enumerate}
\end{Example}

\begin{Definition}
\label{d.leftshifts}
Let $\Lambda$ be a category of paths.
For any $\alpha \in \Lambda$ we define the \textit{left shift} $\sigma^\alpha : \alpha\Lambda \to s(\alpha)\Lambda$ by $\sigma^\alpha(\alpha\beta) = \beta$  ($\sigma^\alpha$ is well-defined by left-cancellation).  Of course for a subset $E$ of $\Lambda$, $\sigma^\alpha(E) = \sigma^\alpha(E \cap \alpha\Lambda)$.  (In the case where $\Lambda$ is a higher-rank graph, the difference between the functions $\sigma^\alpha$ for $\alpha \in \Lambda$, and $\sigma^n$ for $n \in \IN^k$, should be clear from the context.  However, we will not have occasion to use the latter notation.)
\end{Definition}

\begin{Remark}\label{r.rightcancellation}
The \textit{right shift maps} $\beta \in s(\alpha)\Lambda \mapsto \alpha\beta \in r(\alpha)\Lambda$ are one-to-one, by left-cancellation.  It is often useful to think of a category of paths as a category of injective maps.  Since the maps are generally not surjective, the hypothesis of right-cancellation is somewhat artificial.  We will see that much of the theory can be developed satisfactorily without it.  (But see Remark \ref{r.rightgroupoid}, section \ref{s.aperiodicity}, and the proof of Theorem \ref{t.amalgamationdegree}.)
\end{Remark}

\begin{Definition}
\label{d.extensions}
Let $\Lambda$ be a category of paths.  Let $\alpha$, $\beta \in \Lambda$.  We say that $\beta$ \textit{extends} $\alpha$ if there exists $\alpha' \in \Lambda$ such that $\beta = \alpha \alpha'$ (we may express this by writing $\beta \in \alpha\Lambda$).  If $\beta$ is an extension of $\alpha$, we call $\alpha$ an \textit{initial segment} of $\beta$. The set of initial segments of $\beta$ is denoted $[\beta]$.  It follows easily from Definition \ref{d.categoryofpaths}\eqref{d.categoryofpathsa} and \eqref{d.categoryofpathsc} that the relation $\alpha \in [\beta]$ is a partial order on $\Lambda$.   We follow Exel (\cite{exel}) in using the notations $\alpha \Cap \beta$ ($\alpha$ \textit{meets} $\beta$) if $\alpha\Lambda \cap \beta\Lambda \not= \emptyset$, and $\alpha \perp \beta$ ($\alpha$ is \textit{disjoint from} $\beta$) otherwise.  We let $\alpha \vee \beta$ denote the set of \textit{minimal common extensions} of $\alpha$ and $\beta$, i.e. the minimal elements of $\alpha\Lambda \cap \beta\Lambda$.  For a subset $F \subseteq \Lambda$ we let $\bigvee F$ denote the set of minimal common extensions of the elements of $F$.  (The context should suffice to distinguish this use of the symbol $\vee$ from its use to indicate the join of a family of projections in a \cstar-algebra (see section \ref{s.gensandrels}).)
\end{Definition}

Of course, not every common extension of $\alpha$ and $\beta$ need extend a minimal common extension; in fact,  $\alpha \vee \beta$ may be empty even if $\alpha \Cap \beta$.  If $\Lambda$ is a higher-rank graph, then every common extension does extend a minimal common extension.  Moreover, in the case of a higher-rank graph, distinct elements of $\alpha \vee \beta$ are disjoint; generally, this will not be the case in a category of paths.

Let $\Lambda$ be a category of paths.  The central object of study is the ring of subsets of $\Lambda$ generated by the tail sets $\{\alpha\Lambda : \alpha \in \Lambda\}$, and the action on it of the left shift maps $\{\sigma^\alpha : \alpha \in \Lambda\}$.  Much of the analysis rests on the following lemma from elementary set theory.  We use $\sqcup$ to indicate a union in which the sets presented are pairwise disjoint.

\begin{Lemma}
\label{l.setlemma}
Let $\CD^{(0)}$ be a collection of sets with the property that the intersection of two sets from $\CD^{(0)}$ is a finite union of sets from $\CD^{(0)}$.  Let $\CD$ be the collection of all non-empty sets of the form $A\setminus(B_1 \cup \cdots \cup B_n)$, where $A$, $B_1$, $\ldots$, $B_n\in\CD^{(0)}$, and $B_i\subseteq A$ for all $i$.  Let $\CA$ be the collection of all finite disjoint unions of sets from $\CD$.  Then $\CA$ is the ring of sets generated by $\CD^{(0)}$.
\end{Lemma}

\begin{proof}
The proof is routine.
\end{proof}

\begin{Remark}
The stipulation in the definition of $\CD$ that $B_i \subseteq A$ is unnecessary, since $B_i$ can be replaced with $B_i \cap A$.  However, it will be useful later to have this convention.
\end{Remark}

In the next section we will introduce the notion of \textit{finite alignment} for categories of paths.  This is a strong assumption that leads to very significant simplifications in the structure of the $C^*$-algebras.  For example, it implies that the ``usual generators and relations'' can be written in a particularly simple form.  However, much of the theory can be developed without this assumption, and one of our goals is to initiate the study of the non-finitely aligned theory.  In order to motivate the following definitions, we first present some examples.

\begin{Example}
\label{ex_nonfinitelyaligned}
The 2-graph in Figure~\ref{fig1cat} is the simplest non-finitely aligned example.  We have the identifications $\alpha \gamma_i = \beta \delta_i$ for $i \in \IN$.

\begin{figure}[ht]

\[
\begin{tikzpicture}[scale=2]

\node (0_0) at (0,0) [circle] {};
\node (2_0) at (2,0) [rectangle] {$v_0$};
\node (3_0) at (3,0) [rectangle] {$v_1$};
\node (4_0) at (4,0) [rectangle] {$v_2$};
\node (1_1) at (1,1) [circle] {};
\node (1_m1) at (1,-1) [circle] {};

\foreach \x in {0} \filldraw [black] (\x,0) circle (1pt);
\foreach \y in {1,-1} \filldraw [black] (1,\y) circle (1pt);

\draw[-latex,thick] (1_1) -- (0_0) node[pos=0.5, inner sep=0.5pt, anchor=south east] {$\alpha$};
\draw[-latex,thick] (1_m1) -- (0_0) node[pos=0.5, inner sep=0.5pt, anchor=north east] {$\beta$};

\draw[-latex,thick] (2_0) -- (1_1) node[pos=0.4, inner sep=0.5pt, anchor=north east] {$\gamma_0$};
\draw[-latex,thick] (3_0) -- (1_1) node[pos=0.3, inner sep=0.5pt, anchor=north east] {$\gamma_1$};
\draw[-latex,thick] (4_0) -- (1_1) node[pos=0.3, inner sep=0.5pt, anchor=south west] {$\gamma_2$};

\draw[-latex,thick] (2_0) -- (1_m1) node[pos=0.4, inner sep=0.5pt, anchor=south east] {$\delta_0$};
\draw[-latex,thick] (3_0) -- (1_m1) node[pos=0.3, inner sep=0.5pt, anchor=south east] {$\delta_1$};
\draw[-latex,thick] (4_0) -- (1_m1) node[pos=0.3, inner sep=0.5pt, anchor=north west] {$\delta_2$};

\draw (4.5,0) node {\textbf{. . .}};

\end{tikzpicture}
\]

\caption{} \label{fig1cat}

\end{figure}

Let us consider the \textit{tail sets} at the vertex $s(\alpha)$ (we anticipate notation introduced later --- see Remark \ref{r.finitelyalignedring}).  For $v \in \Lambda^0$ let $\CE_v = \{ \mu \Lambda : r(\mu) = v \}$, and let $\CA_v$ be the ring of sets generated by $\CE_v$.  Note that $s(\alpha)\Lambda = \{ \gamma_i : i \in \IN \} \cup \{s(\alpha)\}$, and $\gamma_i\Lambda = \{\gamma_i\}$; these are the elements of $\CE_{s(\alpha)}$.  In this example, the rings $\CA_v$ are not preserved by the shift maps.  For example, $\beta\Lambda = \{ \beta\delta_i : i \in \IN \} \cup \{ \beta \}$, and $\sigma^\alpha \beta \Lambda = \{ \gamma_i : i \in \IN \} \not\in \CA_{s(\alpha)}$.

Figure~\ref{fig2cat} illustrates the kind of further complication that can arise.  Again this is a non-finitely aligned 2-graph.  In this example, the central diamond represents the same 2-graph as in Figure~\ref{fig1cat}.  The edges $\eps_i$ are present only for $i$ even, while the edges $\theta_i$ are present only when $i$ is divisible by 3.  We have additional identifications $\mu \gamma_{2i} = \nu \eps_{2i}$ and $\xi \delta_{3i} = \eta \theta_{3i}$.

\begin{figure}[ht]

\[
\begin{tikzpicture}[scale=2]

\node (0_0) at (0,0) [circle] {};
\node (2_0) at (2,0) [circle] {};
\node (m2_0) at (-2,0) [circle] {};
\node (m3_1) at (-3,1) [circle] {};
\node (m1_1) at (-1,1) [circle] {};
\node (1_1) at (1,1) [circle] {};
\node (3_1) at (3,1) [circle] {};
\node (0_2) at (0,2) [circle] {$v_i$};

\foreach \x in {0,-2,2} \filldraw [black] (\x,0) circle (1pt);
\foreach \x in {-3,-1,1,3} \filldraw [black] (\x,1) circle (1pt);

\draw[-latex,thick] (m1_1) -- (0_0) node[pos=0.5, inner sep=0.5pt, anchor=north east] {$\alpha$};
\draw[-latex,thick] (1_1) -- (0_0) node[pos=0.5, inner sep=0.5pt, anchor=north west] {$\beta$};
\draw[-latex,thick] (m1_1) -- (m2_0) node[pos=0.5, inner sep=0.5pt, anchor=north west] {$\mu$};
\draw[-latex,thick] (m3_1) -- (m2_0) node[pos=0.5, inner sep=0.5pt, anchor=north east] {$\nu$};
\draw[-latex,thick] (1_1) -- (2_0) node[pos=0.5, inner sep=0.5pt, anchor=north east] {$\xi$};
\draw[-latex,thick] (3_1) -- (2_0) node[pos=0.5, inner sep=0.5pt, anchor=north west] {$\eta$};
\draw[-latex,thick] (0_2) -- (1_1) node[pos=0.5, inner sep=0.5pt, anchor=north east] {$\delta_i$};
\draw[-latex,thick] (0_2) -- (m1_1) node[pos=0.5, inner sep=0.5pt, anchor=north west] {$\gamma_i$};
\draw[-latex,thick] (0_2) -- (3_1) node[pos=0.5, inner sep=0.5pt, anchor=south west] {$\theta_i \text{ (if } 3|i)$};
\draw[-latex,thick] (0_2) -- (m3_1) node[pos=0.5, inner sep=0.5pt, anchor=south east] {$\eps_i \text{ (if } 2|i)$};

\end{tikzpicture}
\]

\caption{} \label{fig2cat}

\end{figure}

Applying shift maps to tail sets allows us to construct
\begin{align*}
\sigma^\alpha \beta \Lambda &= \{ \gamma_i : i \in \IN \} \\
\sigma^\mu \nu \Lambda &= \{ \gamma_i : i \in 2\IN \} \\
\sigma^\beta \alpha \sigma^\mu \nu \Lambda &= \{ \delta_i : i \in 2\IN \} \\
\sigma^\eta \xi \sigma^\beta \alpha \sigma^\mu \nu \Lambda &= \{ \theta_i : i \in 6\IN \} \\
\sigma^\alpha \beta \sigma^\eta \xi \sigma^\beta \alpha \sigma^\mu \nu \Lambda &= \{ \gamma_i : i \in 6\IN \}. 
\end{align*}

\end{Example}

We see from these examples that sets obtained from tail sets by applying a sequence of shift maps must be considered if we wish to build a ring of sets that is preserved by the shift maps (that is, in order to have a setting for the \textit{generalized symbolic dynamics} associated to a category of paths).  This is the motivation for the following definition.

\begin{Definition}
\label{d.zigzag}
Let $\Lambda$ be a category of paths.  A \textit{zigzag} is an even tuple of the form
\[
\zeta = (\alpha_1,\beta_1,\ldots,\alpha_n,\beta_n),
\]
where $\alpha_i$, $\beta_i \in \Lambda$, $r(\alpha_i) = r(\beta_i)$, $1 \le i \le n$, and $s(\alpha_{i+1}) = s(\beta_i)$, $1 \le i < n$.  We might draw a zigzag like this (whence the name):
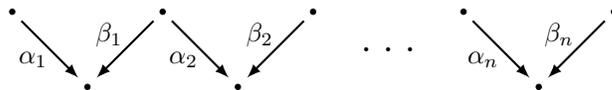
\[
\begin{tikzpicture}

\node (0_1) at (0,1) [circle] {};
\node (2_1) at (2,1) [circle] {};
\node (4_1) at (4,1) [circle] {};
\node (6_1) at (6,1) [circle] {};
\node (8_1) at (8,1) [circle] {};
\node (1_0) at (1,0) [circle] {};
\node (3_0) at (3,0) [circle] {};
\node (7_0) at (7,0) [circle] {};

\foreach \x in {0,2,4,6,8} \filldraw [black] (\x,1) circle (1pt);
\foreach \x in {1,3,7} \filldraw [black] (\x,0) circle (1pt);

\draw[-latex,thick] (0_1) -- (1_0) node[pos=0.5, inner sep=0.5pt, anchor=north east] {$\alpha_1$};
\draw[-latex,thick] (2_1) -- (1_0) node[pos=0.5, inner sep=0.5pt, anchor=south east] {$\beta_1$};
\draw[-latex,thick] (2_1) -- (3_0) node[pos=0.5, inner sep=0.5pt, anchor=north east] {$\alpha_2$};
\draw[-latex,thick] (4_1) -- (3_0) node[pos=0.5, inner sep=0.5pt, anchor=south east] {$\beta_2$};
\draw[-latex,thick] (6_1) -- (7_0) node[pos=0.5, inner sep=0.5pt, anchor=north east] {$\alpha_n$};
\draw[-latex,thick] (8_1) -- (7_0) node[pos=0.5, inner sep=0.5pt, anchor=south east] {$\beta_n$};

\node (dots) at (5,.5) [circle] {\textbf{. . .}};

\end{tikzpicture}
\]
We let $\CZ_{\Lambda}$ denote the set of all zigzags.  We may omit the subscript $\Lambda$ when it is clear from the context.  We define the maps $s$ and $r$ on $\CZ$ by $s(\zeta) = s(\beta_n)$ and $r(\zeta) = s(\alpha_1)$, and the \textit{reverse} of $\zeta$ as 
\[
\overline{\zeta} =
(\beta_n,\alpha_n,\ldots,\beta_1,\alpha_1).
\]
Each zigzag $\zeta \in \CZ_{\Lambda}$ defines a \textit{zigzag map} on $\Lambda$ that we denote $\pphi_\zeta \equiv \pphi^\Lambda_\zeta$, given by
\[
\pphi_\zeta =
\sigma^{\alpha_1} \beta_1 \cdots \sigma^{\alpha_n} \beta_n.
\]
We let $A(\zeta) \equiv A_\Lambda(\zeta)$ denote the domain of $\pphi_\zeta$.  Thus $A(\zeta) = \pphi_{\overline{\zeta}}(\Lambda)\subseteq s(\zeta)\Lambda$, and the range of $\pphi_\zeta$ equals $A(\overline{\zeta})$.  We call $A(\zeta)$ a \textit{zigzag set}.
\end{Definition}

\begin{Remark}
\label{r.zigzag}
\begin{enumerate}
\item $\Lambda$ can be identified with a subset of $\CZ$ by the pairs $(r(\alpha),\alpha)$.  Then 
$\pphi_{(r(\alpha),\alpha)}$ is the right shift defined by $\alpha$ (Remark \ref{r.rightcancellation}), and $\alpha\Lambda = A\bigl((\alpha,r(\alpha))\bigr)$ is a zigzag set.
\label{r.zigzag.b}
\item $\CZ$ is clearly closed under concatenation, and $\pphi_{\zeta_1\zeta_2} = \pphi_{\zeta_1} \circ \pphi_{\zeta_2}$.  For $\zeta$, $\xi \in \CZ$ we denote by $\xi\CZ$, $\CZ\zeta$, and $\xi\CZ\zeta$ the obvious subsets of $\CZ$.
\label{r.zigzag.c}
\item Since left- and right-shifts are one-to-one, all zigzag maps are one-to-one.  The inverse of $\pphi_\zeta$ is $\pphi_{\overline{\zeta}}$.
\label{r.zigzag.d}
\item $A(\zeta) = A(\overline{\zeta}\zeta)$ for $\zeta \in \CZ$.  If $\zeta_1$, $\zeta_2 \in \CZ v$, then $A(\zeta_1) \cap A(\zeta_2) = A(\overline{\zeta_1}\zeta_1\overline{\zeta_2}\zeta_2)$.  (Thus the collection of zigzag sets is closed under intersection.)
\label{r.zigzag.e}
\item For $\zeta \in \CZ$ and $\alpha \in A(\zeta)$, we have $\alpha\Lambda \subseteq A(\zeta)$.  For $\beta \in s(\alpha) \Lambda$, $\pphi_\zeta(\alpha \beta) = \pphi_\zeta(\alpha) \beta$.
\label{r.zigzag.f}
\end{enumerate}
\end{Remark}

Much of what we wish to do remains valid in the following more general context.

\begin{Definition}
\label{d.relativepair}
Let $\Lambda$ be a category of paths, and let $\Lambda_0$ be a subcategory.  We call the pair $(\Lambda_0,\Lambda)$ a \textit{relative category of paths}.
\end{Definition}

\begin{Definition}
\label{d.ringofsets}
Let $(\Lambda_0,\Lambda)$ be a relative category of paths.  Write 
\[
\CZ_{(\Lambda_0,\Lambda)} = \{ \zeta \in \CZ_\Lambda : \zeta = (\alpha_1, \beta_1, \ldots, \alpha_n, \beta_n), \text{ where } \alpha_i, \beta_i \in \Lambda_0 \}.
\]
For $v\in\Lambda_0^0$, let $\CD^{(0)}_v$, (or $ \CD(\Lambda_0,\Lambda)^{(0)}_v$), denote the collection of all nonempty sets of the form $A_\Lambda(\zeta)$ for $\zeta \in \CZ_{(\Lambda_0,\Lambda)} v$.  (Note that while the set $A_\Lambda(\zeta)$ (and the map $\pphi_\zeta$) does not depend on the subcategory $\Lambda_0$, the collection $\CD(\Lambda_0,\Lambda)_v^{(0)}$ does depend on $\Lambda_0$.)  Let $\CA_v$, (or $\CA(\Lambda_0,\Lambda)_v$), denote the ring of sets generated by $\CD^{(0)}_v$.  We let $\CD_v$ denote the collection of nonempty sets of the form $E \setminus \bigcup_{i=1}^n F_i$, where $E$, $F_1$, $\ldots$, $F_n \in \CD^{(0)}_v$, and $F_i \subseteq E$.  It follows from Lemma \ref{l.setlemma} and Remark \ref{r.zigzag}\eqref{r.zigzag.e} that $\CA_v$ equals the collection of finite disjoint unions of sets from $\CD_v$.
\end{Definition}

\begin{Remark}
\label{r.subring}
The subcategory of a relative category of paths serves to select a certain subcollection of the zigzag sets in $\Lambda$.  When $\Lambda_0 = \Lambda$, no such selection is made, and we will omit reference to the subcategory.  In general, we have $\CA(\Lambda_0,\Lambda)_v \subseteq \CA(\Lambda)_v$.
\end{Remark}

\begin{Lemma}
\label{l.ringofsets}
Let $(\Lambda_0,\Lambda)$ be a relative category of paths.  If $\zeta \in \CZ$ then $\pphi_\zeta (\CA_{s(\zeta)}) \subseteq \CA_{r(\zeta)}$.
\end{Lemma}

\begin{proof}
Let $\xi \in \CZ s(\zeta)$.  Then
\[
A(\xi) = \pphi_{\overline{\xi}}(\Lambda),
\]
hence
\[
\pphi_\zeta (A(\xi))
= \pphi_{\zeta \overline{\xi}} (\Lambda)
= \pphi_{\overline{\xi \overline{\zeta}}} (\Lambda)
= A(\xi \overline{\zeta})
\in \CA_{r(\zeta)}.
\]
\end{proof}

\begin{Proposition}
\label{p.booleanring}
Let $(\Lambda_0,\Lambda)$ be a relative category of paths.  For each $v \in \Lambda_0^0$, let $\CB_v$ be a Boolean ring of subsets of $v\Lambda$ such that 
\begin{enumerate}
\item $v\Lambda \in \CB_v$ for all $v \in \Lambda^0$.
\label{p.booleanring.a}
\item $\alpha\CB_{s(\alpha)} \subseteq \CB_{r(\alpha)}$ for all $\alpha \in \Lambda_0$.
\label{p.booleanring.b}
\item $\sigma^\alpha \CB_{r(\alpha)} \subseteq \CB_{s(\alpha)}$ for all $\alpha \in \Lambda_0$.
\label{p.booleanring.c}
\end{enumerate}
Then $\CA_v \subseteq \CB_v$ for all $v \in \Lambda_0^0$.  Moreover, $\alpha\CB_{s(\alpha)} = \{ E \cap \alpha\Lambda : E \in \CB_{r(\alpha)} \}$ and $\sigma^\alpha \CB_{r(\alpha)} = \CB_{s(\alpha)}$ for all $\alpha \in \Lambda_0$.
\end{Proposition}

\begin{proof}
Since $\alpha\Lambda = \alpha\bigl(s(\alpha)\Lambda\bigr) \in \alpha \CB_{s(\alpha)} \subseteq \CB_{r(\alpha)}$, we have $\CD^{(0)}_v \subseteq \CB_v$ for all $v \in \Lambda_0^0$.  It now follows from \eqref{p.booleanring.b} and \eqref{p.booleanring.c} that $A(\zeta) \in \CB_{s(\zeta)}$ for all $\zeta \in \CZ$, and hence $\CA_v \subseteq \CB_v$ for all $v \in \Lambda_0^0$.  Next, we have for $E \in \CB_{r(\alpha)}$, $E \cap \alpha\Lambda = \alpha \sigma^\alpha (E) \in \alpha \CB_{s(\alpha)}$.  Hence $\{ E \cap \alpha\Lambda : E \in \CB_{r(\alpha)} \} \subseteq \alpha \CB_{s(\alpha)}$.  The reverse containment is clear.  Finally,
\[
\CB_{s(\alpha)}
= \sigma^\alpha \alpha \CB_{s(\alpha)}
= \sigma^\alpha \bigl( \{ E \cap \alpha\Lambda : E \in \CB_{r(\alpha)} \} \bigr)
= \sigma^\alpha \CB_{r(\alpha)}.
\]
\end{proof}

\begin{Corollary}
\label{c.booleanring}
$\{\CA_v : v \in \Lambda_0^0\}$ is the family of Boolean rings generated by the collection $\{v \Lambda : v \in \Lambda^0_0\}$ and the family of left and right shift maps by elements of $\Lambda_0$.  Moreover, $\alpha\CA_{s(\alpha)} = \{ E \cap \alpha\Lambda : E \in \CA_{r(\alpha)} \}$ and $\sigma^\alpha \CA_{r(\alpha)} = \CA_{s(\alpha)}$ for all $\alpha \in \Lambda_0$.
\end{Corollary}

\section{Finite alignment}
\label{s.finitealignment}

The finitely aligned case is particularly important.  For example, the generators of the \cstar-algebra satisfy a Wick ordering principle.  Moreover, the treatment is significantly simpler in several ways.  Therefore we will take special pains to work out the details of the finitely aligned case, taking it further than the general case.

\begin{Definition}
\label{d.finitelyalignedpair}
The relative category of paths $(\Lambda_0,\Lambda)$ is \textit{finitely aligned} if
\begin{enumerate}
\item For every pair of elements $\alpha$, $\beta \in \Lambda_0$, there is a finite subset $G$ of $\Lambda_0$ such that $\alpha\Lambda \cap \beta\Lambda = \bigcup_{\eps \in G} \eps\Lambda$,
\label{d.finitelyalignedpair.a}
\item For every $\alpha \in \Lambda_0$, $\alpha\Lambda \cap \Lambda_0 = \alpha\Lambda_0$.
\label{d.finitelyalignedpair.b}
\end{enumerate}
If $\Lambda_0 = \Lambda$ then \eqref{d.finitelyalignedpair.b} is vacuous.  In this case we say that $\Lambda$ is finitely aligned.
\end{Definition}

\begin{Lemma}
\label{l.finitelyalignedpair}
Let $(\Lambda_0,\Lambda)$ be a finitely aligned relative category of paths, and let $F$ be a finite subset of $\Lambda_0$.  Let $\bigvee F$ denote the set of minimal common extensions of $F$ in $\Lambda$.  Then $\bigvee F$ is finite, $\bigvee F \subseteq \Lambda_0$, and $\bigcap_{\alpha \in F} \alpha \Lambda = \bigcup_{\beta \in \bigvee F} \beta \Lambda$.
\end{Lemma}

\begin{proof}
By Definition \ref{d.finitelyalignedpair}\eqref{d.finitelyalignedpair.a}, and induction, there is a finite set $G \subseteq \Lambda_0$ such that $(*) \ \bigcap_{\alpha \in F} \alpha \Lambda = \bigcup_{\beta \in G} \beta \Lambda$.  Replacing $G$ by a subset, if necessary, we may assume that $\beta \not\in \beta'\Lambda$ when $\beta$, $\beta'$ are distinct elements of $G$.  We will show that $G = \bigvee F$ to finish the proof.  By $(*)$, $G \subseteq \bigcap_{\alpha \in F} \alpha \Lambda$.  Suppose $\beta \in G$ and $\gamma \in \bigcap_{\alpha \in F} \alpha \Lambda$ are such that $\beta \in \gamma \Lambda$.  By $(*)$ there is $\beta' \in G$ such that $\gamma \in \beta' \Lambda$.  But then $\beta \in \beta' \Lambda$, so we have $\beta = \beta' = \gamma$.  Hence $G \subseteq \bigvee F$.  Conversely, if $\delta \in \bigvee F$ then $\delta \in \bigcap_{\alpha \in F} \alpha \Lambda$.  By $(*)$ there is $\beta \in G$ such that $\delta \in \beta \Lambda$.  Since $\delta$ is a minimal element of $\bigcap_{\alpha \in F} \alpha \Lambda$, we have $\delta = \beta \in G$.  Thus $\bigvee F \subseteq G$.
\end{proof}

In the next result we will use the following notation.  Let $\{ f_i : i \in I \}$ be a (finite) collection of functions.  Then $\bigcup_{i \in I} f_i$ is a function if and only if for all $i$, $j \in I$, $f_i$ and $f_j$ agree on the intersection of their domains.  We will write $\bigcup_{i \in I} f_i$ if and only if the $\{ f_i \}$ satisfy this condition.

\begin{Lemma}
\label{l.finitelyalignedzigzag}
If $(\Lambda_0,\Lambda)$ is finitely aligned, then every zigzag map is a finite union of maps of the form $\gamma \sigma^{\delta}$ with $\gamma$, $\delta \in \Lambda_0$.
\end{Lemma}

\begin{proof}
If $\alpha$, $\beta \in \Lambda_0$ with $r(\alpha) = r(\beta)$, then the domain of $\sigma^\beta \alpha$ is 
\[
\sigma^\alpha(\beta\Lambda) 
= \sigma^\alpha(\alpha\Lambda \cap \beta\Lambda)
= \sigma^\alpha \left( \bigcup_{\eps \in \alpha \vee \beta} \eps\Lambda \right)
= \bigcup_{\eps \in \alpha \vee \beta} (\sigma^\alpha \eps)\Lambda.
\]
Thus we see that
\[
\sigma^\beta \alpha
= \bigcup_{\eps \in \alpha \vee \beta} (\sigma^\beta \eps)(\sigma^{\sigma^\alpha \eps}).
\]
It follows from Definition \ref{d.finitelyalignedpair}\eqref{d.finitelyalignedpair.b} that $\sigma^\alpha \eps$, $\sigma^\beta \eps \in \Lambda^0$.  Repeated application of this reduces a zigzag map to the required form.
\end{proof}

\begin{Corollary}
\label{c.finitelyalignedzigzag}
Let $(\Lambda_0,\Lambda)$ be a finitely aligned relative category of paths, and let $\zeta \in \CZ$.  Then $A(\zeta)$ is a finite union of sets of the form $\alpha\Lambda$ with $\alpha \in \Lambda_0$.
\end{Corollary}

\begin{proof}
For $\gamma$, $\delta \in \Lambda_0$, $A(\gamma \sigma^\delta) = \delta\Lambda$.
\end{proof}

\begin{Remark}
\label{r.finitelyalignedring}
Let $(\Lambda_0,\Lambda)$ be a finitely aligned relative category of paths.  For $v \in \Lambda_0^0$ we let $\CE^{(0)}_v$ denote the collection of all sets of the form $\alpha\Lambda$ for $\alpha \in v\Lambda_0$, and $\CE_v$ the collection of nonempty sets of the form $\alpha\Lambda \setminus \bigcup_{i=1}^n \beta_i\Lambda$ for $\alpha \in v\Lambda_0$ and $\beta_i \in \alpha\Lambda_0$.  Thus,  $\CA_v$ is the collection of finite disjoint unions of sets in $\CE_v$, by Lemma \ref{l.setlemma} and Corollary \ref{c.finitelyalignedzigzag}.
\end{Remark}

\begin{Proposition}
\label{p.relativefinitelyaligned}
Let $(\Lambda_0,\Lambda)$ be a finitely aligned relative category of paths.  The map $E \mapsto E \cap \Lambda_0$ from $\CP(\Lambda) \to \CP(\Lambda_0)$ restricts to give a ring isomorphism $\CA(\Lambda_0,\Lambda)_v \to \CA(\Lambda_0)_v$, for each $v \in \Lambda_0^0$.  These isomorphisms are equivariant for the shift maps defined by elements of $\Lambda_0$.
\end{Proposition}

\begin{proof}
It follows from the fact that intersection distributes over union and difference that the given map is a ring homomorphism $\CP(\Lambda) \to \CP(\Lambda_0)$.  By Remark \ref{r.finitelyalignedring}, a typical (nonempty) set in $\CA(\Lambda_0,\Lambda)_v$ has the form $E = \bigsqcup_i (\alpha_i\Lambda \setminus \bigcup_j \beta_{ij}\Lambda)$, where $\alpha_i$, $\beta_{ij} \in \Lambda_0$, $\beta_{ij} \in \alpha_i\Lambda_0$, $\beta_{ij} \not= \alpha_i$, and the unions are finite.  Then $E \cap \Lambda_0 = \bigsqcup_i (\alpha_i\Lambda_0 \setminus \bigcup_j \beta_{ij}\Lambda_0)$ by Definition \ref{d.finitelyalignedpair}\eqref{d.finitelyalignedpair.b}.  Therefore the restriction of the given map to $\CA(\Lambda_0,\Lambda)_v$ does have range in $\CA(\Lambda_0)_v$.  It is clear that if $E \not= \emptyset$ then $E \cap \Lambda_0 \not= \emptyset$, since $\alpha_i \in E \cap \Lambda_0$ for all $i$, and hence the restriction is one-to-one.  Since every set in $\CA(\Lambda_0)_v$ is of the above form, the restriction is onto.
\end{proof}

\begin{Remark}
\label{r.relativefinitelyaligned}
Because of Proposition \ref{p.relativefinitelyaligned}, we need not consider relative categories of paths in the finitely aligned case.
\end{Remark}

\begin{Remark}
We note that if $(\Lambda_0,\Lambda)$ is finitely aligned then the family of Boolean rings generated by the $\CE^{(0)}_v$ is already invariant under the shift maps.  If $(\Lambda_0,\Lambda)$ is not finitely aligned, we have seen that this need not be the case (Example \ref{ex_nonfinitelyaligned}).  Even if $\Lambda_0$ and $\Lambda$ are individually finitely aligned, the relative category of paths $(\Lambda_0,\Lambda)$ need not be.  For example, let $\Lambda$ be the subcategory of Figure~\ref{fig1cat} generated by $\alpha$, $\beta$, $\gamma_0$, $\gamma_1$, $\delta_0$, $\delta_1$, and let $\Lambda_0$ be the subcategory generated by $\alpha$, $\beta$, $\gamma_0$, $\delta_0$.  Then $\alpha\Lambda \cap \beta\Lambda = \{\alpha\gamma_0,\alpha\gamma_1\}$, so that Definition \ref{d.finitelyalignedpair}\eqref{d.finitelyalignedpair.a} fails.  It is easy to check that the ring generated by tail sets is not invariant under left shifts by elements of $\Lambda_0$.

In the above two examples, it was condition \eqref{d.finitelyalignedpair.a} of Definition \ref{d.finitelyalignedpair} that failed.  We present an example where condition \eqref{d.finitelyalignedpair.b} fails.  For this, let $\Lambda$ be the subcategory of the first example generated by $\alpha$, $\beta$, $\gamma_0$, $\delta_0$, and let $\Lambda_0 = \{\alpha,\beta,\alpha\gamma_0\} \cup \Lambda^0$.  In this example, Definition \ref{d.finitelyalignedpair}\eqref{d.finitelyalignedpair.a} holds, but \eqref{d.finitelyalignedpair.b} fails.  In this case, $\CE_v$ is an elementary family, but is not invariant under left shifts:  $\sigma^\alpha(\beta\Lambda) = \{\gamma_0\}$ is not in the ring generated by $\CE_{s(\alpha)}$.
\end{Remark}

\section{The groupoid of a relative category of paths}
\label{s.groupoid}

Recall that a \textit{filter} in a ring of sets $\CA$ is a nonempty collection, $\CU$, of nonempty elements of $\CA$, which is closed under the formation of intersections and supersets.  An \textit{ultrafilter} is a maximal filter.  Ultrafilters are characterized by the property:  for each $E \in \CA$, either $E \in \CU$ or there is $F \in \CU$ with $E \cap F = \emptyset$.  A \textit{filter base} is a nonempty collection of nonempty subsets of $\CA$ such that the intersection of any two of its elements is a superset of a third element.  A filter base defines a unique filter by closing with respect to supersets.  An \textit{ultrafilter base} is a filter base such that the filter it defines is an ultrafilter.  Ultrafilter bases ($\CU_0$) are characterized by the property:  for each $E \in \CA$, either there is $F \in \CU_0$ with $E \supseteq F$, or there is $F \in \CU_0$ with $E \cap F = \emptyset$.  Each element $\alpha$ in the set underlying $\CA$ determines an ultrafilter $\CU_\alpha = \{ E \in \CA : \alpha \in E \}$, called a \textit{fixed ultrafilter}.

Let $(\Lambda_0,\Lambda)$ be a relative category of paths.
For $v \in \Lambda_0^0$ let $A_v = \overline{\text{span}} \{\chi_E : E \in \CA_v\} \subseteq \ell^\infty(v\Lambda)$, a commutative $C^*$-algebra.  As in \cite{spi_graphalg}, p. 250, we identify $\widehat{A_v}$ with the set of ultrafilters in $\CA_v$:  the complex homomorphism $\omega \in \widehat{A_v}$ corresponds to the ultrafilter $\{ E \in \CA_v : \omega(\chi_E) = 1 \}$.  We will write  $X_v$ for the spectrum, $\widehat{A_v}$, of $A_v$.  We let $\CU_x$ denote the ultrafilter corresponding to $x \in X_v$.  For $E \in \CA_v$ let $\widehat{E} \subseteq X_v$ denote the support of $\chi_E$.  Thus $x \in \widehat{E}$ if and only if $E \in \CU_x$.  The collection $\{\widehat{E} : E \in \CD_v\}$ is a base for the topology of $X_v$ consisting of compact-open sets.

In order to define the groupoid associated to a relative category of paths, we have to study the maps between the spaces $\{X_v\}$ that are induced by concatenation.

\begin{Lemma}
\label{l.concatenation}
Let $(\Lambda_0,\Lambda)$ be a relative category of paths, and for each $v \in \Lambda_0^0$ let $\CB_v$ be a ring of subsets of $v\Lambda$ such that conditions \eqref{p.booleanring.a}, \eqref{p.booleanring.b} and \eqref{p.booleanring.c} of Proposition \ref{p.booleanring} are satisfied.  If $\alpha \in \Lambda_0$ and $\CU$ is an ultrafilter in $\CB_{s(\alpha)}$, then $\alpha\CU$ is an ultrafilter base in $\CB_{r(\alpha)}$. 
\end{Lemma}

\begin{proof}
It is clear that $\alpha\CU$ is a filter base.  Let $E \in \CB_{r(\alpha)}$ be such that $E \not\supseteq \alpha F$ for all $F \in \CU$.  Then $\sigma^\alpha E \not\supseteq F$ for all $F \in \CU$.  Since $\CU$ is an ultrafilter, and since $\sigma^\alpha E \in \CB_{s(\alpha)}$, there is $F \in \CU$ with $\sigma^\alpha E \cap F = \emptyset$.  Then $E \cap \alpha F = \emptyset$.
\end{proof}

\begin{Theorem}
\label{t.concatenation}
Let $(\Lambda_0,\Lambda)$ be a relative category of paths.  Let $\alpha \in \Lambda_0$.  The map 
$\alpha : \CA_{s(\alpha)} \to \CA_{r(\alpha)}$ induces a continuous one-to-one map $\widehat{\alpha} : X_{s(\alpha)} \to X_{r(\alpha)}$.  For $y \in X_{r(\alpha)}$, we have $y \in \widehat{\alpha}(X_{s(\alpha)})$ if and only if $\alpha\Lambda \in \CU_y$.
\end{Theorem}

\begin{proof}
By Lemma \ref{l.concatenation}, $\alpha$ defines a map $\widehat{\alpha}$ as in the statement.
Let $x_1 \not= x_2$ in $X_{s(\alpha)}$.  Then there are $E_i \in \CU_{x_i}$ with $E_1 \cap E_2 = \emptyset$.  Hence $\alpha E_1 \cap \alpha E_2 = \emptyset$, so that $\alpha\CU_{x_1}$ and $\alpha\CU_{x_2}$ define distinct ultrafilters in $\CA_{r(\alpha)}$.  Therefore $\widehat{\alpha}$ is one-to-one.  For continuity, let $x_0 \in X_{s(\alpha)}$, and let $V$ be a neighborhood of $\widehat{\alpha}(x_0)$.  Then there is $B \in \CA_{r(\alpha)}$ such that $\widehat{\alpha}(x_0) \in \widehat{B} \subseteq V$.  Hence $B \supseteq \alpha E_0$ for some $E_0 \in \CU_{x_0}$.  Then $\sigma^\alpha(B) \supseteq E_0$, so $\sigma^\alpha(B) \in \CU_{x_0}$; i.e. $\widehat{\sigma^\alpha(B)}$ is a neighborhood of $x_0$.  We will show that $\widehat{\alpha}(\widehat{\sigma^\alpha(B)}) \subseteq \widehat{B}$.  Let $x \in \widehat{\sigma^\alpha(B)}$.  Then $\sigma^\alpha(B) \in \CU_x$, hence $\alpha\Lambda \cap B \in \alpha\CU_x$, hence $\widehat{\alpha}(x) \in \widehat{B}$.  Thus $\widehat{\alpha}$ is continuous.

Now let $y \in X_{r(\alpha)}$.  If $y = \widehat{\alpha}(x)$ for some $x \in X_{s(\alpha)}$, then $\CU_y \supseteq \alpha\CU_x$.  Since $\alpha\Lambda \in \alpha\CU_x$, we have $\alpha\Lambda \in \CU_y$.  Conversely, suppose that $\alpha\Lambda \in \CU_y$.  Then $\alpha\Lambda \cap \CU_y$ is a filter base for $\CU_y$, hence is an ultrafilter base.  Since $\alpha\Lambda \cap \CU_y = \bigl\{ \alpha\sigma^\alpha(E) : E \in \CU_y \bigr\}$, we have $\alpha\Lambda \cap \CU_y = \alpha\sigma^\alpha(\CU_y)$.  We claim that $\sigma^\alpha(\CU_y)$ is an ultrafilter in $\CA_{s(\alpha)}$.  To see that it is a filter (and not merely a filter base), let $F \in \CA_{s(\alpha)}$ be such that $F \supseteq \sigma^\alpha(E)$, where $E \in \CU_y$.  Then $\alpha F \supseteq \alpha\Lambda \cap E$, so $\alpha F \in \CU_y$.  Then $F \in \sigma^\alpha(\CU_y)$.  We now show that it is an ultrafilter.  Let $F \in \CA_{s(\alpha)}$ be such that $F \not\supseteq \sigma^\alpha(E)$ for all $E \in \CU_y$. Then $\alpha F \not\supseteq E$ for all $E \in \CU_y$.  Since $\CU_y$ is an ultrafilter, there exists $E \in \CU_y$ such that $\alpha F \cap E = \emptyset$.  Then $F \cap \sigma^\alpha(E) = \emptyset$.  Finally, let $x \in X_{s(\alpha)}$ with $\CU_x = \sigma^\alpha(\CU_y)$.  Then $\alpha\CU_x \subseteq \CU_y$.  Since $\alpha\CU_x$ is an ultrafilter base, we have $\widehat{\alpha}(x) = y$.
\end{proof}

\begin{Corollary}
\label{c.concatenation.two}
With hypotheses as in Theorem \ref{t.concatenation}, we have $\widehat{\alpha}(X_{s(\alpha)})$ is a compact-open subset of $X_{r(\alpha)}$.
\end{Corollary}

\begin{Definition}
\label{d.concatenation}
Let $(\Lambda_0,\Lambda)$ be a relative category of paths.  Let $\alpha \in \Lambda_0$.  In the sequel we will omit the caret, writing $x \in X_{s(\alpha)} \mapsto \alpha x \in X_{r(\alpha)}$.  We define $\sigma^\alpha : \alpha X_{r(\alpha)} \to X_{s(\alpha)}$ by the equation $\alpha \sigma^\alpha(y) = y$, for $y \in X_{r(\alpha)}$ with $\alpha\Lambda \in \CU_y$.  It now follows that for each $\zeta \in \CZ$ there is a homeomorphism $\Phi_\zeta : \widehat{A(\zeta)} \to \widehat{A(\overline{\zeta})}$, determined by the maps given above, and the formula $\Phi_{\zeta_1\zeta_2} = \Phi_{\zeta_1} \circ \Phi_{\zeta_2}$ holds.
\end{Definition}

We will now define the groupoid of a relative category of paths $(\Lambda_0,\Lambda)$ (we refer to \cite{pat}; see also \cite{arzren}).

\begin{Definition}
\label{d.groupoid.a}
We let $X = \bigsqcup_{v \in \Lambda_0^0} X_v$.  We define $r : X \to \Lambda_0^0$ by $r(x) = v$ if $x \in X_v$.
We define a relation $\sim$ on
\[
\CZ * X = \bigcup_{v \in \Lambda_0^0} \CZ v \times X_v
\]
by:  $(\zeta,x) \sim (\zeta',x')$ if $x = x'$ and $\Phi_{\zeta}|_U = \Phi_{\zeta'}|_U$ for some neighborhood $U$ of $x$.
\end{Definition}

\begin{Remark}
It is clear that $\sim$ is an equivalence relation.
\end{Remark}

\begin{Definition}
\label{d.groupoid.b}
The \textit{groupoid} of $(\Lambda_0,\Lambda)$ is the set
\[
G \equiv G(\Lambda_0,\Lambda) = (\CZ * X) / \sim.
\]
The set of composable pairs is
\[
G^2 = \bigl\{ \bigl([\zeta,x], \; [\zeta',x'] \bigr) : x = \Phi_{\zeta'}x' \bigr\},
\]
and inversion is given by $[\zeta,x]^{-1} = [\overline{\zeta},\zeta x]$ (where we use square brackets to denote equivalence classes).  Multiplication is given by
\[
[\zeta,\Phi_{\zeta'}x][\zeta',x] = [\zeta\zeta',x].
\]
\end{Definition}

It is clear that Definition \ref{d.groupoid.b} does not depend on the choices made of representatives of the equivalence classes.  It is elementary to check that the above does in fact define a groupoid (e.g. the conditions on page 7 of \cite{pat}).

\begin{Remark}
\label{r.groupoid.c}
It is immediate that if $[ r(x),x ] = [ r(y),y ]$ then $x = y$.  Thus the map $x \in X \mapsto [ r(x),x ] \in G^0$ is bijective.  We will often identify $X$ and $G^0$ via this map.
\end{Remark}

\begin{Definition}
\label{d.topology}
For $\zeta \in \CZ$ and $E \in \CA_{s(\zeta)}$, let
\[
[\zeta,E]
= \bigl\{ [\zeta,x] : x \in \widehat{E} \bigr\}.
\]
Let $\CB = \bigl\{ [\zeta,E] : E \in \CA_{s(\zeta)} \bigr\}$.
\end{Definition}

We note that since $[\zeta,E] = [\zeta, E \cap A(\zeta)]$, we may assume that $E \subseteq A(\zeta)$.

\begin{Proposition}
\label{p.topology}
$\CB$ is a base for a locally compact topology on $G$ for which the elements of $\CB$ are compact and Hausdorff, and making $G$ into an ample \'etale groupoid.
\end{Proposition}

\begin{proof}
Suppose $[\zeta, E] \cap [\zeta', E'] \not= \emptyset$.  Then $s(\zeta) = s(\zeta')$.  Let $[\zeta,x] \in [\zeta, E] \cap [\zeta', E']$.  Then $x \in \widehat{E} \cap \widehat{E'}$, and there is $F \in \CA_{s(\zeta)}$ such that $\Phi_\zeta|_{\widehat{F}} = \Phi_{\zeta'}|_{\widehat{F}}$ and $x \in \widehat{F}$.  Then 
\[
[\zeta,x] \in [\zeta, E \cap E' \cap F] \subseteq [\zeta, E] \cap [\zeta', E'].
\]
Therefore $\CB$ is a base for a topology on $G$.

Next, we show that multiplication and inversion are continuous.  Let $\bigl([\zeta,\Phi_{\zeta'} x], \; [\zeta', x]\bigr) \in G^2$.  Let $E = A(\zeta\zeta')$.  Then $E \in \CA_{r(x)}$ with $x \in \widehat{E}$.  Then $[\zeta\zeta',E]$ is a basic neighborhood of the product $[\zeta\zeta',x]$.  We have that $\pphi_{\zeta'}(E) \in \CA_{r(\zeta')}$, and $\Phi_{\zeta'} x \in \widehat{\pphi_{\zeta'}(E)}$.  Then $[\zeta, \pphi_{\zeta'}(E)] [\zeta', E] = [\zeta\zeta', E]$, hence multiplication is continuous.  Since $[\zeta, E]^{-1} = [\overline{\zeta},\pphi_\zeta(E)]$, inversion is continuous.

Finally, note that $r|_{[\zeta,E]} : [\zeta, x] \mapsto \Phi_\zeta x$ is an injective open map, and similarly for $s$.  Therefore $r$ and $s$ are local homeomorphisms.  It follows that $[\zeta, E]$ is compact (since $\widehat{E}$ is compact in $X$), and that the sets in $\CB$ are Hausdorff.  Therefore $G$ is ample and \'etale (cf. \cite{pat}, Definition 2.2.4).
\end{proof}

In general, $G$ is not a Hausdorff groupoid.  We next give an example of this.

\begin{Example}
\label{(e.nonhausdorff.a)}

\begin{figure}[ht]

\[
\begin{tikzpicture}

\node (w) at (0,0) [circle] {$w$};
\node (u) at (-2,0) [circle] {$u$};
\node (v) at (2,0) [circle] {$v$};
\node (x) at (0,2) [circle] {$x$};
\node (y) at (0,-2) [circle] {$y$};

\draw[-latex,thick] (u) -- (x) node[pos=0.5, inner sep=0.5pt, anchor=south east] {$\alpha_1$};
\draw[-latex,thick] (v) -- (x) node[pos=0.5, inner sep=0.5pt, anchor=south west] {$\beta_1$};
\draw[-latex,thick] (v) -- (y) node[pos=0.5, inner sep=0.5pt, anchor=north west] {$\beta_2$};
\draw[-latex,thick] (u) -- (y) node[pos=0.5, inner sep=0.5pt, anchor=north east] {$\alpha_2$};
\draw[-latex,thick] (w) .. controls (-.3,.3) and (-.7,.3) .. (u) node[pos=0.5, inner sep=0.5pt, anchor=south] {$\gamma_i$};
\draw[-latex,thick] (w) .. controls (-.3,-.3) and (-.7,-.3) .. (u) node[pos=0.5, inner sep=0.5pt, anchor=north] {$\mu_i$};
\draw[-latex,thick] (w) .. controls (.3,.3) and (.7,.3) .. (v) node[pos=0.5, inner sep=0.5pt, anchor=south] {$\delta_i$};
\draw[-latex,thick] (w) .. controls (.3,-.3) and (.7,-.3) .. (v) node[pos=0.5, inner sep=0.5pt, anchor=north] {$\nu_i$};

\end{tikzpicture}
\]

\caption{} \label{fig3cat}

\end{figure}

In Figure~\ref{fig3cat}, $\Lambda$ is a non-finitely aligned 2-graph (where the subscript $i$ takes values in $\IZ$),
with identifications $\alpha_1\gamma_i = \beta_1 \delta_i$, $\alpha_1 \mu_i = \beta_1 \nu_i$, $\alpha_2\gamma_i = \beta_2 \delta_i$, $\alpha_2 \mu_i = \beta_2 \nu_{i+1}$.  (To realize $\Lambda$ as a 2-graph, let edges pointing right have degree $(1,0)$, and let edges pointing left have degree $(0,1)$.)  We let $\zeta_j = (\beta_j,\alpha_j) \in \CZ u$.  Note that $\pphi_{\zeta_j}(\gamma_i) = \delta_i$, $\pphi_{\zeta_1}(\mu_i) = \nu_i$, and $\pphi_{\zeta_2}(\mu_i) = \nu_{i+1}$.  Further, note that $X_u = \{\CU_u,\,\CU_{\gamma_i},\, \CU_{\mu_i},\, \CU_{u,\infty}\}$, where we let $\CU_\alpha = \{ E \in \CA_{r(\alpha)} : \alpha \in E \}$ denote the fixed ultrafilter at $\alpha$, and $\CU_{u,\infty}$ the ultrafilter generated by the cofinite subsets of $\{\gamma_i,\mu_i : i \in \IZ\}$. Then the domain of $\Phi_{\zeta_j}$ is $X_u \setminus \{\CU_u\}$, and we have $\Phi_{\zeta_j}(\CU_{u,\infty}) = \CU_{v,\infty}$, $\Phi_{\zeta_j}(\CU_{\gamma_i}) = \CU_{\delta_i}$, $\Phi_{\zeta_1}(\CU_{\mu_i}) = \CU_{\nu_i}$, and $\Phi_{\zeta_2}(\CU_{\mu_i}) = \CU_{\nu_{i+1}}$.  We note that basic neighborhoods of $\CU_{u,\infty}$ may be taken to be cofinite subsets of $\{\CU_{\gamma_i}, \CU_{\mu_i} : i \in \IZ\}$.  Thus $\Phi_{\zeta_1} \not= \Phi_{\zeta_2}$ in any neighborhood of $\CU_{u,\infty}$:  $[\zeta_1,\CU_{u,\infty}] \not= [\zeta_2,\CU_{u,\infty}]$.  We claim that these two points do not have disjoint neighborhoods.  To see this, let $U$ be a basic neighborhood of $\CU_{u,\infty}$, say $U = \widehat{E}$, where $E =  A(\zeta_1) \setminus \{\gamma_i,\mu_i : |i| \le n\}$.  Then $[\zeta_1,E] \cap [\zeta_2,E] \supseteq \{[\zeta_1,\CU_{\gamma_i}] : |i| > n\}$.
\end{Example}

In the case where $\Lambda$ is finitely aligned, however, $G$ is Hausdorff.  Because of its importance, we will next work out some details of the finitely aligned case.

Let $\Lambda$ be finitely aligned.  We know from Lemma \ref{l.finitelyalignedzigzag} that every element of $G$ can be expressed in the form $[\alpha\sigma^\beta, y]$, where $\alpha$, $\beta \in \Lambda$.  In this case, $y \in \widehat{A(\sigma^\beta)} = \beta X_{s(\beta)}$.  Thus we may write $[\alpha\sigma^\beta, y] = [\alpha\sigma^\beta, \beta x]$ for some $x \in X_{s(\beta)}$.

\begin{Lemma}
\label{l.commontail}
Let $\Lambda$ be a finitely aligned category of paths.
Let $\alpha$, $\alpha' \in \Lambda$, and let $x \in X_{s(\alpha)}$, $x' \in X_{s(\alpha')}$ with $\alpha x = \alpha' x'$.  Then there exist $\delta \in s(\alpha)\Lambda$, $\delta' \in s(\alpha')\Lambda$, and $z \in X_{s(\delta)}$ such that $x = \delta z$, $x' = \delta' z$, and $\alpha\delta = \alpha'\delta'$.  (We refer to $z$ as a \textit{common tail} of $x$ and $x'$.)  Moreover we may assume that $\alpha\delta \in \alpha \vee \alpha'$.
\end{Lemma}

\begin{proof}
Since $\alpha\Lambda \in \CU_{\alpha x}$ and $\alpha'\Lambda \in \CU_{\alpha' x'}$, the equality $\alpha x = \alpha' x'$ implies that $\alpha \Cap \alpha'$.  Then there is $\eps \in \alpha \vee \alpha'$ such that $\eps\Lambda \in \CU_{\alpha x}$.  By Theorem \ref{t.concatenation} there is $z \in X_{s(\eps)}$ such that $\alpha x = \eps z$.  Write $\eps = \alpha \delta = \alpha' \delta'$.  Then $\alpha x = \alpha\delta z$, and hence $x = \sigma^\alpha(\alpha x) = \delta z$; similarly $x' = \delta' z$.
\end{proof}

\begin{Lemma}
\label{l.fa.equivalencerelation}
Let $\Lambda$ be a finitely aligned category of paths.  Let $\alpha$, $\beta$, $\alpha'$, $\beta' \in \Lambda$ and $x$, $x' \in X$ define elements of $G$ as in the remarks before Lemma \ref{l.commontail}.  Then $(\alpha\sigma^\beta, \beta x) \sim (\alpha'\sigma^{\beta'}, \beta' x')$ if and only if there are $\delta$, $\delta' \in \Lambda$ and $z \in X$ such that $x = \delta z$, $x' = \delta' z$, $\alpha\delta = \alpha'\delta'$, and $\beta\delta = \beta'\delta'$.
\end{Lemma}

\begin{proof}
Let $(\alpha\sigma^\beta, \beta x) \sim (\alpha'\sigma^{\beta'}, \beta' x')$.  Then $\beta x = \beta' x'$ and $\alpha \sigma^\beta = \alpha' \sigma^{\beta'}$ near $\beta x$.  By Lemma \ref{l.commontail}, there are $\gamma$, $\gamma'$, and $y$ such that $x = \gamma y$, $x' = \gamma' y$, and $\beta\gamma = \beta'\gamma'$.  Let $E \in \CA_{r(\beta)}$ be such that $\beta x \in \widehat{E}$ and $\alpha \sigma^\beta = \alpha' \sigma^{\beta'}$ in $\widehat{E}$.  We may assume that $E \in \CE_{r(\beta)}$.  Write $E = \eta\Lambda \setminus \bigcup_{i=1}^n \theta_i\Lambda$.  Since $\beta\gamma y = \beta x \in \widehat{E} \subseteq \eta X_{s(\eta)}$, we can write $\beta\gamma y = \eta u$ for some $u \in X_{s(\eta)}$.  We again apply Lemma \ref{l.commontail} to obtain $\eps$, $\eps'$, and $z$ such that $y = \eps z$, $u = \eps' z$, and $\beta\gamma\eps = \eta\eps'$.  Then $\eta\eps' z = \eta u \in  \widehat{E}$.  It follows that $\eta\eps' \in E$:  for if not, there would be an $i$ such that $\eta\eps' \in \theta_i\Lambda$, and hence that $\eta\eps' z \in \widehat{\theta_i\Lambda} \subseteq \widehat{E}^c$, a contradiction.  Let $\CU_w$ be the fixed ultrafilter in $\CA_{r(\beta)}$ at $\eta\eps'$.  Then $E \in \CU_w$, so $w \in \widehat{E}$.  Hence $\alpha\sigma^\beta w = \alpha'\sigma^{\beta'} w$, and hence 
\[
\alpha\gamma\eps
= \alpha\sigma^\beta \beta\gamma\eps
= \alpha\sigma^\beta \eta\eps' 
= \alpha'\sigma^{\beta'}\eta\eps'
= \alpha'\sigma^{\beta'}\beta'\gamma'\eps
= \alpha'\gamma'\eps.
\]
We still have $\beta\gamma\eps = \beta'\gamma'\eps$, and moreover $x = \gamma y = \gamma\eps z$ and $x' = \gamma' y = \gamma' \eps z$.  Thus we may take $\delta = \gamma\eps$ and $\delta' = \gamma'\eps$.

Conversely, suppose that $\delta$, $\delta'$ and $z$ are as in the statement.  We must show that $(\alpha\sigma^\beta, \beta x) \sim (\alpha'\sigma^{\beta'}, \beta' x')$.  First, we have
\[
\beta x = \beta \delta z = \beta'\delta' z = \beta' x'.
\]
Second, let $y \in X_{r(z)}$.  Then
\[
\alpha\sigma^\beta \beta\delta y = \alpha\delta y = \alpha'\delta' y = \alpha' \sigma^{\beta'} \beta'\delta' y = \alpha' \sigma^{\beta'} \beta\delta y.
\]
Hence $\alpha\sigma^\beta = \alpha'\sigma^{\beta'}$ on $\beta\delta X_{r(z)}$, a neighborhood of $\beta\delta z = \beta x$.
\end{proof}

Using Lemma \ref{l.fa.equivalencerelation}, we may redefine the groupoid $G$ (when $\Lambda$ is finitely aligned) as follows.

\begin{Definition}
\label{d.fa.groupoid.one}
Let $\Lambda$ be a finitely aligned category of paths.  We define a relation $\sim$ on
\[
\Lambda * \Lambda * X = \bigcup_{v \in \Lambda^0} \Lambda v \times \Lambda v \times X_v
\]
by:  $(\alpha,\beta,x) \sim (\alpha',\beta',x')$ if there exist $z \in X$, and $\delta$, $\delta' \in \Lambda r(z)$, such that
\begin{enumerate}
\item $x = \delta z$.
\item $x' = \delta' z$.
\item $\alpha\delta = \alpha'\delta'$.
\item $\beta\delta = \beta'\delta'$.
\end{enumerate}
\end{Definition}

\begin{Lemma}
\label{l.fa.equivalencerelationtwo}
The relation $\sim$ in Definition \ref{d.fa.groupoid.one} is an equivalence relation.
\end{Lemma}

\begin{proof}
Notice that the map $(\alpha \sigma^\beta, \beta x) \in \CZ * X \mapsto (\alpha, \beta, x) \in \Lambda * \Lambda * X$ is bijective.  By Lemma \ref{l.fa.equivalencerelation}, this map carries the equivalence relation on $\CZ * X$ to the relation of Definition \ref{d.fa.groupoid.one}.
\end{proof}

\begin{Remark}
\label{r.fa.groupoid.two}
The \textit{groupoid} of $\Lambda$ is the set 
\[
G \equiv G(\Lambda) = \Bigl( \bigcup_{v \in \Lambda^0} \Lambda v \times \Lambda v \times X_v \Bigr) \Bigm/ \sim.
\]
The set of composable pairs is
\[
G^2 = \bigl\{ \bigl([\alpha,\beta,x],\;
[\gamma,\delta,y]\bigr) : \beta x = \gamma y \bigr\},
\]
and inversion is given by $[\alpha,\beta,x]^{-1} = [\beta,\alpha,x]$ (where we use square brackets to denote equivalence classes).  Multiplication $G^2 \to G$ is given as follows.  Let $\bigl([\alpha,\beta,x], \; [\gamma,\delta,y]\bigr) \in G^2$.  Since $\beta x = \gamma y$, Lemma \ref{l.commontail} provides $z$, $\xi$, and $\eta$ such that $x = \xi z$, $y = \eta z$, and $\beta\xi = \gamma\eta$.  Then
\[
[\alpha,\beta,x] \, [\gamma,\delta,y]\ =\ [\alpha\xi,\delta\eta,z].
\]

The topology of $G$ is given as follows.  Let $v \in \Lambda^0$.  For $E \in \CA_v$ and $\alpha$, $\beta \in \Lambda v$, let
\[
[\alpha,\beta,E]
= \bigl\{ [\alpha,\beta,x] : x \in \widehat{E} \bigr\}.
\]
Thus $[\alpha\sigma^\beta, \beta E]$ corresponds to $[\alpha,\beta,E]$.  Then $\CB = \bigl\{ [\alpha,\beta,E] : s(\alpha) = s(\beta),\ E \in \CA_{s(\alpha)} \bigr\}$ is the base for the topology of $G$.
\end{Remark}

\begin{Proposition}
\label{p.fa.topology}
Let $\Lambda$ be a finitely aligned category of paths.  Then the topology of $G$ is Hausdorff.
\end{Proposition}

\begin{proof}
Let $(\alpha,\beta,x) \not\sim (\alpha',\beta',x')$.  First suppose that $\alpha x \not= \alpha' x'$.  By continuity of concatenation (Theorem \ref{t.concatenation}), there are $E \in \CA_{r(x)}$ and $E' \in \CA_{r(x')}$ such that $\alpha\widehat{E} \cap \alpha'\widehat{E'} = \emptyset$.  Then $[\alpha,\beta,x] \in [\alpha,\beta,E]$, $[\alpha',\beta',x'] \in [\alpha',\beta',E']$, and $[\alpha,\beta,E] \cap [\alpha',\beta',E'] = \emptyset$.  A similar argument treats the case where $\beta x \not= \beta' x'$.  Suppose now that $\alpha x = \alpha' x'$ and $\beta x = \beta' x'$.  Since $\alpha x = \alpha' x'$, Lemma \ref{l.commontail} provides $z$, $\delta$, $\delta'$ such that $x = \delta z$, $x'= \delta' z$, and $\alpha\delta = \alpha'\delta'$.  We have
\[
(\alpha\delta,\beta\delta,z)
\sim (\alpha,\beta,x)
\not\sim (\alpha',\beta',x')
\sim (\alpha'\delta',\beta'\delta',z).
\]
Thus $\beta\delta \not= \beta'\delta'$.    We claim that 
\[
[\alpha\delta,\beta\delta,X_{r(z}]
\cap
[\alpha'\delta',\beta'\delta',X_{r(z}]
= \emptyset.
\]
For let $y$, $y' \in X_{r(z)}$ be such that $(\alpha\delta,\beta\delta,y) \sim (\alpha'\delta',\beta'\delta',y')$.  Then there are $w$, $\eps$, $\eps'$ such that $y = \eps w$, $y' = \eps' w$, $\alpha\delta\eps = \alpha'\delta'\eps'$, and $\beta\delta\eps = \beta'\delta'\eps'$.  Since $\alpha\delta = \alpha'\delta'$, we have $\eps = \eps'$, by left-cancellation.  But then right-cancellation implies that $\beta\delta = \beta'\delta'$, a contradiction.
\end{proof}

\begin{Remark}
\label{r.rightgroupoid}
We note that it is only in the last sentence of the above proof that right-cancellation has been used (cf. Remark \ref{r.rightcancellation}).  If right-cancellation is not assumed, the groupoid $G(\Lambda)$ need not be Hausdorff in the finitely aligned case.  However it will still be ample and \'etale, and will define a \cstar-algebra as in the general case.
\end{Remark}

\section{Boolean ring homomorphisms}
\label{s.boolean}

We now wish to characterize Boolean ring homomorphisms from $\CA_v$.  Our goal is the following theorem.

\begin{Theorem}
\label{t.ringhomomorphisms.one}
Let $(\Lambda_0,\Lambda)$ be a relative category of paths, $v \in \Lambda_0^0$, and $\CR$ a Boolean ring.  A map $\mu : \CD^{(0)}_v \to \CR$ extends to a Boolean ring homomorphism $\CA_v \to \CR$ if and only if the following conditions hold:
\begin{enumerate}
\item $\mu(E\cap F) = \mu(E) \cap \mu(F)$, for $E$, $F \in \CD^{(0)}_v$.
\label{t.ringhomomorphisms.one.a}
\item $\mu(E) = \bigcup_{i=1}^n \mu(F_i)$ for $E$, $F_1$, $\ldots$, $F_n \in \CD^{(0)}_v$ with $E = \bigcup_{i=1}^n F_i$.
\label{t.ringhomomorphisms.one.b}
\end{enumerate}
In this case, the extension to $\CA_v$ is unique.
\end{Theorem}

This is truly a set-theoretic result: it relies only on Lemma \ref{l.setlemma} and the fact that $\CD^{(0)}_v$ is closed under intersection.  Before proving the theorem, we will present several lemmas.

\begin{Lemma}
\label{l.ringhomomorphisms.a}
Let $(\Lambda_0,\Lambda)$, $v$, $\CR$ and $\mu$ be as in the statement of Theorem \ref{t.ringhomomorphisms.one}, and suppose that \eqref{t.ringhomomorphisms.one.a} and \eqref{t.ringhomomorphisms.one.b} hold.  If $E$, $F_1$, $\ldots$, $F_n \in \CD^{(0)}_v$, and $E \subseteq \bigcup_{i=1}^n F_i$, then $\mu(E) \subseteq \bigcup_{i=1}^n \mu(F_i)$.
\end{Lemma}

\begin{proof}
We first note that $\mu$ is monotone on $\CD^{(0)}_v$.  Let $E$, $F \in \CD^{(0)}_v$ with $E \subseteq F$.  Then $E = E \cap F$, so by \eqref{t.ringhomomorphisms.one.a} we have $\mu(E) = \mu(E) \cap \mu(F) \subseteq \mu(F)$.  Now, from $E \subseteq \bigcup_{i=1}^n F_i$ we have $E = \bigcup_{i=1}^n (E \cap F_i)$, and $E \cap F_i \in \CD^{(0)}_v$.  By \eqref{t.ringhomomorphisms.one.b}, and monotonicity,
\[
\mu(E) = \bigcup_{i=1}^n \mu(E \cap F_i)
\subseteq \bigcup_{i=1}^n \mu(F_i). \qedhere
\]
\end{proof}

\begin{Corollary}
\label{c.ringhomomorphisms.b}
If $E_1$, $\ldots$, $E_m$, $F_1$, $\ldots$, $F_n \in \CD^{(0)}_v$, and $\bigcup_{i=1}^m E_i \subseteq \bigcup_{j=1}^n F_j$, then $\bigcup_{i=1}^m \mu(E_i) \subseteq \bigcup_{j=1}^n \mu(F_j)$.
\end{Corollary}

\begin{proof}
By Lemma \ref{l.ringhomomorphisms.a}, $\mu(E_i) \subseteq \bigcup_{j=1}^n \mu(F_j)$ for all $i$.
\end{proof}

\begin{Remark}
\label{r.ringhomomorphisms.c}
It follows from Corollary \ref{c.ringhomomorphisms.b} that the analgous result with equalities in place of containments also holds, and hence that $\mu$ can be extended to finite unions of sets in $\CD^{(0)}_v$ by setting $\mu(\cup E_i) = \cup_i \mu(E_i)$.
\end{Remark}

\begin{Lemma}
\label{l.ringhomomorphisms.d}
We retain the hypotheses of Lemma \ref{l.ringhomomorphisms.a}.  Let $E_i$, $F_j \in \CD^{(0)}_v$ for $0 \le i \le m$ and $0 \le j \le n$, and suppose that $E_0 \setminus \bigcup_{i=1}^m E_i \subseteq F_0 \setminus \bigcup_{j=1}^n F_j$.  Then $\mu(E_0) \setminus \bigcup_{i=1}^m \mu(E_i) \subseteq \mu(F_0) \setminus \bigcup_{j=1}^n \mu(F_j)$.
\end{Lemma}

\begin{proof}
We note the set-theoretic identity:
\[
X \setminus Y \subseteq Z \setminus W \text{ if and only if } X \subseteq Y \cup Z \text{ and } X \cap W \subseteq Y.
\]
Thus
\[
E_0 \setminus \bigcup_{i=1}^m E_i \subseteq F_0 \setminus \bigcup_{j=1}^n F_j,
\]
hence
\[
E_0 \subseteq \bigcup_{i=1}^m E_i \cup F_0
\text{ and }
E_0 \cap \bigcup_{j=1}^n F_j \subseteq \bigcup_{i=1}^m E_i,
\]
hence
\[
E_0 \subseteq \bigcup_{i=1}^m E_i \cup F_0
\text{ and }
\bigcup_{j=1}^n (E_0 \cap F_j) \subseteq \bigcup_{i=1}^m E_i,
\]
hence
\[
\mu(E_0) \subseteq \bigcup_{i=1}^m \mu(E_i) \cup \mu(F_0)
\text{ and }
\bigcup_{j=1}^n \mu(E_0 \cap F_j) \subseteq \bigcup_{i=1}^m \mu(E_i),
\text{ by Corollary \ref{c.ringhomomorphisms.b},}
\]
hence
\[
\mu(E_0) \setminus \bigcup_{i=1}^m \mu(E_i) \subseteq \mu(F_0) \setminus \bigcup_{j=1}^n \mu(F_j),
\text{ by \eqref{t.ringhomomorphisms.one.a}.} \qedhere
\]
\end{proof}

\begin{Remark}
\label{r.ringhomomorphisms.e}
It follows from Lemma \ref{l.ringhomomorphisms.d} that $\mu$ can be extended to $\CD_v$ by setting $\mu(E \setminus \bigcup_i E_i) = \mu(E) \setminus \bigcup_i \mu(E_i)$.
\end{Remark}

\begin{Lemma}
\label{l.ringhomomorphisms.f}
We retain the hypotheses of Lemma \ref{l.ringhomomorphisms.a}.  Let $\mu$ be extended to $\CD_v$ as in Remark \ref{r.ringhomomorphisms.e}.  Let $A$, $B \in \CD_v$ be disjoint.  Then $\mu(A)$ and $\mu(B)$ are disjoint.
\end{Lemma}

\begin{proof}
We note the set-theoretic identity:
\[
(X \setminus Y) \cap (Z \setminus W) = \emptyset
\text{ if and only if }
X \cap Z \subseteq Y \cup W.
\]
Let $A = E_0 \setminus \bigcup_{i=1}^m E_i$ and $B = F_0 \setminus \bigcup_{j=1}^n F_j$, with $E_i$, $F_j \in \CD^{(0)}_v$.  Since $A \cap B = \emptyset$, we have
\[
E_0 \cap F_0 \subseteq \bigcup_{i=1}^m E_i \cup \bigcup_{j=1}^n F_j,
\]
hence
\[
\mu(E_0) \cap \mu(F_0) \subseteq \bigcup_{i=1}^m \mu(E_i) \cup \bigcup_{j=1}^n \mu(F_j),
\]
by \eqref{t.ringhomomorphisms.one.a} and Lemma \ref{l.ringhomomorphisms.a}; hence
\[
\left[ \mu(E_0) \setminus \bigcup_{i=1}^m \mu(E_i) \right] \cap \left[ \mu(F_0) \setminus \bigcup_{j=1}^n \mu(F_j) \right] = \emptyset. \qedhere
\] 
\end{proof}

\begin{Lemma}
\label{l.ringhomomorphisms.g}
We retain the hypotheses of Lemma \ref{l.ringhomomorphisms.a}.  Let $\mu$ be extended to $\CD_v$ as in Remark \ref{r.ringhomomorphisms.e}.  Let $A$, $A_1$, $\ldots$, $A_p \in \CD_v$ with $A = \bigsqcup_{k=1}^p A_k$.  Then $\mu(A) = \bigsqcup_{k=1}^p \mu(A_k)$.
\end{Lemma}

\begin{proof}
The containment ``$\supseteq$'' follows from Lemma \ref{l.ringhomomorphisms.d}, and the disjointness from Lemma \ref{l.ringhomomorphisms.f}.  For ``$\subseteq$'' we first note the set-theoretic identity:
\[
X \setminus Y \subseteq \bigcup_{k=1}^p (Z_k \setminus W_k)
\]
if and only if for every $I \subseteq \{1, \ldots, p\}$,
\[
X \cap \bigcap_{k\in I} W_k \subseteq Y \cup \bigcup_{\ell \not\in I} Z_\ell.
\]
Let $A = E \setminus \bigcup_{i=1}^m E_i$ and $A_k = F_k \setminus \bigcup_{j=1}^{n_k} F_{kj}$, where $E$, $E_i$, $F_k$, $F_{kj} \in \CD^{(0)}_v$.  Let $I \subseteq \{1, \ldots, p\}$.  We have
\[
E \cap \bigcap_{k \in I} \bigcup_{j=1}^{n_k} F_{kj} \subseteq \bigcup_{i=1}^m E_i \cup \bigcup_{\ell \not\in I} F_\ell.
\]
Since the left-hand side is a union of intersections, it follows from Corollary \ref{c.ringhomomorphisms.b} and \eqref{t.ringhomomorphisms.one.a} that
\[
\mu(E) \cap \bigcap_{k \in I} \bigcup_{j=1}^{n_k} \mu(F_{kj}) \subseteq \bigcup_{i=1}^m \mu(E_i) \cup \bigcup_{\ell \not\in I} \mu(F_\ell).
\]
Since this is true for every set $I \subseteq \{1, \ldots, p\}$, we have
\[
\mu(E) \setminus \bigcup_{i=1}^m \mu(E_i) \subseteq \bigcup_{k=1}^p \left( \mu(F_k) \setminus \bigcup_{j=1}^{n_k} \mu(F_{kj}) \right),
\]
that is $\mu(A) \subseteq \bigcup_{k=1}^p \mu(A_k)$.
\end{proof}

\begin{Corollary}
\label{c.ringhomomorphisms.h}
Let $\{A_1, \ldots, A_m\}$ and $\{B_1, \ldots, B_n\}$ be families of pairwise disjoint elements of $\CD_v$, and suppose that $\sqcup_i A_i = \sqcup_j B_j$.  Then $\sqcup_i \mu(A_i) = \sqcup_j \mu(B_j)$.
\end{Corollary}

\begin{proof}
Let $A_i \cap B_j = \sqcup_k C_{ijk}$, where $C_{ijk} \in \CD_v$ and the union is finite.  Since $A_i = \sqcup_j  (A_i \cap B_j) = \sqcup_{j,k} C_{ijk}$, Lemma \ref{l.ringhomomorphisms.g} gives
\[
\sqcup_i \mu(A_i) = \sqcup_i \mu( \sqcup_{j,k} C_{ijk}) = \sqcup_{i,j,k} \mu(C_{ijk}).
\]
Similarly, $\sqcup_j \mu(B_j) = \sqcup_{i,j,k} \mu(C_{ijk})$.
\end{proof}

\begin{Definition}
\label{d.ringhomomorphisms.i}
Let $(\Lambda_0,\Lambda)$, $v$, $\CR$ and $\mu$ be as in the statement of Theorem \ref{t.ringhomomorphisms.one}, and suppose that \eqref{t.ringhomomorphisms.one.a} and \eqref{t.ringhomomorphisms.one.b} hold.  We define $\mu$ on $\CD_v$ as in Remark \ref{r.ringhomomorphisms.e}, and then on $\CA_v$ by setting $\mu(\sqcup_i A_i) = \sqcup_i \mu(A_i)$ (which is unambiguous by Corollary \ref{c.ringhomomorphisms.h}).
\end{Definition}

\begin{proof} \textit{(of Theorem \ref{t.ringhomomorphisms.one}.)}
The necessity and uniqueness are clear.  For the sufficiency, first note that it follows from Lemma \ref{l.ringhomomorphisms.f} and Definition \ref{d.ringhomomorphisms.i} that if $A$, $B \in \CA_v$ are disjoint, then $\mu(A)$ and $\mu(B)$ are disjoint, and $\mu(A \sqcup B) = \mu(A) \sqcup \mu(B)$.  Now, for arbitrary $A$, $B \in \CA_v$ we have
\begin{align*}
\mu(A) &= \mu(A \setminus B) \sqcup \mu(A \cap B) \\
\mu(B) &= \mu(B \setminus A) \sqcup \mu(A \cap B) \\
\mu(A \setminus B) \cap \mu(B \setminus A) &= \emptyset.
\end{align*}
It follows that 
\[
\mu(A \setminus B) = \mu(A) \setminus \mu(B),
\]
and hence that
\[
\mu(A \cup B)
= \mu\bigl((A \setminus B) \sqcup B\bigr)
= \mu(A \setminus B) \sqcup \mu(B)
= \bigl(\mu(A) \setminus \mu(B)\bigr) \sqcup \mu(B)
= \mu(A) \cup \mu(B). \qedhere
\]
\end{proof}

In the case that $\Lambda$ is finitely aligned, the hypotheses of the theorem may be relaxed considerably.  It is the following result that fundamentally explains why the (Toeplitz) Cuntz-Krieger relations have their usual form.

\begin{Theorem}
\label{t.ringhomomorphisms.two}
Let $\Lambda$ be a finitely aligned category of paths, $v \in \Lambda^0$, $\CR$ a Boolean ring, and $\mu : \CE^{(0)}_v \to \CR$.  Then $\mu$ extends to a Boolean ring homomorphism $\CA_v \to \CR$ if and only if the following condition holds:
\begin{enumerate}
\setcounter{enumi}{2}
\item $\mu(\alpha\Lambda) \cap \mu(\beta\Lambda) = \bigcup_{\eps \in \alpha \vee \beta} \mu(\eps\Lambda)$ for all $\alpha$, $\beta \in v\Lambda$.
\label{t.ringhomomorphisms.two.c}
\end{enumerate}
In this case, the extension is unique.
\end{Theorem}

\begin{proof}
The necessity and uniqueness are clear.  For the sufficiency, first note that \eqref{t.ringhomomorphisms.two.c} implies monotonicity of $\mu$ on $\CE^{(0)}_v$.  Next, we observe that the sets in $\CE^{(0)}_v$ have the following property:  if $E$, $F_1$, $\ldots$, $F_n \in \CE^{(0)}_v$ and $E \subseteq \cup_j F_j$, then there is $j_0$ such that $E \subseteq F_{j_0}$.  (For, letting $E = \alpha\Lambda$ and $F_j = \beta_j\Lambda$, $\alpha \in E$ implies $\alpha \in \beta_{j_0}\Lambda$ for some $j_0$, and hence $\alpha\Lambda \subseteq \beta_{j_0}\Lambda$.)  It follows that $\mu$ can be defined unambiguously on finite unions of sets in $\CE^{(0)}_v$, and that $\mu$ remains monotone when so extended.  By Corollary \ref{c.finitelyalignedzigzag} it follows that $\mu$ is defined on $\CD^{(0)}_v$, and that \eqref{t.ringhomomorphisms.one.b} holds.  Let $A$, $B \in \CD^{(0)}_v$.  Write $A = \bigcup_{i=1}^m E_i$ and $B = \bigcup_{j=1}^n F_j$, with $E_i = \alpha_i\Lambda$, $F_j = \beta_j\Lambda$ in $\CE^{(0)}_v$.  Then $E_i \cap F_j = \bigcup_{\eps \in \alpha_i \vee \beta_j} G_{ij\eps}$, where $G_{ij\eps} = \eps\Lambda$, $\eps \in \alpha_i \vee \beta_j$.  Then 
\[
\mu(A \cap B) = \bigcup_{ij\eps} \mu(G_{ij\eps}).
\]
On the other hand,
\[
\mu(A) \cap \mu(B) = \bigcup_{ij} \mu(E_i) \cap \mu(F_j)
= \bigcup_{ij} \bigcup_{\eps \in \alpha_i \vee \beta_j} \mu(G_{ij\eps}),
\]
by \eqref{t.ringhomomorphisms.two.c}.  Thus \eqref{t.ringhomomorphisms.one.a} holds.  By Theorem \ref{t.ringhomomorphisms.one}, $\mu$ has an extension to $\CA_v$.
\end{proof}

\section{Generators and relations}
\label{s.gensandrels}

Let $G$ be the groupoid of a relative category of paths $(\Lambda_0,\Lambda)$.  For $\zeta \in \CZ$ we let $t_\zeta = \chi_{[\zeta,A(\zeta)]}$.  The collection $\{ t_\zeta : \zeta \in \CZ \}$ is a total set in $C_c(G) \subseteq C^*(G)$.  Since $\pphi_{\zeta_2}^{-1} (A(\zeta_1)) = \pphi_{\overline{\zeta_2}} (\pphi_{\overline{\zeta_1}}(\Lambda)) = \pphi_{\overline{\zeta_1 \zeta_2}}(\Lambda) = A(\zeta_1 \zeta_2)$, we see that $t_{\zeta_1} t_{\zeta_2} = t_{\zeta_1 \zeta_2}$.  It is clear that $t_{\overline{\zeta}} = t_\zeta^*$.  Since $A(\zeta) \mapsto \chi_{[\overline{\zeta} \zeta, A(\zeta)]} = t_{\overline{\zeta} \zeta}$ is a Boolean ring homomorphism, conditions \eqref{t.ringhomomorphisms.one.a} and \eqref{t.ringhomomorphisms.one.b} of Theorem \ref{t.ringhomomorphisms.one} hold for this map.  In fact, \eqref{t.ringhomomorphisms.one.a} follows from the previous two properties and Remark \ref{r.zigzag}\eqref{r.zigzag.e}.  
Notice also that $\pphi_\zeta = \text{id}_{A(\zeta)}$ if and only if $\Phi_\zeta = \text{id}_{\widehat{A(\zeta)}}$ (for the only if, consider the fixed ultrafilter determined by a path moved by $\pphi_\zeta$).  Therefore, if $\pphi_\zeta = \text{id}_{A(\zeta)}$ we have $t_\zeta = t_\zeta^* t_\zeta$.
These are enough to characterize representations of $C^*(G)$.
Before presenting the characterization, we wish to comment on the last property mentioned above, as it is special to the non-finitely aligned case.  (We thank Allan Donsig and David Milan for noticing a gap in our earlier proof of Theorem \ref{t.gensandrels}, which brought to our attention the need for the relation \eqref{t.gensandrels.d} of the theorem.)  We give an example.  In Figure \ref{fig1gens}, $i = 1$, 2, 3, $\ldots$, and we have the identifications $\alpha_j \gamma_i = \beta_j \delta_i$ for $j = 1$, 2, and for all $i$.  Letting $\zeta_j = (\alpha_j,\beta_j)$, we find that $A(\zeta_j) = \{ \delta_i : i \ge 1 \}$, and that $\pphi_{\zeta_j}(\delta_i) = \gamma_i$ for $j = 1$, 2, and for all $i$.  Thus $\Phi_{\zeta_1} = \Phi_{\zeta_2}$: the groupoid does not distinguish $\zeta_1$ and $\zeta_2$.

\begin{figure}[ht]

\[
\begin{tikzpicture}[scale=1]

\node (0_0) at (0,0) [circle] {};
\node (2_0) at (2,0) [circle] {};
\node (m2_0) at (-2,0) [circle] {};
\node (0_2) at (0,2) [circle] {};
\node (0_m2) at (0,-2) [circle] {};

\foreach \x in {0,-2,2} \filldraw [black] (\x,0) circle (1pt);
\foreach \x in {0,-2,2} \filldraw [black] (0,\x) circle (1pt);

\draw[-latex,thick] (0_0) -- (m2_0) node[pos=0.5, inner sep=0.5pt, anchor=north] {$\gamma_i$};
\draw[-latex,thick] (m2_0) -- (0_2) node[pos=0.5, inner sep=0.5pt, anchor=south east] {$\alpha_2$};
\draw[-latex,thick] (m2_0) -- (0_m2) node[pos=0.5, inner sep=0.5pt, anchor=north east] {$\alpha_1$};
\draw[-latex,thick] (0_0) -- (2_0) node[pos=0.5, inner sep=0.5pt, anchor=north] {$\delta_i$};
\draw[-latex,thick] (2_0) -- (0_2) node[pos=0.5, inner sep=0.5pt, anchor=south west] {$\beta_2$};
\draw[-latex,thick] (2_0) -- (0_m2) node[pos=0.5, inner sep=0.5pt, anchor=north west] {$\beta_1$};

\end{tikzpicture}
\]

\caption{} \label{fig1gens}

\end{figure}

\begin{Theorem}
\label{t.gensandrels}
Let $G$ be the groupoid of a relative category of paths $(\Lambda_0,\Lambda)$.  The representations of $C^*(G)$ are in one-to-one correspondence with the families $\{ T_\zeta : \zeta \in \CZ \}$ of Hilbert space operators satisfying the relations
\begin{enumerate}
\item $T_{\zeta_1} T_{\zeta_2} = T_{\zeta_1 \zeta_2}$.
\label{t.gensandrels.a}
\item $T_{\overline{\zeta}} = T_\zeta^*$.
\label{t.gensandrels.b}
\item $T_\zeta^* T_\zeta = \bigvee_{i=1}^n T_{\zeta_i}^* T_{\zeta_i}$ if $A(\zeta) = \bigcup_{i=1}^n A(\zeta_i)$.
\label{t.gensandrels.c}
\item $T_\zeta = T^*_\zeta T_\zeta$ if $\pphi_\zeta = \text{id}_{A(\zeta)}$.
\label{t.gensandrels.d}
\end{enumerate}
\end{Theorem}

\begin{proof}
First, let $\pi : C^*(G) \to B(H)$ be a representation.  Let $T_\zeta = \pi\bigl(\chi_{[\zeta,A(\zeta)]}\bigr)$.  As indicated in the discussion before the statement of the theorem, conditions \eqref{t.gensandrels.a} - \eqref{t.gensandrels.d} hold.

Conversely, let $\{T_\zeta : \zeta \in \CZ \} \subseteq B(H)$ satisfy \eqref{t.gensandrels.a} - \eqref{t.gensandrels.d}.  
Since $A(\zeta) = A(\overline{\zeta}\zeta)$, we have
\begin{align*}
T_{\overline{\zeta}\zeta} &= T_\zeta^* T_\zeta, \text{ by \eqref{t.gensandrels.a} and \eqref{t.gensandrels.b}}, \\
&= T_{\overline{\zeta}\zeta}^* T_{\overline{\zeta}\zeta}, \text{ by \eqref{t.gensandrels.c},} \\
&= T_{\overline{\zeta}\zeta} T_{\overline{\zeta}\zeta}, \text{ since } T_{\overline{\zeta}\zeta} = T_\zeta^* T_\zeta \text{ is self-adjoint,} \\
&= T_{\overline{\zeta}\zeta \overline{\zeta}\zeta}, \text{ by \eqref{t.gensandrels.a}}.
\end{align*}
Therefore $T_\zeta$ is a partial isometry.  
Next we use \eqref{t.gensandrels.d}.  We claim that
\begin{equation}
\text{if } \pphi_{\zeta_1} = \pphi_{\zeta_2} \text{ then } T_{\zeta_1} = T_{\zeta_2}.
\tag{5}
\end{equation}
To see this, note that if $\pphi_{\zeta_1} = \pphi_{\zeta_2}$, then $\pphi_{\overline{\zeta_1}\zeta_2} = \text{id}_{A(\zeta_2)} = \text{id}_{A(\zeta_1)} = \pphi_{\overline{\zeta_1}\zeta_1}= \pphi_{\overline{\zeta_2}\zeta_2}$.  By \eqref{t.gensandrels.d} we then have $T_{\zeta_1}^* T_{\zeta_2} = T_{\overline{\zeta_1}\zeta_2} = T_{\zeta_1}^*T_{\zeta_1}= T_{\zeta_2}^*T_{\zeta_2}$.  We also have $\pphi_{\zeta_2\overline{\zeta_1}} = \text{id}_{A(\overline{\zeta_1})} = \text{id}_{A(\overline{\zeta_2})} = \pphi_{\zeta_1\overline{\zeta_1}}= \pphi_{\zeta_1\overline{\zeta_1}}$, and hence similarly, that $T_{\zeta_2} T_{\zeta_1}^* = T_{\zeta_2\overline{\zeta_1}} = T_{\zeta_1} T_{\zeta_1}^*= T_{\zeta_2} T_{\zeta_2}^*$.  Now we have
\[
T_{\zeta_1} = T_{\zeta_1} T_{\zeta_1}^* T_{\zeta_1} = T_{\zeta_2} T_{\zeta_1}^* T_{\zeta_1} = T_{\zeta_2} T_{\zeta_2}^* T_{\zeta_2} = T_{\zeta_2},
\]
proving (5).

We define $\mu : \CD^{(0)} \to B(H)$ by $\mu(A(\zeta)) = T_\zeta^* T_\zeta$.  Since $ T_{\overline{\zeta_1} \zeta_1}  T_{\overline{\zeta_2} \zeta_2} =  T_{\overline{\zeta_1} \zeta_1 \overline{\zeta_2} \zeta_2}$, condition \eqref{t.ringhomomorphisms.one.a} of Theorem \ref{t.ringhomomorphisms.one} holds.  Condition \eqref{t.ringhomomorphisms.one.b} of Theorem \ref{t.ringhomomorphisms.one} holds by \eqref{t.gensandrels.c}.  Then  by Theorem \ref{t.ringhomomorphisms.one} we obtain a $*$-homomorphism $\pi_0 : C_0(G^0) \to B(H)$ such that $\pi_0(\chi_{A(\zeta)}) = T_\zeta^* T_\zeta$.

In order to extend $\pi_0$ to all of $C^*(G)$, we proceed locally.  Let $f \in C_c(G)$ be such that there is $\zeta \in \CZ$ with $\text{supp}\,(f) \subseteq [\zeta,A(\zeta)]$.  Define $\widetilde{f} \in C\bigl(\widehat{A(\zeta)}\bigr)$ by $\widetilde{f}(x) = f([\zeta,x])$.  To show that $\widetilde{f}$ is well-defined, suppose that $\text{supp}\,(f) \subseteq [\zeta',A(\zeta')]$.  For each $w \in \text{supp}\,(f)$, $w = [\zeta,x] = [\zeta',x]$ for some $x \in \widehat{A(\zeta)} \cap \widehat{A(\zeta')}$.  Therefore $\text{supp}\,(f) \subseteq [\zeta, A(\zeta) \cap A(\zeta')] \cap [\zeta', A(\zeta) \cap A(\zeta')]$.  Let $E \in \CA$ with $\widehat{E} = \{x \in \widehat{A(\zeta)} \cap \widehat{A(\zeta')} : \Phi_\zeta = \Phi_{\zeta'} \text{ near } x \}$.  Then $\text{supp}\,(f) \subseteq [\zeta,E] \cap [\zeta',E]$.  Thus $\widetilde{f}$ doesn't depend on the choice of $\zeta$.

Now we define $\pi(f) = T_\zeta \pi_0(\widetilde{f})$ if $\text{supp}\,(f) \subseteq [\zeta,A(\zeta)]$.
To see that this is well-defined, let supp$(f) \subseteq [\zeta_1, A(\zeta_1)] \cap [\zeta_2, A(\zeta_2)]$.  First suppose that $f = \chi_{[\xi,A(\xi)]}$.  Then $A(\xi) \subseteq A(\zeta_1) \cap A(\zeta_2)$, and $\Phi_{\zeta_j}|_{A(\xi)} = \Phi_\xi$, $j = 1$, 2.  But then $\pphi_{\zeta_1 \overline{\xi} \xi} = \pphi_{\zeta_2 \overline{\xi} \xi}$, and hence $T_{\zeta_1} T_\xi^* T_\xi = T_{\zeta_2} T_\xi^* T_\xi$, by (5).  We have $\widetilde{f} = \chi_{\widehat{A(\xi)}}$, and hence $\pi_0(\widetilde{f}) = T_\xi^* T_\xi$.  But then $T_{\zeta_1} \pi_0(\widetilde{f}) = T_{\zeta_1}T_\xi^* T_\xi = T_{\zeta_2} T_\xi^* T_\xi = T_{\zeta_2} \pi_0(\widetilde{f})$.  It follows that this also holds for $f$ in the span of such characteristic functions.  By continuity of $\pi_0$, it follows for all $f$ supported in basic sets $[\zeta,A(\zeta)]$.

For an arbitrary $f \in C_c(G)$, we may find $\zeta_i \in \CZ$ and $A_i \in \CA$, $1 \le i \le n$, such that $\text{supp}\,(f) \subseteq \bigsqcup_i [\zeta_i,A_i]$.  Then $f = \sum_i f |_{[\zeta_i,A_i]}$.  If also $\text{supp}\,(f) \subseteq \bigsqcup_j [\xi_j,B_j]$, then since $[\zeta_i,A_i] \cap [\xi_j,B_j] = [\zeta_i,A_i \cap B_j] = [\xi_j, A_i \cap B_j]$, we have
\[
\sum_i f |_{[\zeta_i,A_i]} = \sum_{i,j} f |_{[\zeta_i,A_i \cap B_j]} = \sum_{i,j} f |_{[\xi_j,A_i \cap B_j]} = \sum_j f |_{[\xi_j,B_j]}.
\]
Thus $\sum_i \pi \bigl( f |_{[\zeta_i,A_i]} \bigr) = \sum_j \bigl( f |_{[\xi_j,B_j]} \bigr)$.  Therefore this last expression is a well-defined extension of $\pi$ to all of $C_c(G)$, and is a self-adjoint linear map.  We note that $\pi$ is continuous for the inductive limit topology, since by the above it reduces to uniform convergence on the sets $\widehat{A(\zeta)}$.  Finally, since $\pi$ is multiplicative on the characteristic functions of the basic sets $[\zeta,A(\zeta)]$, the continuity implies that $\pi$ is multiplicative on $C_c(G)$.  Therefore $\pi$ extends to all of $C^*(G)$, by Renault's disintegration theorem (\cite{pat}, Theorem 3.1.1).
\end{proof}

The third relation in Theorem \ref{t.gensandrels} is fairly complicated, reflecting the complexity of general non-finitely aligned categories of paths.  Our use of relative categories of paths gives a nesting relation for subcategories.

\begin{Corollary}
\label{c.nestedcategories}
Let $\Lambda$ be a category of paths, and let $\Lambda_0 \subseteq \Lambda_1 \subseteq \Lambda$ be subcategories.  For $\zeta \in \Lambda_j$ let $t_{\zeta,j} \in C^*(G(\Lambda_j,\Lambda))$ be the generators described in the remarks before Theorem \ref{t.gensandrels}.  There is a $*$-homomorphism $C^*(G(\Lambda_0,\Lambda)) \to C^*(G(\Lambda_1,\Lambda))$ given by $t_{\zeta,0} \mapsto t_{\zeta,1}$ for $\zeta \in \CZ(\Lambda_0,\Lambda)$.  Moreover, if $G(\Lambda_0,\Lambda)$ is amenable then this map is injective.
\end{Corollary}

\begin{proof}
The existence of the map follows from Theorem \ref{t.gensandrels}.  Now let $v \in \Lambda_0^0$.  If $\CU$ is an ultrafilter in $\CA(\Lambda_1,\Lambda)_v$, let $\CU_0 = \CU \cap \CA(\Lambda_0,\Lambda)_v$.  It is easy to see that $\CU_0$ is a filter.  Moreover, if $E \in \CA(\Lambda_0,\Lambda)_v \setminus \CU_0$, then $E \not\in \CU$.  But then $v\Lambda \setminus E \in \CU \cap \CA(\Lambda_0,\Lambda)_v$.  Therefore $\CU_0$ is an ultrafilter.  Thus we obtain a map $X(\Lambda_1,\Lambda)_v \to X(\Lambda_0,\Lambda)_v$ given by $\CU \mapsto \CU \cap \CA(\Lambda_0,\Lambda)_v$.  It is easy to check that this is a continuous proper surjection.  It extends to the groupoids as follows.  Let $\zeta \in \CZ(\Lambda_0,\Lambda) v$, and $\CU \in X(\Lambda_1,\Lambda)_v$.  Let $\CU_0 = \CU \cap \CA(\Lambda_0,\Lambda)_v$.  Then the map on groupoids is given by $[\zeta,\CU] \mapsto [\zeta,\CU_0]$.  It follows that the regular representation of $G(\Lambda_0,\Lambda)$ induced from the point mass at $\CU_0$ is the restriction of the regular representation of $G(\Lambda_1,\Lambda)$ induced from the point mass at $\CU$ (\cite{pat}, section 3.1).  If $G(\Lambda_0,\Lambda)$ is amenable, then the norm in $C^*(G(\Lambda_0,\Lambda))$ is determined by these regular representations.  Therefore in this case, the map on $C^*$-algebras is isometric.
\end{proof}

We next present the simplification occurring when $\Lambda$ is finitely aligned.  In the next result we will use the following notations.  Recall the convention on unions of functions stated before Lemma \ref{l.finitelyalignedzigzag}.  In a similar way, let $\{ u_i : i \in I \}$ be a finite collection of partial isometries (in a $C^*$-algebra), whose initial and final projections form a commuting family.  The $\{ u_i \}$ determine a partial isometry, with initial projection $\bigvee_{i \in I} u_i^* u_i$, and final projection $\bigvee_{i \in I} u_i u_i^*$, if and only if for all $i$, $j \in I$, $u_i u_j^* u_j = u_j u_i^* u_i$.  We will write $\bigvee_{i \in I} u_i$ for the partial isometry so determined if and only if the $\{ u_i \}$ satisfy this condition.

\begin{Theorem}
\label{t.toeplitzrelations}
Let $\Lambda$ be a finitely aligned category of paths, and let $G = G(\Lambda)$.  The representations of $C^*(G)$ are in one-to-one correspondence with the families $\{T_\alpha : \alpha \in \Lambda\}$ of Hilbert space operators satisfying the relations
\begin{enumerate}
\item $T_\alpha^* T_\alpha = T_{s(\alpha)}$.
\label{t.toeplitzrelations.a}
\item $T_\alpha T_\beta = T_{\alpha\beta}$, if $s(\alpha) = r(\beta)$.
\label{t.toeplitzrelations.b}
\item $T_\alpha T_\alpha^* T_\beta T_\beta^*
= \bigvee_{\gamma \in \alpha \vee \beta} T_\gamma T_\gamma^*$.
\label{t.toeplitzrelations.c}
\end{enumerate}
\end{Theorem}

\begin{proof}
First suppose that we have a representation of $C^*(G)$.  By Theorem \ref{t.gensandrels} we have a family of Hilbert space operators $\{ T_\zeta : \zeta \in \CZ \}$.  For $\alpha \in \Lambda$ we define $T_\alpha = T_{(r(\alpha),\alpha)}$.  Note that $A(r(\alpha),\alpha) = s(\alpha) \Lambda = A(s(\alpha),s(\alpha))$.  Then 
\[
T_\alpha^* T_\alpha = T_{(r(\alpha),\alpha)}^* T_{(r(\alpha),\alpha)} = T_{\overline{(r(\alpha),\alpha)} (r(\alpha),\alpha)} = T_{(s(\alpha),s(\alpha))} = T_{s(\alpha)},
\]
verifying \eqref{t.toeplitzrelations.a}.  Next, let $s(\alpha) = r(\beta)$.  Then
\[
T_\alpha T_\beta = T_{(r(\alpha),\alpha)} T_{(r(\beta),\beta)} = T_{(r(\alpha),\alpha,r(\beta),\beta)} = T_{(r(\alpha),\alpha\beta)} = T_{\alpha\beta},
\]
verifying \eqref{t.toeplitzrelations.b}.  Finally, let $\alpha$, $\beta \in \Lambda$.  Then $T_\alpha T_\alpha^* T_\beta T_\beta^* = T_{(r(\alpha),\alpha) \overline{(r(\alpha),\alpha)} (r(\beta),\beta) \overline{(r(\beta),\beta)}}$.  Note that
\begin{align*}
\pphi_{(r(\alpha),\alpha) \overline{(r(\alpha),\alpha)} (r(\beta),\beta) \overline{(r(\beta),\beta)}}
&= \pphi_{(r(\alpha),\alpha,\alpha,r(\alpha),r(\beta),\beta,\beta,r(\beta))}
= \alpha \sigma^\alpha \beta \sigma^\beta \\
&= \bigcup_{\eps \in \alpha \vee \beta} \alpha (\sigma^\alpha \eps)  \sigma^{\displaystyle (\sigma^\beta \eps)} \sigma^\beta
= \bigcup_{\eps \in \alpha \vee \beta} (\alpha \sigma^\alpha \eps) \sigma^{\displaystyle (\beta \sigma^\beta \eps)}
= \bigcup_{\eps \in \alpha \vee \beta} \eps \sigma^\eps. \\
\noalign{Hence}
A((r(\alpha),\alpha) \overline{(r(\alpha),\alpha)} (r(\beta),\beta) \overline{(r(\beta),\beta)})
&= \bigcup_{\eps \in \alpha \vee \beta} \eps \Lambda
= \bigcup_{\eps \in \alpha \vee \beta} A(\eps,r(\eps)).
\end{align*}
Therefore $T_{(r(\alpha),\alpha) \overline{(r(\alpha),\alpha)} (r(\beta),\beta) \overline{(r(\beta),\beta)}} = \bigvee_{\eps \in \alpha \vee \beta} T_\eps T_\eps^*$, verifying \eqref{t.toeplitzrelations.c}.

Conversely, let $\{T_\alpha : \alpha \in \Lambda \}$ be given satisfying \eqref{t.toeplitzrelations.a}, \eqref{t.toeplitzrelations.b} and \eqref{t.toeplitzrelations.c} of the theorem.  For $\zeta = (\alpha_1, \beta_1, \ldots, \alpha_n, \beta_n) \allowbreak \in \CZ$ define $T_\zeta = T_{\alpha_1}^* T_{\beta_1} \cdots T_{\alpha_n}^* T_{\beta_n}$.  Then Theorem \ref{t.gensandrels}\eqref{t.gensandrels.a} and \eqref{t.gensandrels.b} clearly hold.  We will verify Theorem \ref{t.gensandrels}\eqref{t.gensandrels.c}.  

We first prove the following claim.  If $\gamma_i$, $\delta_i$, $\xi_j$, $\eta_j \in \Lambda$ are finite collections such that $\bigcup_i \gamma_i \sigma^{\delta_i} = \bigcup_j \xi_j \sigma^{\eta_j}$, then $\bigvee_i T_{\gamma_i} T_{\delta_i}^* = \bigvee_j T_{\xi_j} T_{\eta_j}^*$.  To prove this claim, first fix $i_0$.  Since $\delta_{i_0}$ is in the domain of $\bigcup_i \gamma_i \sigma^{\delta_i}$, there exists $j_0$ such that $\eta_{j_0} \in [\delta_{i_0}]$.  Similarly, there is $i_1$ such that $\delta_{i_1} \in [\eta_{j_0}]$.  Therefore $\delta_{i_1} \in [\delta_{i_0}]$.  Let $\delta_{i_0} = \delta_{i_1} \mu$.  Since any two terms of $\bigcup_i \gamma_i \sigma^{\delta_i}$ must agree on the intersection of their domains, we have 
\[
\gamma_{i_0} = \gamma_{i_0} \sigma^{\delta_{i_0}}(\delta_{i_0}) = \gamma_{i_1} \sigma^{\delta_{i_1}}(\delta_{i_0}) = \gamma_{i_1} \mu.
\]
Therefore $\gamma_{i_0} \sigma^{\delta_{i_0}} = \gamma_{i_1} \sigma^{\delta_{i_1}} |_{\mu \Lambda}$.  Thus the $i_0$ term may be deleted from $\bigcup_i \gamma_i \sigma^{\delta_i}$.  We repeat this process until we have that $\delta_i \not\in [\delta_{i'}]$ for all $i \not= i'$.  Moreover, we have $T_{\gamma_{i_1}} T_{\delta_{i_1}}^* = T_{\gamma_{i_0}} T_{\delta_{i_0}}^* + T_{\gamma_{i_1}} (T_{s(\gamma_{i_1})} + T_\mu T_\mu^*) T_{\delta_{i_1}}^*$.  Therefore $T_{\gamma_{i_0}} T_{\delta_{i_0}}^*$ can be deleted from $\bigvee_i T_{\gamma_i} T_{\delta_i}^*$.  Repeating this for the other map and operator, we may also assume that $\eta_j \not\in [\eta_{j'}]$ for all $j \not= j'$.  Now for each $i$ there is $j$ such that $\eta_j \in [\delta_i]$.  Then there is $i'$ such that $\delta_{i'} \in [\eta_j]$.  Hence $\delta_{i'} \in [\delta_i]$, so we must have $\delta_i = \delta_{i'} = \eta_j$.  Applying both maps to $\delta_i = \eta_j$ we find that $\gamma_i = \xi_j$.  Thus the two presentations of the map are identical, and thus so are the operators.  This finishes the proof of the claim.  

Next we claim that if $\zeta \in \CZ$ and $\pphi_\zeta = \bigcup_i \gamma_i \sigma^{\delta_i}$ is a finite union, then $T_\zeta = \bigvee_i T_{\gamma_i} T_{\delta_i}^*$.  We prove this by induction on the length of $\zeta$.  First suppose that $\zeta = (\alpha,\beta)$.  Then $\sigma^\alpha \beta = \bigcup_{\eps \in \alpha \vee \beta} (\sigma^\alpha \eps) \sigma^{(\sigma^\beta \eps)}$.  Moreover,
\[
T_\alpha^* T_\beta = T_\alpha^* (T_\alpha T_\alpha^* T_\beta T_\beta^*) T_\beta = \bigvee_{\eps \in \alpha \vee \beta} T_\alpha^* T_\eps T_\eps^* T_\beta = \bigvee_{\eps \in \alpha \vee \beta} T_{\sigma^\alpha \eps} T_{\sigma^\beta \eps}^*.
\]
By the previous claim, we know that this doesn't depend on the decomposition chosen for $\pphi_\zeta$.  Now suppose that the current claim is true for zigzags of length at most $n$.  Let $\zeta = (\alpha_1, \beta_1, \ldots, \alpha_{n+1}, \beta_{n+1})$.  Let $\zeta_0 = (\alpha_1, \beta_1, \ldots, \alpha_n, \beta_n)$.  Write $\pphi_{\zeta_0} = \bigcup_i \gamma_i \sigma^{\delta_i}$ and $\pphi_{(\alpha_{n+1},\beta_{n+1})} = \bigcup_j \mu_j \sigma^{\nu_j}$.  Then
\[
\pphi_\zeta = \pphi_{\zeta_0} \circ \pphi_{(\alpha_{n+1},\beta_{n+1})} = \bigcup_{i,j} \gamma_i \sigma^{\delta_i} \mu_j \sigma^{\nu_j} = \bigcup_{i,j,k} \gamma_i \xi_k \sigma^{\eta_k} \sigma^{\nu_j} = \bigcup_{i,j,k} \gamma_i\xi_k \sigma^{\nu_j\eta_k}.
\]
Then the inductive hypothesis gives
\[
T_\zeta = T_{\zeta_0} T_{(\alpha_{n+1},\beta_{n+1})} = \bigvee_{i,j} T_{\gamma_i} T_{\delta_i}^* T_{\mu_j} T_{\nu_j}^* = \bigvee_{i,j,k} T_{\gamma_i} T_{\xi_k} T_{\eta_k}^* T_{\nu_j}^* = \bigvee_{i,j,k} T_{\gamma_i \xi_k} T_{\nu_j \eta_k}^*.
\]
Again, the first claim shows that this is independent of the choice of decomposition of $\pphi_\zeta$.
It is clear that Theorem \ref{t.gensandrels}\eqref{t.gensandrels.c} follows from the last claim.
Finally, Theorem \ref{t.gensandrels}\eqref{t.gensandrels.d} also follows from the last claim.
\end{proof}

\begin{Remark}
\label{r.toeplitzrelations}
Because of the importance of the finitely aligned case, we formally present some simple consequences of the relations in Theorem \ref{t.toeplitzrelations} (some of these were used in the proof of the theorem).  Let $\{T_\alpha : \alpha \in \Lambda\}$ be as in the statement of Theorem \ref{t.toeplitzrelations}.
\begin{enumerate}
\item From Theorem \ref{t.toeplitzrelations}\eqref{t.toeplitzrelations.a} we see that $T_u$ is a self-adjoint projection when $u \in \Lambda^0$.  Hence $T_\alpha$ is a partial isometry.
\label{r.toeplitzrelations.one}
\item From Theorem \ref{t.toeplitzrelations}\eqref{t.toeplitzrelations.b}, letting $\alpha = r(\beta)$, we see that $T_\beta T_\beta^* \le T_{r(\beta)}$.
\label{r.toeplitzrelations.two}
\item From Theorem \ref{t.toeplitzrelations}\eqref{t.toeplitzrelations.c} we see that $T_u T_v = 0$ if $u \not= v$, for $u$, $v \in \Lambda^0$.  It follows that $T_\alpha T_\beta =0$ if $s(\alpha) \not= r(\beta)$.
\label{r.toeplitzrelations.three}
\item From Theorem \ref{t.toeplitzrelations}\eqref{t.toeplitzrelations.c} (or its proof),  we see that $T_\alpha^* T_\beta = \bigvee_{\eps \in \alpha \vee \beta} T_{\sigma^\alpha \eps} T_{\sigma^\beta \eps}^*$.  Thus if the elements of $\alpha \vee \beta$ are pairwise disjoint, for example, in the case of a higher-rank graph, we recover the formula (iii), Definition 2.5, of \cite{raesimyee}.
\label{r.toeplitzrelations.four}
\item From Theorem \ref{t.toeplitzrelations}\eqref{t.toeplitzrelations.c} we see that $\{T_\alpha T_\alpha^* : \alpha \in \Lambda\}$ is commutative.
\label{r.toeplitzrelations.five}
\end{enumerate}
\end{Remark}

We end this section with some elementary algebraic consequences of Theorem \ref{t.toeplitzrelations}.  Because distinct minimal common extensions need not be disjoint in general, the monomials of the form $T_\alpha T_\beta^*$ do not necessarily span a dense $*$-subalgebra of the Toeplitz algebra.  However, there is a total set that is not too far from such monomials.

In the following three results, we assume that $\Lambda$ is a finitely aligned category of paths.  Let $P_0 = \{ T_\alpha T_\alpha^* : \alpha \in \Lambda \}$, and $P = \{ p_1 \cdots p_k : p_i \in P_0,\ 1 \le i \le k,\ k \in \IN \}$.  Let $B = \text{span}\, \{ T_\mu T_\nu^* q : \mu, \nu \in \Lambda,\ q \in P \}$.

\begin{Lemma}
\label{l.totalsetone}
For $\alpha$, $\beta \in \Lambda$, $T_\alpha^* T_\beta \in B$.
\end{Lemma}

\begin{proof}
Let $\alpha \vee \beta = \{ \eps_1, \ldots, \eps_k \}$.  For each $i$, write $\eps_i = \alpha \mu_i = \beta \nu_i$.  Let $p_i = T_{\nu_i} T_{\nu_i}^* \in P_0$.  Then
\begin{align*}
T_\alpha^* T_\beta
&= \bigvee_{i=1}^k T_{\mu_i} T_{\nu_i}^* \\
&= T_{\mu_1} T_{\nu_1}^* + T_{\mu_2} T_{\nu_2}^* (1 - p_1) + T_{\mu_3} T_{\nu_3}^* (1 - p_1 - p_2 + p_1 p_2) \\
& \qquad + \cdots + T_{\mu_k} T_{\nu_k}^* \left( 1 - \sum_{i_1 < k} p_{i_1} + \sum_{i_1 < i_2 < k} p_{i_1} p_{i_2} - \cdots + (-1)^{k - 1} p_1 p_2 \cdots p_{k-1} \right) \\
&\in B.
\end{align*}
\end{proof}

\begin{Lemma}
\label{l.totalsettwo}
Let $\alpha \in \Lambda$ and $q \in P$.  Then there is $q' \in P$ such that $q T_\alpha^* = T_\alpha^* q'$.
\end{Lemma}

\begin{proof}
Note that if $p = T_\gamma T_\gamma^* \in P_0$, then $T_\alpha p = T_{\alpha \gamma} T_{\alpha \gamma}^* T_\alpha$.  Iterating this produces $q' \in P$ such that $T_\alpha q = q' T_\alpha$.  The lemma follows by taking adjoints.
\end{proof}

\begin{Proposition}
\label{p.totalset}
$B$ equals the $*$-algebra generated by $\{T_\alpha : \alpha \in \Lambda\}$.
\end{Proposition}

\begin{proof}
It is clear that $B$ is contained in the $*$-algebra.  The $*$-algebra is spanned by monomials of the form $T_{\alpha_1} T_{\beta_1}^* \cdots T_{\alpha_k} T_{\beta_k}^*$, so it suffices to prove that such monomials are in $B$.  This is clear for monomials of length one.  We consider $T_{\alpha_1} T_{\beta_1}^* T_{\alpha_2} T_{\beta_2}^*$.  By Lemma \ref{l.totalsetone} we know that $T_{\beta_1}^* T_{\alpha_2} \in B$.  So it is enough to consider $T_{\alpha_1} ( T_\mu T_\nu^* q) T_{\beta_2}^*$, with $q \in P$.  By Lemma \ref{l.totalsettwo} we know that there is $q' \in P$ such that $q T_{\beta_2}^* = T_{\beta_2}^* q'$.  Thus $T_{\alpha_1} ( T_\mu T_\nu^* q) T_{\beta_2}^* = T_{\alpha_1 \mu} T_{\beta_2 \nu}^* q' \in B$.  Finally, assume inductively that the result holds for monomials of length less than $k$.  Given $T_{\alpha_1} T_{\beta_1}^* \cdots T_{\alpha_k} T_{\beta_k}^*$, we then know that $T_{\alpha_2} T_{\beta_2}^* \cdots T_{\alpha_k} T_{\beta_k}^* \in B$.  Therefore it suffices to show that $(T_{\alpha_1} T_{\beta_1}^*) (T_\mu T_\nu^* q) \in B$.  This follows from the case of length two, which was already proved.
\end{proof}

\section{The boundary of a finitely aligned category of paths}
\label{s.boundary}

Our next task is to identify the ultrafilters in the ring $\CA_v$, for $v \in \Lambda^0$.  We have not carried this out in the general case, so for the rest of the paper (except for the first part of section \ref{s.amalgamation}), we assume that all categories of paths are finitely aligned.

\begin{Definition}
\label{d.directedset}
Let $\Lambda$ be a category of paths.  A subset $C \subseteq \Lambda$ is \textit{directed} if the partial order it inherits from $\Lambda$ (Definition \ref{d.extensions}) is directed:  for all $\alpha$, $\beta \in C$, $\alpha\Lambda \cap \beta\Lambda \cap C \not= \emptyset$.  $C$ is \textit{hereditary} if $[\gamma] \subseteq C$ whenever $\gamma \in C$ (recall from Definition \ref{d.extensions} that $[\gamma]$ is the set of initial segments of $\gamma$).  A directed set $C$ is \textit{finite} if it contains a maximal element; otherwise it is \textit{infinite}.
\end{Definition}

\begin{Remark}
\label{r.directedset}
We note that if $C$ is a directed set, then $\widetilde{C} = \bigcup_{\alpha \in C} [\alpha]$ is directed and hereditary and contains $C$.  We also note that the set of all directed subsets of $\Lambda$, partially ordered by inclusion, satisfies the hypotheses of Zorn's lemma, so that every directed set is contained in a maximal directed set.  Any maximal directed set is necessarily hereditary.
\end{Remark} 

For the next few results we assume that $\Lambda$ is a countable category.  We do not know if this is necessary (other than for our proofs).

\begin{Lemma}
\label{l.directedsets}
Let $\Lambda$ be a countable finitely aligned category of paths, and let $C$ be a directed subset of $\Lambda$.  Suppose that $\beta \in \Lambda$ is such that $\beta \Cap \alpha$ for all $\alpha \in C$.  Then there exists a directed set $\widetilde{C}$ containing both $C$ and $\beta$. 
\end{Lemma}

\begin{proof}
Fix $\alpha \in C$.  We claim that there exists $\eps_\alpha \in \beta\Lambda \cap \alpha\Lambda$ such that $\eps_\alpha \Cap \gamma$ for all $\gamma \in C \cap \alpha\Lambda$.  For suppose not.  By finite alignment, $\beta\Lambda \cap \alpha\Lambda = \bigcup_{\eps \in \beta \vee \alpha} \eps\Lambda$.  The assumption means that for all $\eps \in \beta \vee \alpha$, there exists $\gamma_\eps \in C \cap \alpha\Lambda$ such that $\gamma_\eps \Lambda \cap \eps\Lambda = \emptyset$.  Since $C$ is directed there exits $\gamma \in C \cap \bigcap_{\eps \in \beta \vee \alpha} \gamma_\eps \Lambda$.  We have $\beta \Cap \gamma$, so let $\eta \in \beta\Lambda \cap \gamma\Lambda$.  Since $\eta \in \alpha\Lambda$ we have $\eta \in \beta\Lambda \cap \alpha\Lambda$, so there is $\eps \in \beta \vee \alpha$ with $\eta \in \eps\Lambda$.  Since $\eta \in \gamma\Lambda \subseteq \gamma_\eps\Lambda$, we have $\eta \in \gamma_\eps \Lambda \cap \eps\Lambda$, a contradiction.

Now let $C = \{ \xi_1,\ \xi_2,\ \ldots \}$.  Let $\alpha_1 = \xi_1$, and choose $\eps_{\alpha_1}$ as in the previous claim.  Let $\alpha_2 \in \alpha_1\Lambda \cap \xi_2\Lambda \cap C$.  Note that since $\alpha_1$, $\alpha_2 \in C$ we have $\alpha_1 \Lambda \cap \alpha_2 \Lambda \cap C \not= \emptyset$.  If $\gamma \in \alpha_1 \Lambda \cap \alpha_2 \Lambda \cap C$ then $\eps_{\alpha_1} \Cap \gamma$, and hence $\eps_{\alpha_1} \Cap \alpha_2$.  We now claim that there exists $\eps_{\alpha_2}$ as in the previous claim such that $\eps_{\alpha_2} \in \eps_{\alpha_1}\Lambda$.  Again, suppose not.  Then since $\eps_{\alpha_1} \in \beta \Lambda$, we have that for each $\delta \in \alpha_2\Lambda \cap \eps_{\alpha_1}\Lambda$, there exists $\eta_\delta \in C \cap \alpha_2\Lambda$ such that $\delta \perp \eta_\delta$.  Let $\eta \in C \cap \bigcap_{\delta \in \alpha_2 \vee \eps_{\alpha_1}} \eta_\delta \Lambda$, which is possible because $C$ is directed.  Now we have
\[
\eta\Lambda \cap \eps_{\alpha_1}\Lambda = \eta\Lambda \cap \alpha_2\Lambda \cap \eps_{\alpha_1}\Lambda = \bigcup_{\delta \in \alpha_2 \vee \eps_{\alpha_1}} \eta \Lambda \cap \delta \Lambda \subseteq \bigcup_{\delta \in \alpha_2 \vee \eps_{\alpha_1}} \eta_\delta \Lambda \cap \delta \Lambda = \emptyset.
\]
But since $\eta$, $\alpha_1 \in C$, there is $\eta' \in C \cap \eta \Lambda \cap \alpha_1 \Lambda$.  But then $\eps_{\alpha_1} \Cap \eta'$, so $\eps_{\alpha_1} \Cap \eta$, a contradiction, and the current claim is proved.

Inductively, we let $\alpha_i \in \alpha_{i-1}\Lambda \cap \xi_i\Lambda \cap C$, and choose $\eps_{\alpha_i}$ (as in the first claim above) so that $\eps_{\alpha_i} \in \eps_{\alpha_{i-1}}\Lambda$.  Then $\widetilde{C} = C \cup \{ \eps_{\alpha_i} : i \in \IN \} \cup \{\beta\}$ is directed.  
\end{proof}

The definitions of filters and ultrafilters were recalled at the beginning of section \ref{s.groupoid}.

\begin{Definition}
\label{d.ultrafilters}
Let $\Lambda$ be a finitely aligned category of paths, and let $v \in \Lambda^0$.  For a hereditary directed set $C \subseteq v\Lambda$, let
\begin{align*}
\CU_{C,0}
&= \{ E \in \CA_v : E \supseteq \alpha\Lambda \text{ for some } \alpha \in C \} \\
\CU_C
&= \{ E \in \CA_v : E \supseteq C \cap \alpha\Lambda \text{ for some } \alpha \in C \}.
\end{align*}
\end{Definition}

It is clear that $\CU_{C,0}$ and $\CU_C$ are filters in $\CA_v$ and that $\CU_{C,0} \subseteq \CU_C$.

\begin{Definition}
Let $\Lambda$ be a finitely aligned category of paths, and let $v \in \Lambda^0$.  We let $v\Lambda^*$ denote the collection of all hereditary directed subsets of $v\Lambda$, and $v\Lambda^{**}$ the collection of all maximal directed subsets of $v\Lambda$.  We write $\Lambda^* = \cup_{v \in \Lambda^0} v\Lambda^*$, and $\Lambda^{**} = \cup_{v \in \Lambda^0} v\Lambda^{**}$.
\end{Definition}

\begin{Theorem}
\label{t.ultrafilters}
Let $\Lambda$ be a countable finitely aligned category of paths, let $v \in \Lambda^0$, and let $C \in v\Lambda^*$.
\begin{enumerate}
\item\label{t.ultrafilters.a}
 $\CU_{C,0}$ is an ultrafilter if and only if $C \in v\Lambda^{**}$.
\item\label{t.ultrafilters.b}
 $\CU_C$ is an ultrafilter, fixed if and only if $C$ is finite.
\item\label{t.ultrafilters.c}
 Every ultrafilter in $\CA_v$ is of the form $\CU_C$ for a unique $C \in v\Lambda^*$.
\end{enumerate}
\end{Theorem}

\begin{proof}
\eqref{t.ultrafilters.a} ($\Rightarrow$) Let $C$, $C' \in \Lambda^*$ with $C \subsetneq C'$.  Let $\beta \in C' \setminus C$.  If $\beta\Lambda \in \CU_{C,0}$, then there exists $\alpha \in C$ such that $\beta\Lambda \supseteq \alpha\Lambda$.  But then $\alpha \in \beta\Lambda$, so that $\beta \in C$, since $C$ is hereditary.  From this contradiction we see that $\CU_{C,0} \subsetneq \CU_{C',0}$, and hence $\CU_{C,0}$ is not an ultrafilter.

\vspace*{.1 in}

\noindent
($\Leftarrow$) Let $C \in v\Lambda^{**}$.  Let $E \in \CE_v \setminus \CU_{C,0}$.  Write $E = \beta\Lambda \setminus \cup_{i=1}^n \gamma_i\Lambda$ with $\gamma_i \in \beta\Lambda$.  We must find a set in $\CU_{C,0}$ disjoint from $E$.  If there exists $\alpha \in C$ with $\alpha \perp \beta$, then $\alpha\Lambda$ is such a set.  So let us suppose that $\alpha \Cap \beta$ for all $\alpha \in C$.  Then since $C$ is maximal, Lemma \ref{l.directedsets} implies that $\beta \in C$.  Now suppose that $\gamma_i \not\in C$ for all $i$.  Again by the maximality of $C$, and Lemma \ref{l.directedsets}, for each $i$ there is $\alpha_i \in C$ with $\alpha_i \perp \gamma_i$.  Let $\alpha \in C \cap \bigcap_{i=1}^n \alpha_i\Lambda \cap \beta\Lambda$.  Then $\alpha \perp \gamma_i$ for all $i$, so $\alpha\Lambda \subseteq E$, contradicting the assumption that $E \not\in \CU_{C,0}$.  Therefore there exists $i_0$ such that $\gamma_{i_0} \in C$.  Then $\gamma_{i_0}\Lambda \in \CU_{C,0}$, and $\gamma_{i_0}\Lambda \cap E = \emptyset$.  Finally, if $E \in \CA_v$ with $E \not\in \CU_{C,0}$, then $E = \sqcup_{j=1}^p E_j$, where $E_j \in \CE_v$.  Since $\CU_{C,0}$ is a filter, $E_j \not\in \CU_{C,0}$ for all $j$.  By the above, there is $F_j \in \CU_{C,0}$ with $E_j \cap F_j = \emptyset$.  Then $F = \cap_j F_j \in \CU_{C,0}$, and $E \cap F = \emptyset$. 

\vspace*{.1 in}

\noindent
\eqref{t.ultrafilters.b}  As in the proof of part \eqref{t.ultrafilters.a}, it suffices to consider $E \in \CE_v$ with $E \not \in \CU_C$.  Write $E = \beta\Lambda \setminus \cup_{i=1}^n \gamma_i\Lambda$ with $\gamma_i \in \beta\Lambda$.  If $C \cap \beta\Lambda \not= \emptyset$, let $\alpha \in C \cap \beta\Lambda$.  Then $\alpha\Lambda \subseteq \beta\Lambda$.  Since $E \not\in \CU_C$, there is $\alpha' \in C \cap \alpha\Lambda$ such that $\alpha' \not\in E$.  Then there is $i$ such that $\alpha' \in \gamma_i\Lambda$.  Then $\alpha'\Lambda \subseteq \gamma_i\Lambda \subseteq E^c$, and $\alpha'\Lambda \in \CU_C$.  Now suppose that $C \cap \beta\Lambda = \emptyset$.  Then $v\Lambda \setminus \beta\Lambda \in \CU_C$, and $(v\Lambda \setminus \beta\Lambda) \cap E = \emptyset$. 

If $C$ has a maximal element $\alpha_0$, then $\{ \alpha_0 \} \in \CU_C$, so that $\CU_C$ is fixed.  Conversely, if $\CU_C$ is fixed there is $\alpha_0 \in v\Lambda$ such that $\{ \alpha_0 \} \in \CU_C$.  From the definition of $\CU_C$ it follows that $\alpha_0 \in C$, and must be the maximal element of $C$. 

\vspace*{.1 in}

\noindent
\eqref{t.ultrafilters.c} For the uniqueness, let $C$, $C'$ be distinct elements of $v\Lambda^*$.  Then, say, there exists $\alpha \in C \setminus C'$.  Since $C'$ is hereditary, $C' \cap \alpha\Lambda = \emptyset$.  But then $\alpha\Lambda \in \CU_C \setminus \CU_{C'}$.

Now let $\CU$ be an ultrafilter in $\CA_v$.  Let 
\[
\CS = \{ C \in \Lambda^* : \CU_{C,0} \subseteq \CU \}.
\]
$\CS$ is partially ordered by inclusion, and satisfies the hypotheses of Zorn's lemma (note that $\CS \not= \emptyset$ since $\{v\} \in \CS$).  Let $C$ be a maximal element of $\CS$.  Then $C$ is hereditary.  (Of course, $C$ need not be maximal as a directed set.)  Suppose $\CU \not= \CU_C$.  Let $E \in \CU \setminus \CU_C$.  We write $E = \sqcup_{j=1}^p E_j$ with $E_j \in \CE_v$.  Since $\CU$ is an ultrafilter, one of the $E_j$ is in $\CU$.  Since $\CU_C$ is a filter, none of the $E_j$ is in $\CU_C$.  Thus we may assume that $E \in \CE_v$.  Write $E = \beta\Lambda \setminus \cup_{i=1}^n \gamma_i\Lambda$ with $\gamma_i \in \beta\Lambda$.  Since $E \in \CU$ we have that $\beta\Lambda \in \CU$.  Then $\beta\Lambda \cap \alpha\Lambda \in \CU$, and hence $\beta \Cap \alpha$, for all $\alpha \in C$.  We give a variant of the argument in Lemma \ref{l.directedsets}.  For each $\alpha \in C$,
\[
\bigcup_{\eps \in \beta \vee \alpha} \eps\Lambda
= \beta\Lambda \cap \alpha\Lambda
\in \CU.
\]
Thus there exists $\eps_\alpha \in \beta\Lambda \cap \alpha\Lambda$ with $\eps_\alpha \Lambda \in \CU$.  

Now let $C = \{ \xi_1,\ \xi_2,\ \ldots \}$.  Let $\alpha_1 = \xi_1$, and choose $\eps_{\alpha_1} \in \beta\Lambda \cap \alpha_1\Lambda$ with $\eps_{\alpha_1} \Lambda \in \CU$.  Let $\alpha_2 \in \alpha_1\Lambda \cap \xi_2\Lambda \cap C$.  We claim there exists $\eps_{\alpha_2} \in \eps_{\alpha_1} \Lambda \cap \alpha_2\Lambda$ such that $\eps_{\alpha_2}\Lambda \in \CU$.  For suppose not.  Then for each $\delta \in \alpha_2\Lambda \cap \eps_{\alpha_1}\Lambda$, we have $\delta\Lambda \not\in \CU$.  By finite alignment we have
\[
\alpha_2\Lambda \cap \eps_{\alpha_1}\Lambda
= \bigcup_{\delta \in \alpha_2 \vee \eps_{\alpha_1}} \delta\Lambda 
 \not \in \CU,
\]
since $\CU$ is an ultrafilter, contradicting the fact that $\alpha_2\Lambda$, $\eps_{\alpha_1}\Lambda \in \CU$.  This verifies the claim.  Inductively, there are $\alpha_1$, $\alpha_2$, $\ldots \in C$, and $\eps_{\alpha_i} \in \beta\Lambda \cap \alpha_i\Lambda$, such that $\{\alpha_i\}$ is cofinal in $C$ (in the sense of a directed set), $\eps_{\alpha_i}\Lambda \in \CU$ for all $i$, and $\eps_{\alpha_{i+1}} \in \eps_{\alpha_i}\Lambda$ for all $i$.  Then $C' = C \cup \{\beta\} \cup \{ \eps_\alpha : \alpha \in C \}$ is a directed set, and $\widetilde{C'} \in \CS$ (see Remark \ref{r.directedset}).  By maximality, $\widetilde{C'} = C$, so that $\beta \in C$.  Since $E \not\in \CU_C$ there is $\eta \in C \cap \beta\Lambda$ with $\eta \not\in E$.  Then $\eta \in \gamma_i\Lambda$ for some $i$, and hence $\eta\Lambda \subseteq \gamma_i\Lambda \subseteq E^c$.  But then $\emptyset = \eta\Lambda \cap E \in \CU$, a contradiction.  Therefore $\CU = \CU_C$.
\end{proof}

By this result we may identify $X_v = v\Lambda^*$.  In the next theorem we describe the closure (in $X_v$) of $v\Lambda^{**}$. The next definition was discovered in \cite{raesimyee}.

\begin{Definition}
\label{d.exhaustiveset}
Let $\Lambda$ be a category of paths, and $v \in \Lambda^0$.  A subset $F \subseteq v\Lambda$ is \textit{exhaustive (at $v$)} if for every $\alpha \in v\Lambda$ there exists $\beta \in F$ with $\alpha \Cap \beta$.  $F$ is \textit{trivially exhaustive} if $v \in F$.  We let $\text{FE}(v)$ denote the collection of finite exhaustive sets at $v$.
\end{Definition}

\begin{Theorem}
\label{t.boundary}
Let $\Lambda$ be a countable finitely aligned category of paths, and let $C \in v \Lambda^*$.  Then $C$ is in the closure of $v \Lambda^{**}$ if and only if the following condition holds:
for every $\alpha \in C$, there exists $\alpha' \in C \cap \alpha\Lambda$ such that for every finite exhaustive set $F \subseteq s(\alpha')\Lambda$, we have $\sigma^{\alpha'}C \cap F \not= \emptyset$.
\end{Theorem}

\begin{proof}
($\Leftarrow$) Let $E \in \CA_v$ with $C \in \widehat{E}$; i.e. $E \in \CU_C$.  Writing $E$ as a disjoint union from $\CE_v$, only one of these sets is in $\CU_C$.  Thus we may assume $E \in \CE_v$.  Write $E = \beta\Lambda \setminus \bigcup_{i=1}^n \gamma_i\Lambda$, with $\gamma_i \in \beta\Lambda$.  Let $\alpha \in C$ with $C \cap \alpha\Lambda \subseteq E$.  By hypothesis there is $\alpha' \in C \cap \alpha\Lambda$ such that for every $F \in \text{FE}\bigl(s(\alpha')\bigr)$, we have $\sigma^{\alpha'}C \cap F \not= \emptyset$.  Since $C \cap \alpha'\Lambda \subseteq E$, we may replace $E$ by $\alpha'\Lambda \setminus \bigcup_{i=1}^n \gamma_i\Lambda$, and we may rewrite it so that $\gamma_i \in \alpha'\Lambda$.

Now, for each $i$, $\gamma_i\Lambda \subseteq E^c$, so $\gamma_i \not\in C \cap \alpha'\Lambda$, hence $\sigma^{\alpha'}\gamma_i \not\in \sigma^{\alpha'}C$.  By hypothesis, $\{ \sigma^{\alpha'}\gamma_1, \ldots, \sigma^{\alpha'}\gamma_n \}$ is not exhaustive.  Hence there is $\delta \in s(\alpha')\Lambda$ such that $\delta \perp \sigma^{\alpha'}\gamma_i$ for all $i$.  Then $\alpha'\delta \perp \gamma_i$ for all $i$.  Let $C' \in v\Lambda^{**}$ with $\alpha'\delta \in C'$.  Then $\gamma_i\Lambda \not\in \CU_{C'}$ for all $i$.  But $\alpha'\Lambda \in \CU_{C'}$, so $E \in \CU_{C'}$; i.e. $C' \in \widehat{E}$.

We present a slight sharpening of the forward direction as a sublemma.

\begin{Lemma}
\label{l.boundary}
Suppose that the condition in the statement of Theorem \ref{t.boundary} fails.  There are $\alpha_0 \in C$ and $\Delta \in \text{FE}\bigl(s(\alpha_0)\bigr)$ such that if we let $E = \alpha_0\Lambda \setminus \bigcup_{\delta \in \Delta} \alpha_0\delta\Lambda$, then $C \in \widehat{E} \subseteq v\Lambda^* \setminus v\Lambda^{**}$.
\end{Lemma}

\begin{proof}
The failure of the condition implies in particular (with $\alpha' = \alpha$) that there are $\alpha_0 \in C$ and $\Delta \in \text{FE}\bigl(s(\alpha_0)\bigr)$ such that $\Delta \cap \sigma^{\alpha_0}C = \emptyset$.  Then $\alpha_0\Delta \cap C \cap \alpha_0\Lambda = \emptyset$.  Let $E = \alpha_0\Lambda \setminus \bigcup_{\delta \in \Delta} \alpha_0\delta\Lambda$.  Then $C \cap \alpha_0\Lambda \subseteq E$, hence $E \in \CU_C$, hence $C \in \widehat{E}$.  We claim that $\widehat{E} \cap v\Lambda^{**} = \emptyset$, which will finish the proof.  For suppose that $C' \in \widehat{E} \cap v\Lambda^{**}$.  Then $E \in \CU_{C'}$, so there is $\xi \in C'$ such that $C' \cap \xi\Lambda \subseteq E$.  Then $\xi \in E$, and hence $\xi \in \alpha_0\Lambda$.  We claim that there exists $\delta_0 \in \Delta$ such that $\sigma^{\alpha_0}\eta \Cap \delta_0$ for all $\eta \in C' \cap \xi\Lambda$.  For if not, then for each $\delta \in \Delta$ there is $\eta_\delta \in C' \cap \xi\Lambda$ with $\sigma^{\alpha_0} \eta_\delta \perp \delta$.  Let $\eta \in C' \cap \bigcap_{\delta \in \Delta} \eta_\delta\Lambda$.  Then $\eta \in C' \cap \xi\Lambda$, and $\sigma^{\alpha_0}\eta \in \sigma^{\alpha_0}\eta_\delta\Lambda$ for all $\delta$.  Therefore $\sigma^{\alpha_0}\eta \perp \delta$ for all $\delta \in \Delta$, contradicting exhaustiveness of $\Delta$.  Thus such a $\delta_0$ exists.  But then, $\eta \Cap \alpha_0\delta_0$ for all $\eta \in C'$.  By Lemma \ref{l.directedsets} we have $\alpha_0\delta_0 \in C'$, since $C'$ is maximal.  Then $C' \cap \alpha_0\delta_0\Lambda \cap \xi\Lambda \not= \emptyset$.  But $(C' \cap \xi\Lambda) \cap \alpha_0\delta_0\Lambda \subseteq E \cap E^c = \emptyset$, a contradiction.
\end{proof}

\noindent
($\Rightarrow$) (of Theorem \ref{t.boundary})  This follows immediately from Lemma \ref{l.boundary}.
\end{proof}

\begin{Definition}
\label{d.boundary}
Let $\Lambda$ be a countable finitely aligned category of paths.  The \textit{boundary} of $\Lambda$, denoted $\partial\Lambda$, is the closure in $X$ of the set $\Lambda^{**}$ of maximal directed sets.  We will write $v\partial \Lambda$ for $\partial\Lambda \cap v\Lambda^*$.
\end{Definition}

\section{Restriction to the boundary}
\label{s.restriction}

\begin{Lemma}
\label{l.finiteexhaustive}
Let $\Lambda$ be a countable finitely aligned category of paths, $v \in \Lambda^0$, and $\Delta \subseteq v\Lambda$ a finite set.  Then $\Delta$ is exhaustive if and only if
\[
v\partial \Lambda \subseteq \bigcup_{\alpha \in \Delta} \alpha\Lambda^*.
\]
\end{Lemma}

\begin{proof}
First suppose that $v \partial \Lambda \not\subseteq \bigcup_{\alpha \in \Delta} \alpha \Lambda^*$.  Since each $\alpha \Lambda^*$ is a clopen set, there is $C \in v \Lambda^{**} \setminus \bigcup_{\alpha \in \Delta} \alpha \Lambda^*$.  Then $C \cap \Delta = \emptyset$.  Since $C$ is a maximal directed hereditary set, Lemma \ref{l.directedsets} implies that for each $\alpha \in \Delta$ there is $\eta_\alpha \in C$ such that $\eta_\alpha \perp \alpha$.  Let $\eta \in C \cap \bigcap_{\alpha \in \Delta} \eta_\alpha \Lambda$.  Then $\eta \perp \Delta$.  Therefore $\Delta$ is not exhaustive.

Now suppose that $\Delta$ is not exhaustive.  Then there is $\beta \in v \Lambda$ such that $\beta \perp \Delta$.  Let $C \in v \partial \Lambda$ with $\beta \in C$.  We claim that $C \not\in \bigcup_{\alpha \in \Delta} \alpha \partial \Lambda$.  For otherwise, there is $\alpha \in \Delta$ such that $C \in \alpha \partial \Lambda$.  Then $\alpha \in C$.  Since $\beta \in C$, it follows that $\alpha \Cap \beta$, a contradiction.
\end{proof}

\begin{Theorem}
\label{t.restrictionboundary}
Let $\Lambda$ be a countable finitely aligned category of paths.  Let $G = G(\Lambda)$, and assume that $G$ is amenable.  The representations of $C^*(G|_{\partial\Lambda})$ are in one-to-one correspondence with the families $\{S_\alpha : \alpha \in \Lambda\}$ of Hilbert space operators satisfying the relations
\begin{enumerate}
\item $S_\alpha^* S_\alpha = S_{s(\alpha)}$.
\label{t.restrictionboundary.a}
\item $S_\alpha S_\beta = S_{\alpha\beta}$, if $s(\alpha) = r(\beta)$.
\label{t.restrictionboundary.b}
\item $S_\alpha S_\alpha^* S_\beta S_\beta^*
= \bigvee_{\gamma \in \alpha \vee \beta} S_\gamma S_\gamma^*$.
\label{t.restrictionboundary.c}
\item $S_v = \bigvee_{\beta \in F} S_\beta S_\beta^*$ if $F \in$ FE$(v)$.  (Equivalently, $0 = \prod_{\delta \in F} (S_v - S_\delta S_\delta^*)$.)
\label{t.restrictionboundary.d}
\end{enumerate}
\end{Theorem}

\begin{proof}
Suppose that $\{S_\alpha\}$ satisfy (1) -- (4).  By (1) -- (3), and Theorem \ref{t.toeplitzrelations}, there is a unique representation $\pi : C^*(G) \to B(H)$ with $\pi\bigl(\chi_{[\alpha,s(\alpha),X_{s(\alpha)}]}\bigr) = S_\alpha$.  We claim that $\pi|_{C_0(\partial\Lambda^c)} = 0$.  Let $C \in v\Lambda^* \setminus v\partial\Lambda$.  Let $\alpha_0$, $\Delta$ and $E$ be as in Lemma \ref{l.boundary}.  Then
\[
\pi(\chi_E)
= S_{\alpha_0}S_{\alpha_0}^* - \bigvee_{\delta \in \Delta} S_{\alpha_0\delta}S_{\alpha_0\delta}^* 
= S_{\alpha_0} \left( S_{s(\alpha_0)} - \bigvee_{\delta \in \Delta} S_\delta S_\delta^* \right) S_{\alpha_0}^* 
= 0,
\]
by (4).  Thus $\pi\bigr|_{C_0(\partial\Lambda^c)} = 0$.  By amenability, and \cite{ren} II.4.5, $\pi$ factors through $C^*(G\bigr|_{\partial\Lambda})$.

Conversely, let $\pi$ be a representation of $C^*(G\bigr|_{\partial\Lambda})$.  Let $S_\alpha = \pi\bigl(\chi_{[\alpha,s(\alpha),X_{s(\alpha)}]}\bigr)$.  Composing $\pi$ with the quotient map gives a representation of $C^*(G)$, so (1) -- (3) hold.  We prove (4).  Let $v \in \Lambda^0$ and $F \in \text{FE}(v)$.  Let $E = v\Lambda \setminus \bigcup_{\beta \in F} \beta\Lambda$.  Lemma \ref{l.finiteexhaustive} implies that $\widehat{E} \cap v\partial\Lambda = \emptyset$.  Thus $\pi(\chi_E) = 0$.  But $\pi(\chi_E) = S_v - \bigvee_{\beta \in F} S_\beta S_\beta^*$.
\end{proof}

\begin{Definition}
\label{d.cstaralgebra}
Let $\Lambda$ be a countable finitely aligned category of paths, and let $G = G(\Lambda)$.  We denote by $\CT C^*(\Lambda)$ the $C^*$-algebra $C^*(G)$, and by $C^*(\Lambda)$ the algebra $C^*(G|_{\partial\Lambda})$.
\end{Definition}

\begin{Remark}
If $\Lambda$ is a (finitely aligned) higher-rank graph, then the elements of $\alpha \vee \beta$ are pairwise disjont.  In this case, the join of projections in Theorem \ref{t.restrictionboundary}\eqref{t.restrictionboundary.c} is a sum.  Thus we recover the generators and relations obtained in \cite{raesimyee}.
\end{Remark}

We next describe a family of examples arising from ordered groups.  If this is applied to a free group, with the order determined by the usual presentation, the result is the well-known Morita equivalence between the Cuntz algebra $\CO_n$ and the crossed product algebra associated to the action of the group on the boundary of its directed Cayley graph (this boundary can also be viewed as an unstable equivalence class of the Bernoulli shift; see \cite{cunkri}).  A more complicated example, that of the Baumslag-Solitar groups, is worked out in \cite{spi_baumsol}.

\begin{Example}
\label{e.groupboundary}
Let $Y$ be a countable group, and let $\Lambda$ be a submonoid of $Y$ such that $\Lambda \cap \Lambda^{-1} = \{ e \}$.  Then $(Y,\Lambda)$ is an \textit{ordered group}.  We note that this assumption implies that $\Lambda$ has no inverses.  Since cancellation follows from the group law, it follows that $\Lambda$ is a category of paths.

Nica studied this situation in the special case that he termed a \textit{quasi-lattice order} (\cite{nica}).  We generalize this definition as follows.

\begin{Definition}
\label{d.finitelyalignedgroup}
The ordered group $(Y,\Lambda)$ is \textit{finitely aligned} if for every finite set $D \subseteq Y$ there is a finite set $F \subseteq \bigcap_{t \in D} t\Lambda$ such that $\bigcap_{t \in D} t\Lambda = \bigcup_{u \in F} u\Lambda$.
\end{Definition} 

\begin{Remark}
\label{r.finitelyalignedgroup}
As in \cite{nica}, it is easy to see that finite alignment can be checked two elements at a time:  if it holds whenever $D$ contains two elements, then it holds for all finite $D$.  Moreover, it is enough to verify when one of the two elements is the identity $e$.
\end{Remark}

\begin{Remark}
Nica's definition of quasi-lattice ordered group is the special case of finitely aligned ordered group in which the set $F$ of Definition \ref{d.finitelyalignedgroup} can always be taken to have cardinality one or zero.  Note also that if $(Y,\Lambda)$ is finitely aligned, then in particular, $\Lambda$ is a finitely aligned category of paths.  The converse is not true, even if minimal common extensions in $\Lambda$ are unique (Example \ref{e.groupexamples}).
\end{Remark}

Nica considers the following construction, which he calls the \textit{Wiener-Hopf algebra} of the ordered group.  Let $\pi_\ell$ be the left regular representation of $Y$ on $\ell^2(Y)$, let $H = \ell^2(\Lambda) \subseteq \ell^2(Y)$, and let $P$ be the projection of $\ell^2(Y)$ onto $H$.  For $t \in Y$ let $S_t = P \pi_\ell(t)|_H$.  (Then $S_t \not= 0$ if and only if $t \in \Lambda \Lambda^{-1}$.)  Note that $S_\alpha$ is an isometry if $\alpha \in \Lambda$.  The Wiener-Hopf algebra is defined by $W(Y,\Lambda) = C^* \{ S_t : t \in Y \}$.  Nica shows that $W(Y,\Lambda) = C^* \bigl( \{ S_\alpha : \alpha \in \Lambda \} \bigr)$ if $(Y,\Lambda)$ is quasi-lattice ordered (and hence that the Wiener-Hopf algebra is generated by isometries).  We next show how this idea fits into the finitely aligned situation.

\begin{Lemma}
\label{l.finitelyalignedwienerhopf}
Let $(Y,\Lambda)$ be a ordered group.  Suppose that $\Lambda$ is a finitely aligned category of paths.  Then  $W(Y,\Lambda) = C^* \bigl( \{ S_\alpha : \alpha \in \Lambda \} \bigr)$ if and only if $(Y,\Lambda)$ is finitely aligned.
\end{Lemma}

\begin{proof}
First suppose that $(Y,\Lambda)$ is finitely aligned.  Let $t \in \Gamma$ with $S_t \not= 0$.  Then $t \in \Lambda \Lambda^{-1}$.  There are $\alpha_1$, $\ldots$, $\alpha_n \in \Lambda$ such that $t\Lambda \cap \Lambda = \bigcup_{i=1}^n \alpha_i \Lambda$.  For each $i$ there is $\beta_i \in \Lambda$ such that $\alpha_i = t \beta_i$.  Then $S_t = \bigvee_{i=1}^n S_{\alpha_i} S_{\beta_i}^*$.

For the converse, let $t \in \Lambda \Lambda^{-1}$.  Our assumption implies that there is $x$ in the $*$-algebra generated by $\{ S_\alpha : \alpha \in \Lambda \}$ with $\| S_t - x \| < 1$.  By Proposition \ref{p.totalset} we may write $x = \sum_{i=1}^n c_i S_{\alpha_i} S_{\beta_i}^* q_i$ with $\alpha_i$, $\beta_i \in \Lambda$ and $q_i \in P$ (where $P$ is as in the remarks before Lemma \ref{l.totalsetone}).  Let $F = \{ \alpha_i : t = \alpha_i \beta_i^{-1} \}$.  We will show that $t \Lambda \cap \Lambda = F \Lambda$.  Let $\gamma \in t\Lambda \cap \Lambda$.  Then there is $\delta \in \Lambda$ such that $t \delta = \gamma$.  Letting $\{ e_\mu : \mu \in \Lambda \}$ be the standard orthonormal basis for $\ell^2(\Lambda)$, we have $S_t e_\delta = e_\gamma$.  Then $\langle x e_\delta, e_\gamma \rangle \not= 0$.  Thus there is $i$ such that $\langle S_{\alpha_i} S_{\beta_i}^* q_i e_\delta , e_\gamma \rangle \not= 0$.  Since $q_i e_\delta \not= 0$ if and only if $q_i e_\delta = e_\delta$, we have  $S_{\beta_i}^* e_\delta \not= 0$.  Then we must have $\delta = \beta_i \mu$ for some $\mu \in \Lambda$.  Then $S_{\alpha_i} S_{\beta_i}^* e_\delta = e_{\alpha_i \mu}$.  It then follows that $\alpha_i \mu = \gamma$.  But then $\alpha_i \beta_i^{-1} = \gamma \delta^{-1} = t$.  Then $\alpha_i \in F$, and hence $\gamma \in F \Lambda$.
\end{proof}

We next consider the construction of the \textit{boundary} of $(\Gamma, \Lambda)$.  Our motivation for this is as follows.  Let us suppose that $Y$ is generated as a group by $\Lambda$.  Let $A \subseteq \Lambda$ be a subset that generates $\Lambda$ as a monoid.  The Cayley graph of $Y$ with respect to the generating set $A$ can be identified with a subset of $Y \times \Lambda$:  the vertices are $Y \times \{ e \}$ and the (oriented) edges are $Y \times A$.  Then $Y \times \Lambda$ can be identified with the directed paths in the Cayley graph.  The ``infinite path space at $t$'' can be identified with $\{ t \} \times \partial \Lambda$.  In order to obtain what we wish to consider as the \textit{directed boundary} of the Cayley graph, we have to make some identifications.   (Note that the following does not require a choice of generating subset $A$, and the consequent interpretation of the Cayley graph).  

We first assume only that $\Lambda$ is finitely aligned as a category of paths.  (We will see in a moment what additional hypothesis is necessary.)  Let $Z = Y \times \Lambda$.  We make $Z$ into a category by defining $Z^0 = Y \times \{e \}$, $s(t,\alpha) = (t \alpha, e)$, $r(t, \alpha) = (t, e)$, and $(t,\alpha) (t \alpha, \beta) = (t, \alpha \beta)$.  It is clear that $Z$ is also a category of paths.  We note that $(t,\alpha) \Cap (t', \alpha')$ iff and only if there are $\beta$, $\beta' \in \Lambda$ such that $(t, \alpha \beta) = (t', \alpha' \beta')$, hence if and only if $t = t'$ and $\alpha \Cap  \alpha'$.  Thus $Z$ is finitely aligned if and only if $\Lambda$ is.  It is clear that $\partial Z = Y \times \partial \Lambda$.  We note that $Y$ acts on $Z$ by automorphisms via left-multiplication in the first coordinate (and hence also by homeomorphisms of $\partial Z$).

For the rest of this section, we will interpret elements of the boundary of a category of paths as directed hereditary subsets of the category.

\begin{Lemma}
\label{l.groupbdyequivalence}
Let the relation $\sim$ on $\partial Z$ be defined by $(t,x) \sim (t', x')$ if there exist $\alpha$, $\alpha' \in \Lambda$ and $y \in \partial \Lambda$ such that $x = \alpha y$, $x' = \alpha' y$, and $t \alpha = t' \alpha'$.  Then $\sim$ is an equivalence relation.  Moreover, the quotient topology on $\partial Z / \sim$ is locally compact.
\end{Lemma}

\begin{proof}
It is clear that $\sim$ is reflexive and symmetric.  Transitivity is proved much as in the proof of Lemma \ref{l.fa.equivalencerelation}:  let $(t,x) \sim (t',x')$ and $(t',x') \sim (t'',x'')$.  Then there are $\alpha$, $\alpha'$, $\beta'$, $\beta'' \in \Lambda$ and $y$, $z \in \partial \Lambda$ such that $x = \alpha y$, $x' = \alpha' y$, $t \alpha = t' \alpha'$, $x' = \beta' z$, $x'' = \beta'' z$, and $t' \beta' = t'' \beta''$.  Since $\alpha' y = \beta' z$, Lemma \ref{l.commontail} implies that there are $\gamma'$, $\delta' \in \Lambda$ and $w \in \partial \Lambda$ such that $y =  \gamma' w$, $z = \delta' w$, and $\alpha' \gamma' = \beta ' \delta'$.  Then $x = \alpha \gamma' w$, $x'' = \beta'' \delta' w$, and $t \alpha \gamma' = t' \alpha' \gamma' = t' \beta' \delta' = t'' \beta'' \delta'$.  Thus $(t,x) \sim (t'', x'')$.

To show that $\partial Z / \sim$ is locally compact, we must show that the quotient map is open (\cite{wil}, Theorem 18.5).  For this we show that the saturation of a basic open set in $\partial Z$ is open.  Let $E \in \CA(\Lambda)$ be such that $\widehat{E} \cap \partial \Lambda \not= \emptyset$, and let $t \in Y$.  Then $\{ t \} \times \widehat{E} \cap \partial \Lambda$ is a basic open set in $\partial Z$.  Suppose that $(s,x) \sim (t,y)$ for some $y \in \widehat{E} \cap \partial \Lambda$.  Then there are $\alpha$, $\beta \in \Lambda$ and $z \in \partial \Lambda$ such that $x = \alpha z$, $y = \beta z$ and $s \alpha = t \beta$.  Thinking of elements of $\partial \Lambda$ as directed hereditary subsets of $\Lambda$, we have from $y = \beta z$ that $\beta \in y$, and from $y \in \widehat{E}$ that $E$ contains all sufficiently large elements of $y$.  Hence there is $\eps \in y \cap \beta \Lambda \cap E$.  Let $\delta = \alpha \sigma^\beta (\eps)$.  Then $s = t \beta \alpha^{-1} = t \eps \delta^{-1}$ and $x = \alpha \sigma^\beta (y) = \delta \sigma^\eps (y)$.  Conversely, if $\eps \in E$ and $s(\delta) = s(\eps)$, and if $y \in (\eps \Lambda \cap E)^{\wedge} \cap \partial \Lambda$, then $(t \eps \delta^{-1}, \delta \sigma^\eps (y) ) \sim (t,y)$.  Thus the saturation of $\{ t \} \times \widehat{E} \cap \partial \Lambda$ equals
\[
\bigcup_{\eps \in E} \bigcup_{\delta \in \Lambda s(\eps)} \{ t \eps \delta^{-1} \} \times \delta \sigma^\eps (\widehat{E}) \cap \partial \Lambda,
\]
an open subset of $\partial Z$.
\end{proof}

We need an additional hypothesis in order to obtain a Hausdorff quotient space.

\begin{Definition}
\label{d.locallyfinitelyexhaustible}
The ordered group $(Y, \Lambda)$ is \textit{locally finitely exhaustible} if for each $t \in Y$ there is a finite set $F \subseteq t \Lambda \cap \Lambda$ such that for each $\alpha \in t \Lambda \cap \Lambda$ there is $\mu \in F$ with $\alpha \Cap \mu$.
\end{Definition}

\begin{Remark}
It is clear that if $(Y,\Lambda)$ is finitely aligned then it is locally finitely exhaustible.
\end{Remark}

\begin{Theorem}
\label{t.groupbdyequivalence}
Let $(Y,\Lambda)$ be a countable ordered group such that $\Lambda$ is finitely aligned as a category of paths.  The following are equivalent.
\begin{enumerate}

\item $\partial Z/ \sim$ is Hausdorff.
\label{t.groupbdyequivalenceone}

\item For each $t \in Y$, there is a finite set $F \subseteq t \Lambda \cap \Lambda$ such that $(t \Lambda \cap \Lambda) \partial \Lambda = F \partial \Lambda$.
\label{t.groupbdyequivalencetwo}

\item
$(Y,\Lambda)$ is locally finitely exhaustible.
\label{t.groupbdyequivalencethree}

\end{enumerate}
\end{Theorem}

\begin{proof}
\eqref{t.groupbdyequivalenceone} $\implies$ \eqref{t.groupbdyequivalencetwo}:  We prove the contrapositive.  Suppose that there is $t \in Y$ such that for each finite set $F \subseteq t\Lambda \cap \Lambda$ there are $\alpha \in t\Lambda \cap \Lambda$ and $z \in \partial  \Lambda$ such that $\alpha z \not\in F \partial \Lambda$.  Let $t \Lambda \cap \Lambda = \{\alpha_1, \alpha_2, \ldots \}$.  Let $\beta_n = t^{-1} \alpha_n \in \Lambda$.  Then $(t, \beta_n z) \sim (e,\alpha_n z)$ for each $z \in \partial \Lambda$.  By assumption, for each $n$ there is $m_n > n$ and $z_n \in \partial \Lambda$ such that $\alpha_{m_n} z_n \not\in \bigcup_{j=1}^n \alpha_j \partial \Lambda$.  Since $\partial \Lambda$ is compact, we may pass to a subsequence so as to assume that $\alpha_{m_n} z_n \to x$ and $\beta_{m_n} z_n \to y$.  We claim that $(t,y) \not\sim (e,x)$.  For otherwise, there is $k$ such that $x = \alpha_k z$ and $y = \beta_k z$, for some $z \in \partial \Lambda$.  We have $\alpha_{m_n} z_n \to \alpha_k z \in \alpha_k \partial \Lambda$.  This means that eventually, $\alpha_{m_n} z_n \in \alpha_k \partial \Lambda$, and hence that $\alpha_k \in \alpha_{m_n} z_n$ (thought of as a directed hereditary subset of $\Lambda$).  But then $\alpha_{m_n} z_n = \alpha_k \sigma^{\alpha_k} ( \alpha_{m_n} z_n ) \in \alpha_k \partial \Lambda$ eventually, contradicting our choice of $m_n$ and $z_n$.  Since $(t, \beta_{m_n} z_n) \sim (e, \alpha_{m_n} z_n)$, we see that $\sim$ is not closed, and hence $\partial Z / \sim$ is not Hausdorff (\cite{wil}, Theorem 13.12).

\noindent
\eqref{t.groupbdyequivalencetwo} $\implies$ \eqref{t.groupbdyequivalenceone}:  Let $(s_i,x_i) \to (s,x)$ and $(t_i, y_i) \to (t,y)$, and assume that $(s_i,x_i) \sim (t_i,y_i)$.  Passing to a subsequence, we may assume that $s_i = s$ and $t_i = t$ for all $i$.  Let $\alpha_i$, $\beta_i \in \Lambda$ and $z_i \in \partial \Lambda$ be such that $s \alpha_i = t \beta_i$, $x_i = \alpha_i z_i$, and $y_i = \beta_i z_i$.  Let $F \subseteq (s^{-1} t) \Lambda \cap \Lambda$ be as in \eqref{t.groupbdyequivalencetwo}.  Since $\alpha_i \in s^{-1} t \Lambda \cap \Lambda$, there is $\mu_i \in F$ such that $x_i = \alpha_i z_i \in \mu_i \partial \Lambda$.  Let $\alpha_i z_i = \mu_i w_i$.  There is $\nu_i \in \Lambda$ such that $s^{-1} t \nu_i = \mu_i$, or $s \mu_i = t \nu_i$.  Then $\beta_i z_i = t^{-1} s \alpha_i z_i = t^{-1} s \mu_i w_i = \nu_i w_i$.  Passing to a subsequence, we may assume that $\mu_i = \mu \in F$ for all $i$.  We now have $x_i = \mu w_i \to x$ and $y_i = \nu w_i \to y$.  Since $\mu \partial \Lambda$ is a clopen set, it must contain $x$, and hence $\mu \in x$.  Similarly, $\nu \in y$.  Then $\sigma^\mu x = \lim_i \sigma^\mu x_i = \lim_i w_i = \lim_i \sigma^\nu y_i = \sigma^\nu y$.  Denoting this common value by $z$, we have $s \mu = t \nu$, $x = \mu z$, and $y = \nu z$.  Therefore $(s,x) \sim (t,y)$. Thus $\sim$ is closed, and so $\partial Z / \sim$ is Hausdorff.

\noindent
\eqref{t.groupbdyequivalencetwo} $\implies$ \eqref{t.groupbdyequivalencethree}:  Let $t \in Y$, and choose $F$ as in \eqref{t.groupbdyequivalencetwo}. Now let $\alpha \in t \Lambda \cap \Lambda$.  Choose any $x \in \partial \Lambda$.  There is $\mu \in F$, and $z \in \partial \Lambda$, such that $\alpha x = \mu z$.  But then $\alpha \Cap \mu$.  Therefore $F$ satisfies Definition  \ref{d.locallyfinitelyexhaustible} for $t$.

\noindent
\eqref{t.groupbdyequivalencethree} $\implies$ \eqref{t.groupbdyequivalencetwo}:  Let $t \in Y$.  Let $F \subseteq t\Lambda \cap \Lambda$ be as Definition \ref{d.locallyfinitelyexhaustible}.  Let $\alpha \in t \Lambda \cap \Lambda$.  First let $x \in \Lambda^{**}$.  For each $\gamma \in \alpha x$ there is $\mu_\gamma \in F$ such that $\gamma \Cap \mu_\gamma$.  Let $(\gamma_n)$ be a cofinal sequence in $\alpha x$.  Passing to a subsequence, we may assume that $(\mu_{\gamma_n})$ is constant, equal to $\mu \in F$.  Then $\mu \Cap \gamma$ for all $\gamma \in \alpha x$.  Since $\alpha x$ is maximal, Lemma \ref{l.directedsets} implies that $\mu \in \alpha x$.  But then $\alpha x \in \mu \partial \Lambda \subseteq F \partial \Lambda$.  Thus $\alpha \Lambda^{**} \subseteq F \partial \Lambda$.  Since $F$ is finite, $F \partial \Lambda$ is closed, so that $\alpha \partial \Lambda \subseteq F \partial \Lambda$.
\end{proof}

\begin{Definition}
\label{d.directedboundary}
Let $(Y, \Lambda)$ be a locally finitely exhaustible ordered group, and suppose that $\Lambda$ is finitely aligned as a category of paths.  We denote by $\partial(Y,\Lambda)$ the (locally compact Hausdorff) quotient space $\partial Z / \sim$, and call it the \textit{directed boundary} of $Y$ (with respect to $\Lambda$), or the \textit{boundary} of $(Y, \Lambda)$.
\end{Definition}

We note that left-multiplication by $Y$ in the first coordinate of $\partial Z$ preserves $\sim$, as well as the basic open sets of $\partial(Y,\Lambda)$ given in the proof of Lemma \ref{l.groupbdyequivalence}.  Thus $Y$ acts by homeomorphisms of $\partial(Y,\Lambda)$.  Let $T = (\{ e \} \times \partial \Lambda) / \sim$, a compact-open subset of $\partial(Y,\Lambda)$.

\begin{Lemma}
\label{l.transversal}
\begin{enumerate}
\item The quotient map is one-to-one from $\{ e \} \times \partial \Lambda$ onto $T$.
\label{l.transversalone}
\item $T$ meets every orbit of the action of $Y$.
\label{l.transversaltwo}
\end{enumerate}
\end{Lemma}

\begin{proof}
\eqref{l.transversalone} If $(e,x) \sim (e,y)$ then there are $\alpha$, $\beta \in \Lambda$ and $z \in \partial \Lambda$ such that $x = \alpha z$, $y = \beta z$, and $e \alpha = e \beta$.  Then $\alpha = \beta$ and $x = y$.

\noindent
\eqref{l.transversaltwo} This is clear.
\end{proof}

\begin{Theorem}
\label{t.transversal}
Let $(Y, \Lambda)$ be a locally finitely exhaustible ordered group, such that $\Lambda$ is a finitely aligned category of paths.  Let $Y \times \partial(Y,\Lambda)$ denote the transformation groupoid associated with the action of $Y$ on $\partial (Y,\Lambda)$.  Then $Y \times \partial(Y,\Lambda) \bigr|_T$ is isomorphic to $G(\Lambda)|_{\partial \Lambda}$.
\end{Theorem}

\begin{proof}
We define a map $G(\Lambda)|_{\partial \Lambda} \to Y \times \partial(Y,\Lambda) \bigr|_T$ by $[\alpha, \beta, x] \mapsto (\alpha \beta^{-1}, [e,\beta x] )$ (where we use square brackets to denote the equivalence classes relative to $\sim$).  To see that the map is one-to-one, suppose that $(\alpha \beta^{-1}, [e,\beta x]) = (\gamma \delta^{-1}, [e,\delta y])$.  Then $\alpha \beta^{-1} = \gamma \delta^{-1}$ and $(e,\beta x) \sim (e,\delta y)$ in $Y \times \partial \Lambda$.  By Lemma \ref{l.transversal}\eqref{l.transversalone} we have $\beta x = \delta y$, so Lemma \ref{l.commontail} gives $\mu$, $\nu \in \Lambda$ and $z \in \partial \Lambda$ such that $x = \mu z$, $y = \nu z$, and $\beta \mu = \delta \nu$.  Then $\alpha \mu = \alpha \beta^{-1} \beta \mu = \gamma \delta^{-1} \delta \nu = \gamma \nu$.  Thus $(\alpha, \beta, x) \sim (\gamma, \delta, y)$.  To see that the map is onto, let $(t,[e,x]) \in Y \times \partial(Y,\Lambda)|_T$.  Then $[t,x] \in T$, so 
$(t,x) \sim (e,y)$ for some $y \in \partial \Lambda$.  Hence there are $\xi$, $\eta \in \Lambda$ and $w \in \partial \Lambda$ such that $x = \xi w$, $y = \eta w$, and $t \xi = e \eta$.  Thus $(t,[e,x]) = (\eta \xi^{-1}, [e,\xi w])$ is the image of $[\eta, \xi, w]$.  Continuity of the map is clear.  Since there is a basis of compact-open sets, it is a homeomorphism.  It is routine to check that the map is a homomorphism.
\end{proof}

\begin{Corollary}
\label{c.groupboundary}
Let $Y$ and $\Lambda$  be as in the theorem.  Then $C^*(\Lambda)$ is Morita equivalent to the crossed product algebra $C_0(\partial(Y,\Lambda)) \times Y$.
\end{Corollary}

\begin{proof}
This follows from Theorem \ref{t.transversal} and \cite{muhrenwil}.
\end{proof}

In the case where $(Y,\Lambda)$ is quasi-lattice ordered, $C^*(\Lambda)$ equals the \textit{boundary quotient} of \cite{crilac}.  Thus we give a locally compact space with an action of $Y$, rather than a compact space with a partial action, having crossed product Morita equivalent to the boundary quotient.

\begin{Example}
\label{e.groupexamples}
We give some examples.  The details are routine, and are omitted.
\begin{enumerate}
\item Let $\Gamma = \langle a, b, c \mid ac = ca,\ bc = cb \rangle$, and let $\Lambda$ be the submonoid generated by $\{ac^n, bc^n, c^m : n \in \IZ, m \in \IN \}$.  ($\Gamma = \IZ^2 *_A \IZ^2$, where $A = \{0\} \times \IZ$, with the product order determined by the lexicographic order on $\IZ^2$.)  Then $\Lambda$ is finitely aligned, with unique minimal common extensions, and $(\Gamma, \Lambda)$ is locally finitely exhaustible, but not finitely aligned.  ($\Gamma$ is right-angled Artinian, but $\Lambda$ is not the submonoid considered in \cite{crilac}.)
\item Let $\Gamma = \langle a, b, c_1, c_2, \ldots \mid - \rangle$, and let $\Lambda$ be the submonoid generated by $\{a_n, b_n, c_{n+1} : n \in \IN \}$, where $a_0 = a$ and $a_n = a c_1^{-1} \cdots c_n^{-1}$ for $n > 0$, and similarly for $b_n$.  Then $\Lambda$ is finitely aligned, with unique minimal common extensions, but $(\Gamma, \Lambda)$ is not locally finitely exhaustible.
\item Let $c$, $d \in \IZ^+$ with $c > 1$, let $\Gamma = \langle a, b \mid a b^c = b^{-d} a \rangle$, and let $\Lambda$ be the submonoid generated by $a$ and $b$.  Then $\Lambda$ is finitely aligned, with unique minimal common extensions, and $(\Gamma, \Lambda)$ is locally finitely exhaustible, but not finitely aligned.  (This example is treated in detail in \cite{spi_baumsol}.)
\end{enumerate}
\end{Example}

\end{Example}

\section{Gauge Actions}
\label{s.gaugeactions}

\begin{Definition}
\label{d.gaugegroups}
Let $\Lambda$ be a category of paths.  Set
\[
H(\Lambda) = C_c(\Lambda,\IZ) / \langle e_\alpha + e_\beta - e_{\alpha\beta} : s(\alpha) = r(\beta), \alpha,\,\beta \in \Lambda \rangle,
\]
and let $\theta \equiv \theta_\Lambda : \Lambda \to H(\Lambda)$ be defined by $\theta(\alpha) = [e_\alpha]$.
\end{Definition}

Thus $H(\Lambda)$ is an abelian group, $\theta$ is a homomorphism, and if $\psi : \Lambda \to Q$ is any homomorphism from $\Lambda$ to an abelian group $Q$, then there exists a unique homomorphism making the following diagram commute:
\[
\begin{tikzpicture}

\node (0_0) at (0,0) [circle] {$Q$};
\node (0_2) at (0,2) [circle] {$\Lambda$};
\node (2_2) at (2,2) [circle] {$H(\Lambda)$};

\draw[-latex,thick,dashed] (2_2) -- (0_0) node[pos=0.5, inner sep=0.5pt, anchor=north west] {$\exists !$};
\draw[-latex,thick] (0_2) -- (2_2) node[pos=0.5, inner sep=0.5pt, anchor=south] {$\theta$};
\draw[-latex,thick] (0_2) -- (0_0) node[pos=0.5, inner sep=0.5pt, anchor=east] {$\psi$};

\end{tikzpicture}
\]
We think of such a homomorphism $\psi$ as a \textit{generalized degree functor}, and $\theta$ as the \textit{maximal degree functor}.  In general, $H(\Lambda)$ might be the trivial group.  However, if $H(\Lambda)$ is large it may provide a useful decomposition of $G(\Lambda)$.  A weak restriction in this direction is the following.

\begin{Definition}
\label{d.nondegeneratedegree}
Let $\Lambda$ be a category of paths, and let $\psi$ be a generalized degree functor on $\Lambda$.  We call $\psi$ \textit{nondegenerate} if $\psi(\alpha) \not= 0$ for $\alpha \not\in \Lambda^0$.
\end{Definition}

Let $(\Lambda_0, \Lambda)$ be a relative category of paths, and let $\psi : \Lambda_0 \to Q$ be a generalized degree functor.  We wish to define a cocycle $c_\psi : G(\Lambda_0,\Lambda) \to Q$ by
\[
c_\psi([\zeta,x]) = \sum_{i=1}^n \psi(\beta_i) - \psi(\alpha_i),
\]
where $\zeta = (\alpha_1,\beta_1,\ldots,\alpha_n,\beta_n)$.  To see that $c_\psi$ is well-defined, let $(\zeta,x) \sim (\zeta',x')$.  Then $x = x'$ and $\Phi_\zeta$ equals $\Phi_{\zeta'}$ near $x$.  There is $\xi \in \CZ$ such that $A(\xi) \subseteq A(\zeta) \cap A(\zeta')$, $x \in \widehat{A(\xi)}$, and $\Phi_\zeta$, $\Phi_{\zeta'}$ agree on $\widehat{A(\xi)}$.  Write $x = \mu z$, where $\mu \in A(\xi)$ and $z \in X_{s(\mu)}$.  Then $\mu \in A(\zeta) \cap A(\zeta')$, and $\Phi_{\zeta}|_{\mu X_{s(\mu)}} = \Phi_{\zeta'}|_{\mu X_{s(\mu)}}$.  Moreover, $\Phi_\zeta(\mu y) = \pphi_\zeta(\mu)y$ and $\Phi_{\zeta'}(\mu y) = \pphi_{\zeta'}(\mu)y$ for all $y \in X_{s(\mu)}$.  Let $y$ correspond to the fixed ultrafilter at $s(\mu)$.  Then we find that $\pphi_\zeta(\mu) = \pphi_{\zeta'}(\mu)$.  Let $\zeta = (\alpha_1, \beta_1, \ldots, \alpha_n, \beta_n)$.  Define $\mu_0 = \mu$, and $\mu_{i+1} = \sigma^{\alpha_{n-i}} \beta_{n-i} \cdots \sigma^{\alpha_n} \beta_n (\mu)$ for $0 \le i < n$.  Then $\mu_n = \pphi_\zeta(\mu)$, and $\sigma^{\alpha_{n-i}} \beta_{n-i} (\mu_i) = \mu_{i+1}$.  Thus $\beta_{n-i} \mu_i = \alpha_{n-i} \mu_{i+1}$, so
\[
\sum_{i=1}^n \psi(\beta_i) - \psi(\alpha_i)= \sum_{i=0}^{n-1} \psi(\mu_{i+1}) - \psi(\mu_i) = \psi(\pphi_\zeta(\mu)) - \psi(\mu).
\]
Hence if $\zeta' = (\gamma_1, \delta_1, \ldots, \gamma_m, \delta_m)$, then
\[
\sum_{j=1}^m \psi(\delta_j) - \psi(\gamma_j) = \psi(\pphi_{\zeta'} (\mu)) - \psi(\mu) = \psi(\pphi_\zeta (\mu)) - \psi(\mu) = \sum_{i=1}^n \psi(\beta_i) - \psi(\alpha_i).
\]
It is clear that $c_\psi$ is a homomorphism, and since it is constant on the compact open set $[\zeta,\widehat{A(\zeta)}]$, it is continuous.  The cocycle $c_\psi$ induces an action $\gamma \equiv \gamma_\psi$ of the compact abelian group $\widehat{Q}$ on $C^*(G)$ in the usual way:
\[
\gamma_z(f)([\zeta,x]) = \langle z, c_\psi([\zeta,x]) \rangle f([\zeta,x]),
\]
for $f \in C_c(G)$, $z \in \widehat{Q}$.  We call $\gamma_\psi$ the \textit{(generalized) gauge action defined by $\psi$}, and $\gamma_{\theta_\Lambda}$ the \textit{maximal gauge action}.  In the usual way, $\gamma_\psi$ can be used to define a conditional expectation $E_\psi$ from $C^*(G)$ to the fixed point algebra $C^*(G)^{\gamma_\psi}$.  One of the most important uses of the gauge action has been to prove nuclearity of $C^*(G)$, and the coincidence of the full and reduced $C^*$-algebras.  We give a sketch of this argument, which is by now more-or-less standard.

\begin{Proposition}
\label{p.nuclearity}
Let $G$ be a Hausdorff \'etale groupoid, $Q$ a countable abelian group, and $c : G \to Q$ a continuous homomorphism.  Let $G^c = c^{-1}(0)$, also a Hausdorff \'etale groupoid.  Suppose that $G^c$ is amenable.  Then $C^*(G)$ is nuclear, and $G$ is amenable.
\end{Proposition}

\begin{proof}
As above, we have the action $\gamma$ of $\widehat{Q}$ on $C^*(G)$.  A standard argument using the expectation of $C^*(G)$ onto the fixed-point algebra $C^*(G)^\gamma$ shows that $C^*(G)^\gamma = \overline{C_c(G^c)}^{\| \cdot \|_{C^*(G)}}$.  We claim that $\| f \|_{C^*(G^c)} = \| f \|_{C^*(G)}$ for $f \in C_c(G^c)$, which implies that $C^*(G)^\gamma = C^*(G^c)$.

To see this, first note that for $v \in G^0$,
\[
\text{Ind}_G v \bigr|_{C_c(G^c)} = \text{Ind}_{G^c} v \oplus \pi,
\]
for some representation $\pi$ of $C_c(G^c)$, where 
$\text{Ind}_G v$ is the representation of $C_c(G)$ induced from the point mass at $v$ (see \cite{pat}, page 107).  Thus
\begin{align*}
\| f \|_{C^*(G^c)}
&= \| f \|_{C^*_r(G^c)}, \text{ since } G^c \text{ is amenable,} \\
&= \sup_{v \in (G^c)^0} \bigl\| \text{Ind}_{G^c} v (f) \bigr\|, \text{ by \cite{pat}, Proposition 3.1.2,} \\
&\le \sup_{v \in G^0} \bigl\| \text{Ind}_G v (f) \bigr\| 
\le \| f \|_{C^*(G)}.
\end{align*}
Since $\| \cdot \|_{C^*(G^c)}$ is the maximal $C^*$-norm on $C_c(G^c)$, it follows that $\| f \|_{C^*(G^c)} = \| f \|_{C^*(G)}$.  Thus the fixed-point algebra $C^*(G)^\gamma = C^*(G^c)$ is nuclear.  This is also the fixed-point algebra of the discrete coaction of $Q$ dual to $\gamma$.  By \cite{quigg}, Corollary 2.17, it follows that $C^*(G)$ is nuclear.  Then $G$ is amenable by \cite{anaren}, Corollary 6.2.4(ii).
\end{proof}

In some situations, this result is obtained more explicitly.  For example, when $\Lambda$ is higher-rank graph, the usual degree functor $d : \Lambda \to \IN^k$ has the \textit{unique factorization} property, which may be expressed in the following form:  if $0 \le n \le d(\alpha)$ then there exists a unique $\beta \in [\alpha]$ with $d(\beta) = n$.  In \cite{raesimyee} (proof of Theorem 3.1) it is proved that the fixed point algebra for the degree functor on a finitely aligned higher-rank graph is AF.  The key step is Lemma 3.2 of \cite{raesimyee}.  We will give sufficient conditions on a finitely aligned category of paths for an analogous result to hold.  While these seem to be far from necessary, they apply to many examples with degree functors that do not have the unique factorization property.  In particular,  our conditions allow the possibility that there may be many factorizations of a path into pieces of given degrees.   The fact that the fixed point algebra is AF is a key step in the proof that the $C^*$-algebra of a finitely aligned higher-rank graph is nuclear (\cite{sim}, Proposition 8.1).   Our conditions allow us to prove this for categories of paths (see Theorem \ref{t.af}; we work at the level of the groupoid).

\begin{Lemma}
\label{l.finiteorbits}
Let $\Lambda$ be a countable finitely aligned category of paths, and let $\psi : \Lambda \to Q$ be a nondegenerate degree functor.  Suppose that $\psi$ has the following two properties.
\begin{enumerate}
\item \label{l.finiteorbits_a} 
If $E \subseteq \Lambda$ is infinite, and if $\psi(\alpha) = \psi(\beta)$ for all $\alpha$, $\beta \in E$, then there is a finite subset $F \subseteq E$ such that $\vee F = \emptyset$.
\item \label{l.finiteorbits_b} 
For every finite subset $S \subseteq \psi(\Lambda)$ there is a finite subset $T \subseteq \psi(\Lambda)$ with $S \subseteq T$ such that for any finite set $E \subseteq \Lambda$, if $\psi(E) \subseteq T$, then $\psi(\vee E) \subseteq T$.
\end{enumerate}
Let $E \subseteq \Lambda$ be a finite set, and let $c \in \Lambda^*$ be a directed hereditary subset of $\Lambda$.  Let
\[
P =
\bigl\{ \beta_n \sigma^{\alpha_n} \cdots \beta_1 \sigma^{\alpha_1} : n \in \IN, \alpha_i, \beta_i \in E,\ \psi(\alpha_i) = \psi(\beta_i) \text{ for } i = 1 \ldots n \bigr\}.
\]
Then there exists $\omega \in c$ such that for any $\zeta \in P$, if $c \in \text{dom}(\zeta)$ then $\omega \in \text{dom}(\zeta)$, and the set
\[
\bigl\{ \zeta(\omega) : \zeta \in P \text{ and } \omega \in \text{dom}(\zeta) \bigr\}
\]
is finite.
\end{Lemma}

We briefly describe the significance of the conditions in the lemma, and of the statement of the lemma.  Condition \eqref{l.finiteorbits_a} implies that a directed subset of $\Lambda$ (e.g. an element of $\Lambda^*$) cannot contain infinitely many paths of the same degree.  Condition \eqref{l.finiteorbits_b} provides a bound on the ``size'' of the minimal common extensions of a set of paths in terms of their degrees.  Finally, the purpose of the lemma is to show that the \textit{fixed point subgroupoid} of $G(\Lambda)$, consisting of elements of the form $[\alpha, \beta, x]$ for which $\psi(\alpha) = \psi(\beta)$, is approximately finite (\cite{ren}, III.1.1) --- see Theorem \ref{t.af}.  (The elements of $P$ are zigzag maps, though a vertex is missing from each end of the corresponding zigzags.  We refer to them with the letter $\zeta$ instead of $\pphi$ for convenience.)

\begin{proof}
Let $S = \psi(E)$, and let $T$ be as in condition \eqref{l.finiteorbits_b}.  Let $F = \{ \gamma \in c : \psi(\gamma) \in T \}$.  Every finite subset of $c$ has an upper bound in $c$.  Hence for any $t \in T$, $\{ \gamma \in c : \psi(\gamma) = t \}$ is finite, by condition \eqref{l.finiteorbits_a}.  Since $T$ is finite, it follows that $F = \bigcup_{t \in T} \{ \gamma \in c : \psi(\gamma) = t \}$ is finite.  Let $\omega \in c$ be an upper bound for $F$, with $\omega \in F$.

For $i \ge 1$ let $\alpha_i$, $\beta_i \in E$ be such that $\psi(\alpha_i) = \psi(\beta_i)$ and such that $\beta_n \sigma^{\alpha_n} \cdots \beta_1 \sigma^{\alpha_1} (c)$ is defined for all $n$.  We will construct elements 
\begin{align*}
& \alpha_i^j, \text{ for } i \ge 1, j \ge 0 \\
& \beta_i^j, \text{ for } i \ge 1, i > j \ge 0 \\
& \gamma_i, \text{ for } i \ge 0
\end{align*}
such that
\begin{align*}
\alpha_i^j \beta_{i+j}^{j+1} &= \beta_{i+j-1}^j \alpha_{i-1}^{j+1} \in \alpha_i^j \vee \beta_{i+j-1}^j, \text{ for $i > 1$ and $j \ge 0$.} \tag{$*$} \\
\alpha_1^i \gamma_i &= \gamma_{i-1} \text{ for $i > 0$, $\alpha_1 \gamma_0 = \omega$.  (Hence also $\omega = \alpha_1 \alpha_1^1 \cdots \alpha_1^i \gamma_i$.)} \tag{$**$}
\end{align*}
The construction is illustrated in Figure~\ref{fig1af}.

\begin{figure}[ht]
\[
\begin{tikzpicture}

\node (00) at (0,0) [circle,fill=black,inner sep=1pt,outer sep = 2pt] {};
\node (02) at (0,2) [circle,fill=black,inner sep=1pt,outer sep = 2pt] {};
\node (04) at (0,4) [circle,fill=black,inner sep=1pt,outer sep = 2pt] {};
\node (06) at (0,6) [circle,fill=black,inner sep=1pt,outer sep = 2pt] {};
\node (08) at (0,8) [circle,fill=black,inner sep=1pt,outer sep = 2pt] {};
\node (22) at (2,2) [circle,fill=black,inner sep=1pt,outer sep = 2pt] {};
\node (24) at (2,4) [circle,fill=black,inner sep=1pt,outer sep = 2pt] {};
\node (26) at (2,6) [circle,fill=black,inner sep=1pt,outer sep = 2pt] {};
\node (28) at (2,8) [circle,fill=black,inner sep=1pt,outer sep = 2pt] {};
\node (44) at (4,4) [circle,fill=black,inner sep=1pt,outer sep = 2pt] {};
\node (46) at (4,6) [circle,fill=black,inner sep=1pt,outer sep = 2pt] {};
\node (48) at (4,8) [circle,fill=black,inner sep=1pt,outer sep = 2pt] {};
\node (66) at (6,6) [circle,fill=black,inner sep=1pt,outer sep = 2pt] {};
\node (68) at (6,8) [circle,fill=black,inner sep=1pt,outer sep = 2pt] {};
\node (88) at (8,8) [circle,fill=black,inner sep=1pt,outer sep = 2pt] {};

\node (m47) at (-4,7) [circle,fill=black,inner sep=1pt,outer sep = 2pt] {};

\draw[-latex,thick] (02) -- (00) node[pos=0.5, inner sep=0.5pt, anchor=west] {$\alpha_1^{\phantom{1}}$};
\draw[-latex,thick] (04) -- (02) node[pos=0.5, inner sep=0.5pt, anchor=west] {$\alpha_1^1$};
\draw[-latex,thick] (06) -- (04) node[pos=0.5, inner sep=0.5pt, anchor=west] {$\alpha_1^2$};
\draw[-latex,thick] (08) -- (06) node[pos=0.5, inner sep=0.5pt, anchor=west] {$\alpha_1^3$};
\draw[-latex,thick] (24) -- (22) node[pos=0.5, inner sep=0.5pt, anchor=west] {$\alpha_2^{\phantom{1}}$};
\draw[-latex,thick] (26) -- (24) node[pos=0.5, inner sep=0.5pt, anchor=west] {$\alpha_2^1$};
\draw[-latex,thick] (28) -- (26) node[pos=0.5, inner sep=0.5pt, anchor=west] {$\alpha_2^2$};
\draw[-latex,thick] (46) -- (44) node[pos=0.5, inner sep=0.5pt, anchor=west] {$\alpha_3^{\phantom{1}}$};
\draw[-latex,thick] (48) -- (46) node[pos=0.5, inner sep=0.5pt, anchor=west] {$\alpha_3^1$};
\draw[-latex,thick] (68) -- (66) node[pos=0.5, inner sep=0.5pt, anchor=west] {$\alpha_4^{\phantom{1}}$};

\draw[-latex,thick] (68) -- (88) node[pos=0.5, inner sep=0.5pt, anchor=south] {$\beta_4^{\phantom{1}}$};
\draw[-latex,thick] (48) -- (68) node[pos=0.5, inner sep=0.5pt, anchor=south] {$\beta_4^1$};
\draw[-latex,thick] (28) -- (48) node[pos=0.5, inner sep=0.5pt, anchor=south] {$\beta_4^2$};
\draw[-latex,thick] (08) -- (28) node[pos=0.5, inner sep=0.5pt, anchor=south] {$\beta_4^3$};
\draw[-latex,thick] (46) -- (66) node[pos=0.5, inner sep=0.5pt, anchor=south] {$\beta_3^{\phantom{1}}$};
\draw[-latex,thick] (26) -- (46) node[pos=0.5, inner sep=0.5pt, anchor=south] {$\beta_3^1$};
\draw[-latex,thick] (06) -- (26) node[pos=0.5, inner sep=0.5pt, anchor=south] {$\beta_3^2$};
\draw[-latex,thick] (24) -- (44) node[pos=0.5, inner sep=0.5pt, anchor=south] {$\beta_2^{\phantom{1}}$};
\draw[-latex,thick] (04) -- (24) node[pos=0.5, inner sep=0.5pt, anchor=south] {$\beta_2^1$};
\draw[-latex,thick] (02) -- (22) node[pos=0.5, inner sep=0.5pt, anchor=south] {$\beta_1^{\phantom{1}}$};

\draw[-latex,thick] (m47) -- (00) node[pos=0.65, inner sep=0.5pt, anchor=south west] {$\omega$};
\draw[-latex,thick] (m47) -- (02) node[pos=0.65, inner sep=0.5pt, anchor=south west] {$\gamma_0$};
\draw[-latex,thick] (m47) -- (04) node[pos=0.65, inner sep=0.5pt, anchor=south west] {$\gamma_1$};
\draw[-latex,thick] (m47) -- (06) node[pos=0.65, inner sep=0.5pt, anchor=south west] {$\gamma_2$};
\draw[-latex,thick] (m47) -- (08) node[pos=0.65, inner sep=0.5pt, anchor=south east] {$\gamma_3$};

\node [circle,fill=black,inner sep=1pt] at (-5.4,7) {};
\node [circle,fill=black,inner sep=1pt] at (-6,7) {};
\node [black,above] at (-5.4,7) {$\sigma^\omega(c)$};
\node (m487) at (-4.8,7) [circle,fill=black,inner sep=1pt,outer sep = 2pt] {};
\draw[-latex,thick] (m487) -- (m47) node[inner sep=.5pt]{};

\node [circle,fill=black,inner sep=1pt] at (-2,9) {};
\node [circle,fill=black,inner sep=1pt] at (-2,9.2) {};
\node [circle,fill=black,inner sep=1pt] at (-2,9.4) {};
\node [circle,fill=black,inner sep=1pt] at (9,9) {};
\node [circle,fill=black,inner sep=1pt] at (9.2,9.2) {};
\node [circle,fill=black,inner sep=1pt] at (9.4,9.4) {};
\node [circle,fill=black,inner sep=1pt] at (4,9) {};
\node [circle,fill=black,inner sep=1pt] at (4,9.2) {};
\node [circle,fill=black,inner sep=1pt] at (4,9.4) {};

\end{tikzpicture}
\]

\caption{} \label{fig1af}

\end{figure}

Figure~\ref{fig2af} gives a sketch of a typical square in the diagram ($i > 1$).  Its location is described as follows.  Any square containing the edge $\alpha_i^j$ has its bottom edge $i + j - 1$ units from the bottom of the diagram (measured vertically from $r(\alpha_1)$).  Any square containing the edge $\beta_p^q$ has its left edge $p - q - 1$ units from the leftmost vertical line of the diagram (measured horizontally from $s(\beta_1)$).

\begin{figure}[ht]

\[
\begin{tikzpicture}

\node (00) at (0,0) [circle,fill=black,inner sep=1pt,outer sep = 2pt] {};
\node (20) at (2,0) [circle,fill=black,inner sep=1pt,outer sep = 2pt] {};
\node (02) at (0,2) [circle,fill=black,inner sep=1pt,outer sep = 2pt] {};
\node (22) at (2,2) [circle,fill=black,inner sep=1pt,outer sep = 2pt] {};

\draw[-latex,thick] (02) -- (00) node[pos=0.5, inner sep=0.5pt, anchor=east] {$\alpha_{i-1}^{j+1}$};
\draw[-latex,thick] (22) -- (20) node[pos=0.5, inner sep=0.5pt, anchor=west] {$\alpha_i^j$};
\draw[-latex,thick] (00) -- (20) node[pos=0.5, inner sep=0.5pt, anchor=north] {$\beta_{i+j-1}^j$};
\draw[-latex,thick] (02) -- (22) node[pos=0.5, inner sep=0.5pt, anchor=south] {$\beta_{i+j}^{j+1}$};

\end{tikzpicture}
\]

\caption{} \label{fig2af}

\end{figure}

The construction proceeds inductively on $n \ge 1$, where after stage $n$ we have constructed $\alpha_i^j$ with $i + j \le n$, $\beta_i^j$ with $i \le n$, and $\gamma_i$ with $i \le n - 1$.  (Thus it is stage 4 that is pictured in Figure~\ref{fig1af}.)  When $n = 1$, we have just $\alpha_1$ and $\beta_1$.  Since $c$ is in the domain of $\beta_1 \sigma^{\alpha_1}$, then $\alpha_1 \in c$.  Since $\psi(\alpha_1) \in T$, there exists $\gamma_0$ such that $\omega = \alpha_1 \gamma_0$.

Now suppose that stage $n$ of the induction has been completed.  By the commutativity of the squares and triangles constructed thus far, two paths having the same source and range must be equal.  Therefore we may indicate the path $\mu$ by the ordered pair $\bigl(r(\mu),s(\mu)\bigr)$.  We base the inductive step on three claims. (These claims use only condition $(*)$ of the construction.)

\noindent
\textit{Claim 1.} For $k + \ell \le p + q \le n$ and $k - 1 \ge p$, we have
\[
\bigl(r(\alpha_k^\ell),s(\alpha_p^q)\bigr) 
\in \bigl(r(\alpha_k^\ell),r(\beta_{p+q}^q)\bigr) 
\vee \bigl(r(\beta_{k+\ell-1}^\ell),r(\alpha_p^q)\bigr).
\]
The situation of claim 1 is illustrated in Figure~\ref{fig3af}.

\begin{figure}[ht]

\[
\begin{tikzpicture}

\node (0_45) at (0,4.5) [circle,fill=black,inner sep=1pt,outer sep = 2pt] {};
\node (2_45) at (2,4.5) [circle,fill=black,inner sep=1pt,outer sep = 2pt] {};
\node (0_25) at (0,2.5) [circle,fill=black,inner sep=1pt,outer sep = 2pt] {};
\node (2_25) at (2,2.5) [circle,fill=black,inner sep=1pt,outer sep = 2pt] {};

\node (60) at (6,0) [circle,fill=black,inner sep=1pt,outer sep = 2pt] {};
\node (62) at (6,2) [circle,fill=black,inner sep=1pt,outer sep = 2pt] {};
\node (80) at (8,0) [circle,fill=black,inner sep=1pt,outer sep = 2pt] {};
\node (82) at (8,2) [circle,fill=black,inner sep=1pt,outer sep = 2pt] {};

\draw[-latex,thick] (0_45) -- (0_25) node[pos=0.5, inner sep=0.5pt, anchor=east] {$\alpha_p^q$};
\draw[-latex,thick] (2_45) -- (2_25) node[pos=0.5, inner sep=0.5pt, anchor=west] {$\alpha_{p+1}^{q-1}$};
\draw[-latex,thick] (0_45) -- (2_45) node[pos=0.5, inner sep=0.5pt, anchor=south] {$\beta_{p+q}^q$};
\draw[-latex,thick] (0_25) -- (2_25) node[pos=0.5, inner sep=0.5pt, anchor=north] {$\beta_{p+q-1}^{q-1}$};

\draw[-latex,thick] (82) -- (80) node[pos=0.5, inner sep=0.5pt, anchor=west] {$\alpha_k^\ell$};
\draw[-latex,thick] (62) -- (60) node[pos=0.5, inner sep=0.5pt, anchor=east] {$\alpha_{k-1}^{\ell+1}$};
\draw[-latex,thick] (62) -- (82) node[pos=0.5, inner sep=0.5pt, anchor=south] {$\beta_{k+\ell}^{\ell+1}$};
\draw[-latex,thick] (60) -- (80) node[pos=0.5, inner sep=0.5pt, anchor=north] {$\beta_{k+\ell-1}^\ell$};

\draw[dashed] (2,4.5) -- (8,4.5);
\draw[dashed] (8,4.5) -- (8,2);
\draw[dashed] (0,2.5) -- (0,0);
\draw[dashed] (0,0) -- (6,0);

\node [circle,fill=black,inner sep=1pt] at (3,2.375) {};
\node [circle,fill=black,inner sep=1pt] at (4,2.25) {};
\node [circle,fill=black,inner sep=1pt] at (5,2.125) {};

\end{tikzpicture}
\]

\caption{} \label{fig3af}

\end{figure}

The two constituent paths may be decomposed as follows:
\begin{align*}
\bigl( r(\alpha_k^\ell),r(\beta_{p+q}^q) \bigr)
&= \bigl( r(\alpha_k^\ell), r(\alpha_{p+1}^{q-1}) \bigr) \alpha_{p+1}^{q-1} \\
\bigl( r(\alpha_k^\ell),r(\alpha_p^q) \bigr)
&= \bigl( r(\alpha_k^\ell), r(\beta_{p+q-1}^{q-1}) \bigr) \beta_{p+q-1}^{q-1}. \\
\end{align*}
Since $r(\alpha_{p+1}^{q-1}) = r(\beta_{p+q-1}^{q-1})$, the claim is equivalent to asserting that
\[
\alpha_{p+1}^{q-1} \beta_{p+q}^q = \beta_{p+q-1}^{q-1} \alpha_p^q
\in \alpha_{p+1}^{q-1} \vee \beta_{p+q-1}^{q-1},
\]
which is part of the inductive hypothesis.

\noindent
\textit{Claim 2.} For $j < i \le n$ we have
\[
\beta_i \beta_i^1 \cdots \beta_i^j
= \beta_i \sigma^{\alpha_i} \beta_{i-1} \sigma^{\alpha_{i-1}} \cdots \beta_{i-j} \sigma^{\alpha_{i-j}} 
(\alpha_{i-j} \alpha_{i-j}^1 \cdots \alpha_{i-j}^j).
\]

We prove this by induction on $j$.  If $j=0$, then we have $\beta_i = \beta_i \sigma^{\alpha_i} (\alpha_i)$ whenever $1 \le i \le n$.  Suppose inductively that the claim is true for $j$, and  all $i$ such that $j < i \le n$.  Let $j + 1 < i \le n$.  Then
\begin{align*}
\beta_i \beta_i^1 \cdots \beta_i^{j+1}
&= \beta_i \sigma^{\alpha_i} \cdots \beta_{i-j} \sigma^{\alpha_{i-j}} ( \alpha_{i-j} \cdots \alpha_{i-j}^j \beta_i^{j+1} ) \\
&= \beta_i \sigma^{\alpha_i} \cdots \beta_{i-j} \sigma^{\alpha_{i-j}} ( \beta_{i-j-1} \alpha_{i-j-1}^1 \cdots \alpha_{i-j-1}^{j+1} ) \\
&= \beta_i \sigma^{\alpha_i} \cdots \beta_{i-j} \sigma^{\alpha_{i-j}} \beta_{i-j-1} \sigma^{\alpha_{i-j-1}} ( \alpha_{i-j-1} \alpha_{i-j-1}^1 \cdots \alpha_{i-j-1}^{j+1} ).
\end{align*}

\noindent
\textit{Claim 3.} For $i + j \le n$,
\begin{align*}
\psi( \alpha_i \alpha_i^1 \cdots \alpha_i^j ) &\in T,  \\
\psi( \beta_i \beta_i^1 \cdots \beta_i^j ) &\in T.
\end{align*}

\noindent
We again prove this by induction on $j$.  When $j=0$ the claim follows from the fact that $\alpha_i$, $\beta_i \in E$.  Suppose the claim is true for $j$, and for all $i$ such that $i + j \le n$.  Let $i + j + 1 \le n$.  We have
\begin{align*}
\beta_i \alpha_i^1 \cdots \alpha_i^{j+1}
&= \alpha_{i+1} \alpha_{i+1}^1 \cdots \alpha_{i+1}^j \beta_{i+j+1}^{j+1} \\
&\in \beta_i \alpha_i^1 \cdots \alpha_i^j \vee \alpha_{i + 1} \alpha_{i+1}^1 \cdots \alpha_{i+1}^j,
\end{align*}
by Claim 1.  Since the latter two paths have degrees in $T$ by the inductive hypothesis (and since $\psi(\beta_i) = \psi(\alpha_i)$), so do their minimal common extensions, by the definition of $T$.  Then since $\psi(\beta_i) = \psi(\alpha_i)$ we have that $\psi(\alpha_i \alpha_i^1 \cdots \alpha_i^{j+1}) \in T$.  By Claim 2, we know that $\beta_i \beta_i^1 \cdots \beta_i^{j+1} = \beta_i \sigma^{\alpha_i} \cdots \beta_{i-j-1} \sigma^{\alpha_{i-j-1}} ( \alpha_{i-j-1} \alpha_{i-j-1}^1 \cdots \alpha_{i-j-1}^{j+1})$.  Since the maps $\beta_\ell \sigma^{\alpha_\ell}$ preserve degree, the claim is proved.

We now return to the inductive step in the construction.  By Claim 2 we have
\begin{align*}
\beta_n \beta_n^1 \cdots \beta_n^{n-1} \gamma_{n-1}
&= \beta_n \sigma^{\alpha_n} \cdots \beta_1 \sigma^{\alpha_1} ( \alpha_1 \alpha_1^1 \cdots \alpha_1^{n-1} ) \gamma_{n-1} \\
&= \beta_n \sigma^{\alpha_n} \cdots \beta_1 \sigma^{\alpha_1} ( \omega ).
\end{align*}
We know that 
\[
\alpha_{n+1} 
\in \beta_n \sigma^{\alpha_n} \cdots \beta_1 \sigma^{\alpha_1} (c)
= \beta_n \sigma^{\alpha_n} \cdots \beta_1 \sigma^{\alpha_1} ( \omega ) \sigma^{\omega}(c)
= \beta_n \beta_n^1 \cdots \beta_n^{n-1} \gamma_{n-1} \sigma^\omega (c),
\]
by assumption.  Thus there exists
\[
\alpha_{n+1} \beta_{n+1}^1 = \beta_n \alpha_n^1 
\in ( \alpha_{n+1} \vee \beta_n ) \cap \beta_n \beta_n^1 \cdots \beta_n^{n-1} \gamma_{n-1} \sigma^\omega (c).
\]
Thus $\alpha_n^1 \in \beta_n^1 \cdots \beta_n^{n-1} \gamma_{n-1} \sigma^\omega(c)$.  Thus there exists
\[
\alpha_n^1 \beta_{n+1}^2 = \beta_n^1 \alpha_{n-1}^2 
\in (\alpha_n^1 \vee \beta_n^1) \cap \beta_n^1 \cdots \beta_n^{n-1} \gamma_{n-1} \sigma^\omega(c).
\]
Inductively, we construct the squares
\[
\alpha_{n+1-\ell}^\ell \beta_{n+1}^{\ell+1} = \beta_n^\ell \alpha_{n-\ell}^{\ell+1} 
\in (\alpha_{n+1-\ell}^\ell \vee \beta_n^\ell) \cap \beta_n^\ell \cdots \beta_n^{n-1} \gamma_{n-1} \sigma^\omega(c)
\]
for $0 \le \ell \le n - 1$.  From the case $\ell = n-1$ we have that $\alpha_1^n \in \gamma_{n-1} \sigma^\omega(c)$.  Therefore
\[
\alpha_1 \alpha_1^1 \cdots \alpha_1^n \in \alpha_1 \alpha_1^1 \cdots \alpha_1^{n-1} \gamma_{n-1} \sigma^\omega(c) = c.
\]
By Claim 3 we know that $\psi(\alpha_1 \cdots \alpha_1^n) \in T$, so that $\omega \in \alpha_1 \cdots \alpha_1^n \Lambda$.  Therefore there exists $\gamma_n$ such that $\omega = \alpha_1 \cdots \alpha_1^n \gamma_n$.  In particular, $\gamma_{n-1} = \alpha_1^n \gamma_n$.  This finishes the construction of the diagram.

We now finish the proof of the lemma.  Note that it is possible (and in fact is the case) that some of the elements we have constructed are reduced to vertices.  For the rest of the proof, we will use the term \textit{path} to mean a path in $\Lambda$ given by the composition of some of the $\alpha_i^j$ and $\beta_i^j$ in the above diagram.  By the \textit{length} of such a path we will mean the number of these elements used that are not vertices.  Let $E_0 = E \cup s(E)$, and recursively, for $i > 0$, 
\begin{align*}
E_i^0 &= \bigcup_{\alpha,\beta \in E_{i-1}} \sigma^\alpha (\alpha \vee \beta) \\
E_i &= E_i^0 \cup s(E_i^0).
\end{align*}
Note that $E_i$ consists of all $\alpha_i^j$, $\beta_i^j$, and their sources, that can arise in the diagram.  We claim that for $\delta \in E_j \setminus E_{j-1}$, there is a path $\mu$ of length at least $j$ such that $s(\mu) = r(\delta)$.  We prove this by induction on $j$. When $j=0$ there is nothing to prove.  Suppose that the claim is true for $j-1$.  Let $\delta \in E_j \setminus E_{j-1}$.  First let $\delta \in E_j^0$.  Then there are $\alpha$, $\beta \in E_{j-1}$ with $\alpha\delta \in \alpha \vee \beta$.  Since $\delta \not\in E_{j-1}$ we may assume that $\alpha \not\in E_{j-2}$ or $\beta \not\in E_{j-2}$.  Either way, there is a path $\mu$ of length at least $j-1$ such that $s(\mu) = r(\alpha)$ ($= r(\beta)$).  If $\alpha \not\in \Lambda^0$ then $\mu \alpha$ is a path of length at least $j$ with $s(\mu\alpha) = r(\delta)$.  If $\alpha \in \Lambda^0$, then $\delta = \beta \in E_{j-1}$, a contradiction.  Next, an element of $s(E_j^0)$ is of the form $s(\delta)$, for $\delta \in E_j^0$.  If $\mu$ is as above for $\delta$, then $\mu\delta$ does the job for $s(\delta)$.  This finishes the proof of the claim.

It follows from Claims 1 and 3 that $\psi(\alpha_i^j)$, $\psi(\beta_i^j)$, $\psi(\gamma_i) \in T - T$ for all $i$ and $j$.  Let $U = (T - T) \cap \psi(\Lambda) \setminus \{e\}$.  Then $U$ is a finite set.  Since $\psi$ is nondegenerate, there exists $N$ such that for any $n > N$, and any $u_1$, $\ldots$, $u_n \in U$, we have $u_1 + \cdots + u_n \not\in T$.  It follows that $E_N = E_{N-1}$.    We note that all $E_j$ are finite sets:  this is clearly true for $j=0$, and follows inductively for $j$ by the finite alignment of $\Lambda$.  It now follows that the diagram contains only finitely many paths.  Since for all $n$ there is a path $\mu$ such that $\beta_n \sigma^{\alpha_n} \cdots \beta_1 \sigma^{\alpha_1} (c) = \mu \sigma^\omega(c)$, the lemma is proved.
\end{proof}

We will use the notation $G^\psi$ for the ``fixed point groupoid'':  $G^\psi = \{g \in G : c_\psi(g) = 0\}$.  Lemma \ref{l.finiteorbits} implies that $G^\psi$ is ``locally finite''.  However it does not guarantee that $G^\psi$ has trivial isotropy, which is necessary in order to have an AF groupoid.  The following definition characterizes trivial isotropy in a manner slightly more intrinsic than the property itself.

\begin{Definition}
\label{d.nonisotropic}
Let $\Lambda$ be a countable finitely aligned category of paths, and let $\psi$ be a nondegenerate degree functor on $\Lambda$.  We say that $\psi$ is \textit{non-isotropic} if the following condition holds.  Whenever $\alpha_i$, $\beta_i \in \Lambda$ for $i \in \IN$ are such that 
\begin{enumerate}
\item $\alpha_i \alpha_{i+1} = \beta_i \beta_{i+1}$ for $i \in \IN$
\label{d.nonisotropicone}
\item $\psi(\alpha_1) = \psi(\beta_1)$,
\label{d.nonisotropictwo}
\end{enumerate}
then $\alpha_1 = \beta_1$.
\end{Definition}

\begin{Remark}
It follows from \eqref{d.nonisotropicone} and \eqref{d.nonisotropictwo} that $\psi(\alpha_i) = \psi(\beta_i)$ for all $i$.  If in addition $\alpha_1 = \beta_1$, then $\alpha_i = \beta_i$ for all $i$, by left-cancellation.
\end{Remark}

\begin{Lemma}
\label{l.nonisotropic}
Let $\Lambda$ be a countable finitely aligned category of paths, and let $\psi$ be a nondegenerate degree functor on $\Lambda$.  Then $G^\psi$ has trivial isotropy if and only if $\psi$ is non-isotropic.
\end{Lemma}

\begin{proof}
First we assume that $G^\psi$ has nontrivial isotropy.  Let $[\beta_1,\alpha_1,x]$ be a nontrivial element of $G^\psi_x$.  Thus $\alpha_1 x = \beta_1 x$, $\psi(\alpha_1) = \psi(\beta_1)$, and $\alpha_1 \not= \beta_1$.  Then there are $\alpha_2$, $\beta_2 \in \Lambda$ such that $\alpha_1 \alpha_2 = \beta_1 \beta_2 \in \alpha_1 x$.  But then $\alpha_2$, $\beta_2 \in x$.  Since $x$ is directed, there are $\alpha_3$, $\beta_3 \in \Lambda$ such that $\alpha_2 \alpha_3 = \beta_2 \beta_3 \in x$.  Now we have
\[
\sigma^{\alpha_2} x
= \sigma^{\alpha_1 \alpha_2} \alpha_1 x
= \sigma^{\beta_1 \beta_2} \beta_1 x
= \sigma^{\beta_2} x.
\]
Then since $\alpha_3$, $\beta_3 \in \sigma^{\alpha_2} x$, there are $\alpha_4$, $\beta_4 \in \Lambda$ such that $\alpha_3 \alpha_4 = \beta_3 \beta_4 \in \sigma^{\alpha_2} x$.  Continuing this process inductively, we see that $\psi$ is not non-isotropic.

Conversely, suppose that $\psi$ is not non-isotropic.  Let $\alpha_i$, $\beta_i$ satisfy \eqref{d.nonisotropicone} and \eqref{d.nonisotropictwo} of Definition \ref{d.nonisotropic} with $\alpha_1 \not= \beta_1$.  Let $x = \bigcup_{n \ge 2} [\alpha_2 \cdots \alpha_n]$.  Then $x \in \Lambda^*$.  We claim that $\alpha_1 x = \beta_1 x$.  First note that since $\alpha_i \alpha_{i+1} = \beta_i \beta_{i+1}$ for all $i$, we have that $\alpha_1 \cdots \alpha_n = \beta_1 \cdots \beta_n$ whenever $n$ is even, while for odd $n$ we have that $\alpha_1 \cdots \alpha_{n-1} \beta_n = \beta_1 \cdots \beta_n = \beta_1 \alpha_2 \cdots \alpha_n$.  Now let $\gamma \in \beta_1 x$.  Then there are $\delta \in \Lambda$ and $n \in \IN$ such that $\gamma \delta = \beta_1 \alpha_2 \cdots \alpha_n$.  If $n$ is even, we may replace $\delta$ by $\delta \alpha_{n+1}$; thus we may assume that $n$ is odd.  Hence 
\[
\gamma \delta \beta_{n+1} = \beta_1 \alpha_2 \cdots \alpha_n \beta_{n+1}
= \alpha_1 \cdots \alpha_{n-1} \beta_n \beta_{n+1}
= \alpha_1 \cdots \alpha_{n+1}.
\]
It follows that $\gamma \in \alpha_1 x$.  By symmetry we have that $\alpha_1 x = \beta_1 x$.  Thus $[\beta_1, \alpha_1, x]$ is a nontrivial element of the isotropy of $G^\psi$.
\end{proof}

\begin{Theorem}
\label{t.af}
Let $\Lambda$ be a countable finitely aligned category of paths, and let $\psi$ be a nondegenerate non-isotropic degree functor on $\Lambda$ satisfying properties \eqref{l.finiteorbits_a} and \eqref{l.finiteorbits_b} of Lemma \ref{l.finiteorbits}.  Then $G^\psi$ is an AF groupoid.
\end{Theorem}

\begin{proof}
Let $\Lambda = \{ \lambda_1, \lambda_2, \ldots \}$ be an enumeration of the elements of $\Lambda$.  Put $E_n = \{ \lambda_1, \ldots, \lambda_n \}$ and $P_n = \bigl\{ \beta_j \sigma^{\alpha_j} \cdots \beta_1 \sigma^{\alpha_1} : j \in \IN,\ \alpha_i, \beta_i \in E_n,\ \psi(\alpha_i) = \psi(\beta_i) \text{ for } 1 \le i \le j \bigr\}$.  For $x \in X$ let $\omega_n(x)$ be as in Lemma \ref{l.finiteorbits} for $E_n$ and $x$.  Let $F_n(x) = \bigl\{ \zeta \bigl( \omega_n(x) \bigr) : \zeta \in P_n \bigr\}$, and let $Z_n(x) = \bigcup_{\mu \in F_n(x)} \mu X_{s(\omega_n(x))}$.  Then $F_n(x)$ is a finite set by Lemma \ref{l.finiteorbits}, hence $Z_n(x)$ is a compact-open subset of $X$, and $x \in Z_n(x)$.  By compactness of $X$ we may choose $x_1$, $\ldots$, $x_m \in X$ such that $X = \bigcup_{i = 1}^m Z_n(x_i)$.

Note that since $\psi$ is non-isotropic, if $\mu_1$, $\mu_2 \in F_n(x)$ are such that $\mu_1 x = \mu_2 x$, then $\mu_1 = \mu_2$. Lemma \ref{l.finiteorbits} implies that if $\zeta \in P_n$ and $\mu \in F_n(x)$ are such that $\zeta(\mu) = \mu$, then $\zeta|_{\mu X_{s(\omega_n(x))}} = id$.  Thus if $K \subseteq X_{s(\omega_n(x))}$ is any compact-open subset, then the subgroupoid of $G^\psi$ given by $\bigl\{ [\zeta, \mu y] : \zeta \in P_n, \mu \in F_n(x), y \in K \bigr\}$ is elementary, equal to the Cartesian product of a finite groupoid with the space $K$.

Let $L_i = X_{s(\omega_n(x_i))}$ for $1 \le i \le n$.  Set $K_1 = L_1$, and put $Y_1 = \bigcup_{\mu \in F_n(x_1)} \mu K_1 = Z_n(x_1)$.  For $\nu \in F_n(x_2)$, we have that $\nu L_2 \setminus Y_1 \subseteq \nu X_{s(\nu)}$ is a compact-open subset.  Therefore there is a compact-open set $K_2 \subseteq L_2$ such that $\nu L_2 \setminus Y_1 = \nu K_2$.  In fact, the set $K_2$ doesn't depend on the choice of $\nu \in F_n(x_2)$:  if $\nu' \in F_n(x_2)$, then since $Y_1$ is invariant under the maps in $P_n$, we have that
\[
\nu' K_2 = \nu' \sigma^{\nu} (\nu K_2)
= \nu' \sigma^{\nu} (\nu L_2 \setminus Y_1) 
= \nu' \sigma^{\nu} ( \nu L_2 ) \setminus Y_1
= \nu' L_2 \setminus Y_1.
\]
We now put $Y_2 = \bigcup_{\mu \in F_n(x_2)} \mu K_2 \subseteq Z_n(x_2)$.  Again, $Y_2$ is invariant under $P_n$, so we may define $K_3 \subseteq L_3$ by choosing any $\nu \in F_n(x_3)$ and setting $\nu L_3 \setminus (Y_1 \cup Y_2) = \nu K_3$.  Continuing this process, we obtain compact-open sets $K_i \subseteq X_{s(\omega_n(x_i))}$, and $Y_i = \bigcup_{\mu \in F_n(x_i)} \mu K_i \subseteq Z_n(x_i)$, for $1 \le i \le m$.  Then $X = \bigsqcup_{i=1}^m Y_i$.  Let $H_i = \bigl\{ [\zeta, \mu y] : \zeta \in P_n,\ \mu \in F_n(x_i),\ y \in K_i \bigr\}$.  Then $H_i$ is an elementary groupoid with unit space $H_i^0 = Y_i$, hence $H_i \cap H_j = \emptyset$ for $i \not= j$. We claim that $G_n := \bigcup_{i=1}^m H_i = \bigl\{ [\zeta,x] : \zeta \in P_n,\ x \in X \bigr\}$.  To see this, let $x \in X$.  Then $x \in Z_n(x_{i_0})$ for some $i_0$.  We may assume that $i_0$ is minimal with this property.  Then $x \in \mu X_{s(\omega_n(x_{i_0}))}$ for some $\mu \in F_n(x_{i_0})$.  Since $x \not\in Z_n(x_i)$ for $i < i_0$, $x \not\in \nu K_i$ for any $i < i_0$ and $\nu \in F_n(x_i)$.  Therefore $x \in H_{i_0}^0$. 

Finally, it is clear that $G_n \subseteq G_{n+1}$, and that $G^\psi = \bigcup_{n=1}^\infty G_n$.  Therefore $G^\psi$ is AF.
\end{proof}

\begin{Remark}
\label{r.uniquefactorization}
In the case that there is a degree functor taking values in an ordered abelian group and satisfying the unique factorization property, then all hypotheses of Theorem \ref{t.af} are satisfied.  (For an example, see \cite{ephrem}.)
\end{Remark}

\section{Aperiodicity} 
\label{s.aperiodicity}

Throughout this section, $\Lambda$ will be a countable finitely aligned category of paths.  Note that right-cancellation in $\Lambda$ is crucial for the results of this section.

\begin{Definition}
\label{d.aperiodicity}
A point $x \in \partial \Lambda$ is \textit{aperiodic} (or \textit{left-aperiodic}) if for all $\alpha$, $\beta \in x$ with $\alpha \not= \beta$ we have $\sigma^\alpha(x) \not= \sigma^\beta(x)$.
\end{Definition}

\begin{Lemma}
\label{l.aperiodic}
If $x$ is aperiodic, then so are $\mu x$ and $\sigma^\nu(x)$ for all $\mu \in \Lambda r(x)$ and $\nu \in x$.
\end{Lemma}

\begin{proof}
The fact that $\sigma^\nu(x)$ is aperiodic follows from the definition.  Let $\mu \in \Lambda r(x)$, and let $\alpha$, $\beta \in \mu x$.  Then there are $\alpha'$, $\beta' \in \Lambda$, and $\gamma$, $\delta \in x$, such that $\alpha \alpha' = \mu \gamma$ and $\beta \beta' = \mu \delta$.  Since $\gamma$, $\delta \in x$ we have $\gamma \Cap \delta$.  Let $\gamma \gamma' = \delta \delta' \in x$.  Then
\begin{equation}
\alpha \alpha' \gamma' = \mu \gamma \gamma' = \mu \delta \delta' = \beta \beta' \delta'.
\label{eq.aperiodic}
\end{equation}
Hence
\[
\sigma^\alpha(\mu x) = \sigma^\alpha(\mu \gamma \gamma' \sigma^{\gamma\gamma'}(x))
= \alpha' \gamma' \sigma^{\gamma\gamma'}(x),
\]
and similarly, $\sigma^\beta(\mu x) = \beta'\delta' \sigma^{\delta\delta'}(x)$.  Thus if $\sigma^\alpha(\mu x) = \sigma^\beta(\mu x)$, then $\alpha' \gamma' z = \beta' \delta' z$, where $z = \sigma^{\gamma\gamma'}(x)$.  Since $z$ is aperiodic by the other part of the lemma, we have $\alpha'\gamma' = \beta'\delta'$.  Then by equation \eqref{eq.aperiodic} we have
\[
\alpha \beta' \delta' = \alpha \alpha' \gamma' = \beta \beta' \delta',
\]
and hence $\alpha = \beta$ (by right-cancellation).
\end{proof}

The importance of aperiodicity is illustrated by the following Proposition.

\begin{Proposition}
\label{p.aperiodicityandisotropy}
Let $x \in \partial \Lambda$.  $G$ has trivial isotropy at $x$ if and only if $x$ is aperiodic.
\end{Proposition}

\begin{proof}
First suppose that $G$ has trivial isotropy at $x$.  Let $\alpha$, $\beta \in x$ with $\sigma^\alpha(x) = \sigma^\beta(x) = y$.  Then $[\alpha,\beta,y] \in G(x)$, the isotropy subgroup of $G$ at $x$.  Since $G(x) = \{x\}$ by assumption, we have $[\alpha,\beta,y] = [e,e,x]$.  From Definition \ref{d.fa.groupoid.one} there are $z \in \partial \Lambda$ and $\gamma$, $\delta \in \Lambda$ such that $y = \gamma z$, $x = \delta z$, $\alpha \gamma = e \delta$, and $\beta \gamma = e \delta$.  But then $\alpha \gamma = \beta \gamma$, and hence $\alpha = \beta$ (by right-cancellation).

Next suppose that $x$ is aperiodic.  Let $[\alpha,\beta,y] \in G(x)$.  Then $\alpha y = \beta y = x$, so $\sigma^\alpha (x) = y = \sigma^\beta (x)$.  The aperiodicity implies that $\alpha = \beta$.  Thus $[\alpha,\beta,y] = [e,e,\alpha y] = [e,e,x]$.  Hence $G(x) = \{ x \}$.
\end{proof}

Recall that a Hausdorff \'etale groupoid is \textit{topologically free} (or \textit{essentially free} in \cite{arzren}) if the set of units having trivial isotropy is dense in the unit space (\cite{archspi}, \cite{ana}).

\begin{Proposition}
\label{p.topfree}
$G|_{\partial \Lambda}$ is topologically free if and only if for every $v \in \Lambda^0$, there is an aperiodic point in $v \partial \Lambda$.
\end{Proposition}

\begin{proof}
We will use Proposition \ref{p.aperiodicityandisotropy}.  If $G|_{\partial \Lambda}$ is topologically free then the aperiodic points are dense in $\partial \Lambda$.  Since $v \partial \Lambda$ is an open set, it contains an aperiodic point.  Conversely, suppose that the condition in the statement holds.  Let $U \subseteq \partial \Lambda$ be a nonempty open set.  Let $y \in U \cap \Lambda^{**}$.  Since $\{ \gamma \partial \Lambda : \gamma \in y \}$ is a neighborhood base at $y$, there is $\gamma \in y$ such that $\gamma \partial \Lambda \subseteq U$.  By hypothesis there is an aperiodic point $x \in s(\gamma) \partial \Lambda$.  Then $\gamma x$ is aperiodic, by Lemma \ref{l.aperiodic}, and $\gamma x \in \gamma \partial \Lambda \subseteq U$.
\end{proof}

Aperiodicity with respect to rightward shifts is not as useful as (left-) aperiodicity; for example, if $v \in \Lambda^0$ is a \textit{sink}, i.e. if $\Lambda v = \{ v \}$, then there are no rightward shifts on elements of $v \partial \Lambda$.  Nevertheless, it will be convenient to have some results about the two kinds of aperiodicity.

\begin{Definition}
\label{d.rightaperiodic}
A point $x \in \partial \Lambda$ is \textit{right-aperiodic} if for all $\alpha$, $\beta$ in $\Lambda r(x)$ with $\alpha \not= \beta$, we have $\alpha x \not= \beta x$.
\end{Definition}

\begin{Lemma}
\label{l.rightaperiodic}
Let $x \in \partial \Lambda$.  Then $x$ is aperiodic if and only if $\sigma^\mu (x)$ is right-aperiodic for all $\mu \in x$.
\end{Lemma}

\begin{proof}
First we show that aperiodicity of $x$ implies right-aperiodicity of $x$.  For this, suppose that $x$ is not right-aperiodic.  Then there are $\alpha \not= \beta$ with $\alpha x = \beta x$.  Let $\alpha \alpha' = \beta \beta' \in \alpha x$.  Then
\begin{align*}
\alpha' &= \sigma^\alpha(\alpha \alpha') \in \sigma^\alpha(\alpha x) = x \\
\beta' &= \sigma^\beta(\beta\beta') \in \sigma^\beta(\beta x) = x.
\end{align*}
Since $\alpha \not= \beta$ then $\alpha' \not= \beta'$.  But 
\[
\sigma^{\alpha'}(x) = \sigma^{\alpha\alpha'}(\alpha x) = \sigma^{\beta \beta'}(\beta x) = \sigma^{\beta'}(x).
\]
Therefore $x$ is not aperiodic.  Now, by Lemma \ref{l.aperiodic} we see that if $x$ is aperiodic and $\mu \in x$, then $\sigma^\mu x$ is aperiodic, and hence $\sigma^\mu x$ is right-aperiodic.

For the converse, let $x$ not be aperiodic.  Then there are $\alpha \not= \beta$ in $x$ such that $\sigma^\alpha(x) = \sigma^\beta(x)$.  Let $\gamma = \alpha \alpha' = \beta \beta' \in x$.  Then
\[
\alpha' \sigma^\gamma (x) = \sigma^\alpha (x) = \sigma^\beta (x) = \beta' \sigma^\gamma (x).
\]
Therefore $\sigma^\gamma (x)$ is not right-aperiodic.
\end{proof}

\begin{Lemma}
\label{l.topfree}
$G|_{\partial \Lambda}$ is topologically free if and only if for all $\alpha \not= \beta$ in $\Lambda$ there is $x \in s(\alpha) \partial \Lambda$ such that $\alpha x \not= \beta x$.
\end{Lemma}

\begin{proof}
First suppose that $G|_{\partial \Lambda}$ is topologically free, and let $\alpha \not= \beta$ in $\Lambda$.  By Proposition \ref{p.topfree} there is an aperiodic point $x \in s(\alpha) \partial \Lambda$.  By Lemma \ref{l.rightaperiodic}, $x$ is right-aperiodic, and hence $\alpha x \not= \beta x$.  Conversely, suppose that $G|_{\partial \Lambda}$ is not topologically free.  Then there is a nonempty open set $U \subseteq \partial \Lambda$ containing no aperiodic points.  For $\alpha \not= \beta$ in $\Lambda$ let 
\[
D_{ \{ \alpha, \beta \} } = \{ x \in U : \alpha,\beta \in x \text{ and } \sigma^\alpha (x) = \sigma^\beta (x) \}.
\]
Then $D_{ \{ \alpha, \beta \} }$ is a relatively closed subset of $U$, and $U = \bigcup \{ D_{ \{ \alpha, \beta \} } : \alpha, \; \beta \in \Lambda,\ \alpha \not= \beta \}$.  Since $U$ is a Baire space there are $\alpha \not= \beta$ such that $W = \text{int}( D_{ \{ \alpha, \beta \} } ) \not= \emptyset$.  Then $\sigma^\alpha |_W = \sigma^\beta |_W$.  As in the proof of Proposition \ref{p.topfree}, we may replace $W$ by a set of the form $\gamma \partial \Lambda$.  Further, we may choose $\gamma \in \alpha \Lambda \cap \beta \Lambda$.  Thus $\gamma = \alpha \alpha' = \beta \beta'$.  Then for $x \in W$ we have
\[
\alpha' \sigma^\gamma (x) = \sigma^\alpha(x) = \sigma^\beta (x) = \beta' \sigma^\gamma (x).
\]
Thus $\alpha' y = \beta' y$ for all $y \in s(\gamma) \partial \Lambda$.
\end{proof}

We now give a ``local'' criterion for aperiodicity, in that it refers only to elements of $\Lambda$.  This generalizes a recent result of Lewin and Sims (\cite{lewsim}).

\begin{Definition}
\label{d.periodicity}
Let $\alpha$, $\beta \in \Lambda$ with $\alpha \not= \beta$, $s(\alpha) = s(\beta)$, and $r(\alpha) = r(\beta)$.  We say that $\Lambda$ has \textit{$\{\alpha,\beta\}$-periodicity} if for all $\gamma \in s(\alpha) \Lambda$ we have $\alpha \gamma \Cap \beta \gamma$.  We say that $\Lambda$ is \textit{aperiodic} if $\Lambda$ does not have $\{\alpha,\beta\}$-periodicity for any $\alpha$ and $\beta$.
\end{Definition}

\begin{Lemma}
\label{l.periodicity}
$\Lambda$ has $\{ \alpha,\beta \}$-periodicity if and only if $\alpha x = \beta x$ for all $x \in s(\alpha) \partial \Lambda$.
\end{Lemma}

\begin{proof}
First suppose that $\alpha x = \beta x$ for all $x \in s(\alpha) \partial \Lambda$.  Let $\gamma \in s(\alpha) \Lambda$.  Choose $x \in \gamma \partial \Lambda$; then $\gamma \in x$.  Then $\alpha \gamma \in \alpha x$, $\beta \gamma \in \beta x$, and $\alpha x = \beta x$.  Hence $\alpha \gamma \Cap \beta \gamma$.

Conversely, suppose that there is $x \in s(\alpha) \partial \Lambda$ with $\alpha x \not= \beta x$.  Since $\{ y \in s(\alpha) \partial \Lambda : \alpha y \not= \beta y \}$ is an open set, we may assume that $x \in s(\alpha) \Lambda^{**}$.  Since $\alpha x$ and $\beta x$ are maximal, Lemma \ref{l.directedsets} gives $\gamma_1$, $\gamma_2 \in x$ such that $\alpha \gamma_1 \perp \beta \gamma_2$.  Let $\gamma_1'$ and $\gamma_2'$ be such that $\gamma = \gamma_1 \gamma_1' = \gamma_2 \gamma_2' \in x$.  Then $\alpha \gamma \perp \beta \gamma$.
\end{proof}

\begin{Theorem}
\label{t.aperiodic} 
$G|_{\partial \Lambda}$ is topologically free if and only if $\Lambda$ is aperiodic.
\end{Theorem}

\begin{proof}
The theorem follows from Lemmas \ref{l.topfree} and \ref{l.periodicity}.
\end{proof}

\begin{Remark}
\label{r.aperiodicity}
$\Lambda$ is aperiodic if and only if the following holds:  for all $\alpha \not= \beta$ in $\Lambda$ there is $\gamma \in s(\alpha)\Lambda$ such that $\alpha \gamma \perp \beta \gamma$.
\end{Remark}

We now give analogs of the usual ``uniqueness'' theorems in the subject.  The next two results rely on the characterization of $C^*(\Lambda)$ by generators and relations, given in Theorem \ref{t.restrictionboundary}.

\begin{Theorem} (Cuntz-Krieger uniqueness theorem)
\label{t.ckuniqueness}
Let $\Lambda$ be a countable finitely aligned category of paths with amenable groupoid.  Suppose that $\Lambda$ is aperiodic.  Let $\pi$ be a $*$-homomorphism from $C^*(\Lambda)$ into a $C^*$-algebra.  If $\pi(S_v) \not= 0$ for all $v \in \Lambda^0$, then $\pi$ is faithful.
\end{Theorem}

\begin{proof}
Since $G|_{\partial \Lambda}$ is topologically free, ideals in $C^*(\Lambda)$ are determined by their intersection with $C_0(\partial \Lambda)$ (the argument in \cite{archspi} works for Hausdorff \' etale groupoids.  See also \cite{exel2}).  Since every nonempty open set in $\partial \Lambda$ contains a set of the form $\partial \Lambda \cap \gamma \Lambda$ for some $\gamma$, and $\pi(S_\gamma S_\gamma^*) \not= 0$ since $\pi(S_{s(\gamma)}) \not= 0$, we see that $\pi|_{C_0(\partial \Lambda)}$ is faithful.  Therefore $\pi$ must be faithful.
\end{proof}

\begin{Theorem} (Gauge-invariant uniqueness theorem)
\label{t.gaugeinvariantuniqueness}
Let $\Lambda$ be a countable finitely aligned category of paths, and let $\psi : \Lambda \to Q$ be a non-degenerate non-isotropic degree functor satisfying the conditions \eqref{l.finiteorbits_a} and \eqref{l.finiteorbits_b} of Lemma \ref{l.finiteorbits}.  Let $\gamma = \gamma_\psi$ be the associated gauge action of $\widehat{Q}$ on $C^*(\Lambda)$.  Let $\pi$ be a $*$-homomorpism from $C^*(\Lambda)$ to a $C^*$-algebra $B$, and suppose that there is an action $\delta$ of $\widehat{Q}$ on $B$ such that $\pi$ is equivariant for $\gamma$ and $\delta$.  If $\pi(S_v) \not= 0$ for all $v \in \Lambda^0$ then $\pi$ is faithful.
\end{Theorem}

\begin{proof}
As in the proof of Theorem \ref{t.ckuniqueness}, we find that $\pi|_{C_0(\partial \Lambda)}$ is faithful.  $G^\gamma$ is an AF groupoid, by Theorem \ref{t.af}, and $\partial \Lambda$ is the unit space of $G^\gamma$.  Thus the injectivity of $\pi|_{C_0(\partial \Lambda)}$ implies that $\pi$ is faithful on $C^*(G^\gamma) = C^*(\Lambda)^\gamma$.  By the equivariance of $\pi$ it follows that $\pi(C^*(\Lambda)^\gamma) = \pi(C^*(\Lambda))^\delta$.  Now the usual argument using the faithful conditional expectations defined by the two actions of $\widehat{Q}$ shows that $\pi$ is faithful.
\end{proof}

We now characterize those $\Lambda$ for which the groupoid $G|_{\partial \Lambda}$ is minimal (this is a version of \textit{cofinality} for a category of paths; see \cite{lewsim}, Definition 3.3).

\begin{Theorem}
\label{t.minimality}
$G|_{\partial \Lambda}$ is minimal if and only if for every pair $u, v \in \Lambda^0$ there exists $F \in FE(v)$ such that for each $\alpha \in F$, we have $u \Lambda s(\alpha) \not= \emptyset$.
\end{Theorem}

\begin{proof}
First suppose that the condition in the statement holds, let $x \in \partial \Lambda$, and  let $U \subseteq \partial \Lambda$ be open.  Choose $\gamma \in \Lambda$ such that $\gamma \partial \Lambda \subseteq U$ (as in the proof of Proposition \ref{p.topfree}).  Let $u = s(\gamma)$.  By Theorem \ref{t.boundary} there is $\mu \in x$ such that for all $F \in FE(s(\mu))$ we have $\mu F \cap x \not= \emptyset$.  Let $v = s(\mu)$, and choose $F \in FE(v)$ such that for each $\alpha \in F$, we have $u \Lambda s(\alpha) \not= \emptyset$.  Let $\alpha \in F$ with $\mu \alpha \in x$.  By the assumed condition there is $\beta \in u \Lambda s(\alpha)$.  Then $g = [\gamma \beta, \mu \alpha, \sigma^{\mu \alpha} (x)] \in G|_{\partial \Lambda}$ satisfies $s(g) = x$ and $r(g) \in U$.

Next suppose that $G|_{\partial \Lambda}$ is minimal.  Let $u, v \in \Lambda^0$.  If $x \in v \partial \Lambda$, then there is $[\alpha, \beta, y] \in G|_{\partial \Lambda}$ with $\alpha y = x$ and $\beta y \in u \partial \Lambda$ (i.e. with $r(\beta) = u$).  Thus for each $x \in v \partial \Lambda$ there exists $\alpha(x) \in x$ such that $u \Lambda s(\alpha(x)) \not= \emptyset$.  Then $\{ \alpha(x) \partial \Lambda : x \in v \partial \Lambda \}$ is an open cover of $v \partial \Lambda$.  By compactness there are $x_1$, $\ldots$, $x_n \in v \partial \Lambda$ such that $\{ \alpha(x_i) \partial \Lambda : 1 \le i \le n \}$ is a cover.  We claim that $\{ \alpha(x_i) : 1 \le i \le n \}$ is exhaustive.  For this, let $\gamma \in v \Lambda$.  Choose $x \in v \partial \Lambda$ with $\gamma \in x$.  Choose $i$ such that $x \in \alpha(x_i) \partial \Lambda$.  Then $\alpha(x_i) \in x$.  Hence $\gamma \Cap \alpha(x_i)$.
\end{proof}

We end this section with a sufficient condition that $G|_{\partial \Lambda}$ be locally contractive.  Unlike the case of a directed graph (\cite{spi_graphalg}, Theorem 3.4), we do not know if the condition is necessary (see also \cite{sim}, \cite{evasim}).  We adapt some notions from \cite{evasim}.

\begin{Definition}
\label{d.cycle}
Let $\Lambda$ be a category of paths. A \textit{generalized cycle} in $\Lambda$ is a pair $(\mu,\nu) \in \Lambda \times \Lambda$ such that $\mu \not= \nu$, $s(\mu) = s(\nu)$, $r(\mu) = r(\nu)$, and $\mu \tau \Cap \nu$ for all $\tau \in s(\mu) \Lambda$.
\end{Definition}

(A \textit{cycle} in $\Lambda$ is a path $\alpha \not\in \Lambda^0$ such that $s(\alpha) = r(\alpha)$.  A cycle $\alpha$ defines a generalized cycle $(\alpha,s(\alpha))$.  Examples in \cite{evasim} show that it is possible for a higher-rank graph to contain generalized cycles, but no cycles.)  Lemma 3.2 of \cite{evasim} gives two other equivalent descriptions of generalized cycles in higher-rank graphs.  Since their proof cites other work, and the proof for categories of paths is quite simple, we present it here.

\begin{Lemma}
\label{l.generalizedcycle}
Let $\mu$, $\nu \in \Lambda$ with that $\mu \not= \nu$, $s(\mu) = s(\nu)$, and $r(\mu) = r(\nu)$.  The following are equivalent:
\begin{enumerate}
\item \label{l.generalizedcycle.a} $(\mu,\nu)$ is a generalized cycle.
\item \label{l.generalizedcycle.b} $\sigma^\mu(\mu \vee \nu)$ is exhaustive.
\item \label{l.generalizedcycle.c} $\mu \partial \Lambda \subseteq \nu \partial \Lambda$.
\end{enumerate}
\end{Lemma}

\begin{proof}
\eqref{l.generalizedcycle.a} $\implies$ \eqref{l.generalizedcycle.b}:  Let $\gamma \in s(\mu) \Lambda$.  Then $\mu \gamma \Cap \nu$, so there are $\delta$, $\eps$ such that $\mu \gamma \delta = \nu \eps$.  There is $\eta \in \mu \vee \nu$ such that $\mu \gamma \delta \in \eta \Lambda$.  Write $\eta = \mu \mu' = \nu \nu'$ and $\mu \gamma \delta = \eta \xi$.  Then $\mu \gamma \delta = \eta \xi = \mu \mu' \xi$, so $\gamma \delta = \mu' \xi$.  Then $\mu' \in \sigma^\mu(\mu \vee \nu)$ and $\mu' \Cap \gamma$.

\noindent
\eqref{l.generalizedcycle.b} $\implies$ \eqref{l.generalizedcycle.c}:  Let $x \in s(\mu) \partial \Lambda$, $x$ a directed hereditary subset of $\Lambda$.  Since $\sigma^\mu(\mu \vee \nu)$ is exhaustive, there is $\mu' \in \sigma^\mu(\mu \vee \nu) \cap x$.  There is $\nu'$ such that $\mu \mu' = \nu \nu'$ ($\in \mu \vee \nu$), so $\mu x = \mu \mu' \sigma^{\mu'} x = \nu \nu' \sigma^{\mu'} x \in \nu \partial \Lambda$.

\noindent
\eqref{l.generalizedcycle.c} $\implies$ \eqref{l.generalizedcycle.a}:  Let $\tau \in s(\mu) \Lambda$.  Let $z \in s(\tau) \partial \Lambda$, so $\tau z \in s(\mu) \partial \Lambda$.  Then $\mu \tau z \in \mu \partial \Lambda \subseteq \nu \partial \Lambda$, so $\nu \in \mu \tau z$.  Therefore $\nu \Cap \mu \tau$.
\end{proof}

The next definition is adapted from Definition 3.5 of \cite{evasim}.

\begin{Definition}
\label{d.entrance}
The generalized cycle $(\mu,\nu)$ \textit{has an entrance} if there is $\tau \in s(\mu) \Lambda$ such that $\mu \perp \nu \tau$.
\end{Definition}

It follows from Lemma \ref{l.generalizedcycle} that $(\mu,\nu)$ has an entrance if and only if $\mu \partial \Lambda \subsetneq \nu \partial \Lambda$.  We recall that a groupoid is \textit{locally contractive} if for every nonempty open subset $U$ of the unit space there exists an open $G$-set $\Delta$ such that $s(\overline{\Delta}) \subseteq U$ and $r(\overline{\Delta}) \subsetneq s(\Delta)$ (\cite{ana}, \cite{lacspi}).

\begin{Theorem}
\label{t.contractive}
Let $\Lambda$ be a countable category of paths.  Suppose that for each $v \in \Lambda^0$ there is a generalized cycle, $(\mu,\nu)$, having an entrance, such that $v \Lambda r(\mu) \not= \emptyset$ (i.e. every vertex is \textit{seen} by a generalized cycle having an entrance).  Then $G|_{\partial \Lambda}$ is locally contractive.
\end{Theorem}

\begin{proof}
Let $U \subseteq \partial \Lambda$ be nonempty.  Let $E \in \CA$ be such that $\emptyset \not= \widehat{E} \cap \partial \Lambda \subseteq U$.  Then there is $C \in \Lambda^{**} \cap \widehat{E}$.  By Theorem \ref{t.ultrafilters}\eqref{t.ultrafilters.a} there is $\gamma \in C$ such that $\gamma \Lambda \subseteq E$.  By hypothesis there is a generalized cycle, $(\mu,\nu)$, having an entrance, and $\alpha \in s(\gamma) \Lambda r(\mu)$.  Let $\Delta = [\gamma \alpha \mu, \gamma \alpha \nu, s(\mu) \partial \Lambda]$.  Then $\Delta$ is a compact-open $G$-set, $s(\Delta) \subseteq U$, and $r(\Delta) = \gamma \alpha \mu \partial \Lambda \subsetneq \gamma \alpha \nu \partial \Lambda = s(\Delta)$.
\end{proof}

\section{Example:  amalgamation of categories of paths}
\label{s.amalgamation}

We give a generalization of the examples termed \textit{hybrid graph algebras} that were constructed in \cite{kirchmodels}, Definition 2.1.  In particular, these include the obvious generalizations of those examples, and many others besides.  The results of this section give considerable simplification to those constructions.

\begin{Definition}
\label{d.amalgamation}
Let $\{ \Lambda_i : i \in I \}$ be a collection of categories of paths.  Let $\sim$ be an equivalence relation on $\bigcup_{i \in I} \Lambda_i^0$.  Let 
\[
L = \bigl\{ (\alpha_1, \alpha_2, \ldots, \alpha_n) : \alpha_j \in \bigcup_{i \in I} \Lambda_i,\ s(\alpha_j) \sim r(\alpha_{j+1}), \text{ for all } j, n \ge 1 \bigr\}.
\]
$L$ admits a partially defined concatenation:  $L^2 = \bigl\{ \bigl( (\alpha_1, \ldots, \alpha_m), \; (\beta_1, \ldots, \beta_n) \bigr) : s(\alpha_m) \sim r(\beta_1) \bigr\}$, and then $(\alpha_1, \ldots, \alpha_m) (\beta_1, \ldots, \beta_n) = (\alpha_1, \ldots, \alpha_m, \beta_1, \ldots, \beta_n)$.  (Thus $L$ is a \textit{semigroupoid}, in that composition is not everywhere defined, and there are no units.)  We define a relation $\approx$ on $L$ as follows.  Let $\alpha$, $\beta \in L$.  Then $\alpha \approx \beta$ if there are $\alpha^0$, $\ldots$, $\alpha^n \in L$ such that $\alpha^0 = \alpha$, $\alpha^n = \beta$, and for each $j$ one of the following holds:

\begin{enumerate}
\item 
\label{d.amalgamation_itemone}
$\alpha^j = (\mu_1, \cdots, \mu_k, \theta_1, \cdots, \theta_\ell, \nu_1, \cdots, \nu_m)$, where $\theta_1$, $\ldots$, $\theta_\ell \in \Lambda_i$ for some $i \in I$, and $s(\theta_p) = r(\theta_{p+1})$ for all $p$, and $\alpha_{j+1} = (\mu_1, \cdots, \mu_k, \theta, \nu_1, \cdots, \nu_m)$, where $\theta = \theta_1 \cdots \theta_\ell$ in $\Lambda_i$.
\item 
\label{d.amalgamation_itemtwo}
As in \eqref{d.amalgamation_itemone}, but with the roles of $\alpha_j$ and $\alpha_{j+1}$ reversed.
\item
\label{d.amalgamation_itemthree}
$\alpha^j = (\mu_1, \cdots, \mu_k, w, \nu_1, \cdots, \nu_m)$, where $w \sim s(\mu_k)$ (and hence also $w \sim r(\nu_1)$), and $\alpha_{j+1} = (\mu_1, \cdots, \mu_k, \nu_1, \cdots, \nu_m)$.
\item\label{d.amalgamation_itemfour}
As in \eqref{d.amalgamation_itemthree}, but with the roles of $\alpha_j$ and $\alpha_{j+1}$ reversed.
\end{enumerate}
\end{Definition}

It is clear that $\approx$ is an equivalence relation on $L$.  We note that if $\alpha \approx \alpha'$ and $\beta \approx \beta'$, then $(\alpha,\beta) \in L^2$ implies that $(\alpha', \beta') \in L^2$ and $\alpha \beta \approx \alpha' \beta'$.  Thus concatenation descends to $\Lambda = L / \approx$.  We will show that $\Lambda$ is a category of paths.  First, let $S \subseteq \bigcup_{i \in I} \Lambda_i$ be an equivalence class of $\sim$.  It is easy to see that $S$ must also be an equivalence class of $\approx$.  We define $\Lambda^0 = \bigl( \bigcup_{i \in I} \Lambda_i^0 \bigr) \bigm/ \sim$.  These are the identity elements for concatenation in $\Lambda$.  Associativity in $\Lambda$ follows from associativity in $L$.  In order to verify the properties of a category of paths, it is helpful to have a normal form for elements of $\Lambda$.

\begin{Lemma}
\label{l.amalgamationnormalform}
Let $\alpha \in L$.  There exists a unique element $\beta \in L$ such that
\begin{enumerate}
\item $\alpha \approx \beta$.
\item $\beta = (\beta_1, \ldots, \beta_m)$ with
\begin{enumerate}
\item $\beta_j \not\in \bigcup_{i \in I} \Lambda_i^0$ for all $j$.
\item $s(\beta_j) \not= r(\beta_{j+1})$ for all $j$.
\end{enumerate}
\end{enumerate}
\end{Lemma}

The element $\beta$ is called the \textit{normal form} of $\alpha$.

\begin{proof}
For the existence, note that we can obtain $\beta$ from $\alpha$ by first deleting all vertices among the entries of $\alpha$ and then concatenating adjacent entries $\alpha_j \alpha_{j+1}$ if $s(\alpha_j) = r(\alpha_{j+1})$.  For the uniqueness, note that if $\alpha^k$ are succesive moves as in the definition of $\approx$, then the above process of constructing $\beta$ gives the same result for $\alpha^k$ as for $\alpha^{k+1}$.  Therefore equivalent elements of $L$ yield the same normal form under this process.  In particular, two normal forms for the same element must be equal.
\end{proof}

Now suppose that $[\alpha] [\beta] = [\alpha] [\gamma]$ in $\Lambda$ (where square brackets denote equivalence classes of $\approx$).  Then $\alpha \beta \approx \alpha \gamma$.  Let $\alpha = (\alpha_1, \ldots, \alpha_m)$, $\beta = (\beta_1, \ldots, \beta_k)$, and $\gamma = (\gamma_1, \ldots, \gamma_n)$ in normal form.  First suppose that $s(\alpha_m) \not= r(\beta_1)$.  Then $\alpha \beta$ is in normal form.  If $s(\alpha_m) = r(\gamma_1)$, then the normal form of $\alpha \gamma$ is $(\alpha_1, \ldots, \alpha_{m-1}, \alpha_m \gamma_1, \gamma_2, \ldots, \gamma_n)$.  Since this equals $\alpha \beta$, we must have $\alpha_m = \alpha_m \gamma_1$.  By cancellation in the appropriate $\Lambda_i$ we see that $\gamma_1 \in \Lambda_i^0$, contradicting the assumption that $\gamma$ is in normal form.  Thus $s(\alpha_m) \not= r(\gamma_n)$, and it follows that $\beta = \gamma$.  Now suppose that $s(\alpha_m) = r(\beta_1)$.  By the previous argument, we must have $s(\alpha_m) = r(\gamma_1)$.  Then the same considerations show that $\alpha_m \beta_1 = \alpha_m \gamma_1$ in the appropriate $\Lambda_i$, and hence that $\beta_1 = \gamma_1$, $k = n$, and $\beta_j = \gamma_j$ for $j > 1$.  Hence again $\beta = \gamma$.  The arguments for right cancellation and absence of inverses are similar.

\begin{Definition}
We refer to $\Lambda$ constructed as above as an \textit{amalgmation} of the collection $\{ \Lambda_i\}$.  It depends on the choice of the equivalence relation $\sim$ on the set of units of the collection.
\end{Definition}

We have the following result on common extensions in an amalgamation.

\begin{Lemma}
\label{l.amalgamationextension}
Let $\alpha = (\alpha_1, \ldots, \alpha_m)$ and $\beta = (\beta_1, \ldots, \beta_k)$ be in normal form. Then $[\alpha] \Cap [\beta]$ if and only if 
\begin{enumerate}
\item in case $m \not= k$, we have $\alpha_j = \beta_j$ for $j < \min \{m,k\}$, and if, e.g., $m < k$ then $\beta_m \in \alpha_m \Lambda_i$, where $\alpha_m$, $\beta_m \in \Lambda_i$ for some $i$ (i.e. one extends the other);
\item in case $m = k$, we have $\alpha_j = \beta_j$ for $j < m$; $\alpha_m$, $\beta_m \in \Lambda_i$ for some $i$; and $\alpha_m \Cap \beta_m$.
\end{enumerate}
\end{Lemma}

The proof follows easily from the use of normal forms.  It follows from this Lemma that $\Lambda$ is finitely aligned if all of the $\Lambda_i$ are finitely aligned.

We conclude with a result implying that an amalgamation of finitely aligned categories of paths has a degree functor defining an AF core (as in section \ref{s.gaugeactions}) if each of the individual categories has one.  Thus in particular, such an amalgamation has nuclear $C^*$-algebras.  The degree functor we construct will generally be more complicated than necessary (compare, e.g., with the examples in \cite{kirchmodels}).  We require a couple of additional hypotheses.  First, it is possible that a nondegenerate degree functor takes inverse values on two paths (for example, if such paths cannot occur as parts of the same path).  Since the amalgamation can allow such paths to be composed, we have to proscribe such behavior.  Second, property \eqref{l.finiteorbits_b} of Lemma \ref{l.finiteorbits} could be violated for an amalgamation if the range of the degree functor admits divisibility.  We replace it with a stronger version. The hypotheses we give are convenient rather than sharp, but they are easily verified in many examples, such as an amalgamation of higher-rank graphs.

\begin{Theorem}
\label{t.amalgamationdegree}
Let $\{ \Lambda_i : i \in I \}$ be a collection of finitely aligned categories of paths, and let $\Lambda$ be their amalgamation over an equivalance relation on $\bigcup_{i \in I} \Lambda_i^0$.  For each $i$ let $\psi_i : \Lambda_i \to Q_i$ be a nondegenerate non-isotropic degree functor into an abelian group $Q_i$ satisfying property \eqref{l.finiteorbits_a} of Lemma \ref{l.finiteorbits}.  Suppose additionally that the range of $\psi_i$ lies in a positive cone $Q_i^+ \subseteq Q_i$, and also the following stronger version of property \eqref{l.finiteorbits_b} of Lemma \ref{l.finiteorbits}:
\begin{itemize}
\item[(\ref{l.finiteorbits_b}$'$)]
For every finite subset $S \subseteq \psi_i(\Lambda_i)$ there is a finite subset $T \subseteq \psi_i(\Lambda_i)$ with $S \subseteq T$ such that for any finite set $E \subseteq \Lambda_i$, if $\psi_i(E) \subseteq T$ then 
\begin{itemize}
\item[(a)] $\psi(\vee E) \subseteq T$
\item[(b)] If $\alpha$, $\beta \in \Lambda_i$ with $\psi_i(\alpha\beta) \in T$, then $\psi_i(\alpha)$, $\psi_i(\beta) \in T$.
\end{itemize}
\end{itemize}
Let $Q = \bigoplus_{i \in I} Q_i$, and define $\psi : \Lambda \to Q$ by $\psi([\alpha_1, \ldots, \alpha_m]) = \sum_{j=1}^m \psi_{i_j}(\alpha_j)$, where $i_j \in I$ is such that $\alpha_j \in \Lambda_{i_j}$.  Then $\psi$ is a well-defined nondegenerate non-isotropic degree functor also satisfying properties \eqref{l.finiteorbits_a} and \eqref{l.finiteorbits_b} of Lemma \ref{l.finiteorbits}.
\end{Theorem}

\begin{proof}
It is clear that the definition of $\psi(\alpha)$ is unchanged when $\alpha$ is modified by the moves \eqref{d.amalgamation_itemone} - \eqref{d.amalgamation_itemfour} of Definition \ref{d.amalgamation}.  Thus $\psi$ is well-defined, and functoriality follows easily.  Nondegeneracy follows immediately from nondegeneracy of the $\psi_i$.  Let us prove that $\psi$ is non-isotropic.  Let $[\alpha^1]$, $[\alpha^2]$, $\ldots$, $[\beta^1]$, $[\beta^2]$, $\ldots \in \Lambda$ be such that $[\alpha^j][\alpha^{j+1}] = [\beta^j][\beta^{j+1}]$ for all $j$, and $\psi([\alpha^1]) = \psi([\beta^1])$.  Let $\alpha^j = (\alpha^j_1,\ldots,\alpha^j_{k_j})$ and $\beta^j = (\beta^j_1,\ldots,\beta^j_{\ell_j})$ in normal form.  Suppose that $\alpha^1 \not= \beta^1$.  We may as well assume that $\alpha^1_1 \not= \beta^1_1$, as otherwise they can be deleted.  We claim that $k_1 = 1$.  For suppose that $k_1 > 1$.  Then
\[
[\alpha^1_1,\ldots,\alpha^1_{k_1},\alpha^2_1,\ldots,\alpha^2_{k_2}] = [\alpha^1][\alpha^2] = [\beta^1][\beta^2] = [\beta^1_1,\ldots,\beta^1_{\ell_1},\beta^2_1,\ldots,\beta^2_{\ell_2}].
\]
By the uniqueness of the normal form, we must have $\ell_1 = 1$, $s(\beta^1_1) = r(\beta^2_1)$, and $\alpha^1_1 = \beta^1_1 \beta^2_1$ in $\Lambda_i$, for some $i \in I$.  But then \begin{align*}
\psi_i(\beta^1_1) = \psi(\beta^1) = \psi(\alpha^1) &= \psi_i(\alpha^1_1) + \psi([\alpha^1_2 \cdots \alpha^1_{k_1}]) \\
&= \psi_i(\beta^1_1 \beta^2_1) + \psi([\alpha^1_2 \cdots \alpha^1_{k_1}]) = \psi_i(\beta^1_1) + \psi_i(\beta^2_1) + \psi([\alpha^1_2 \cdots \alpha^1_{k_1}]).
\end{align*}
Thus, in particular, $\psi_i(\beta^2_1) = 0$ (here we use the hypothesis that the degree functors have ranges lying in positive cones).  Hence $\beta^2_1$ is a unit, contradicting the definition of the normal form of $\beta^2$.  It follows that $k_1 = 1$; similarly we have $\ell_1 = 1$.  Now by the uniqueness of the normal form (of $[\alpha^1][\alpha^2]$) we must have $\alpha^1 = \alpha^1_1$ and $\beta^1 = \beta^1_1$ in $\Lambda_i$.  Since $\alpha^1_1 \not= \beta^1_1$, we must have $s(\alpha^1_1) = r(\alpha^2_1)$, $s(\beta^1_1) = r(\beta^2_1)$, and $\alpha^1_1 \alpha^2_1 = \beta^1_1 \beta^2_1$.  Then $\alpha^2_1 \not= \beta^2_1$, by right cancellation.  We may now apply the above argument to $\alpha^j$, $\beta^j$ for $j \ge 2$.  We find that $\alpha^2 = \alpha^2_1$, $\beta^2 = \beta^2_1$, $s(\alpha^2_1) = r(\alpha^3_1)$, $s(\beta^2_1) = r(\beta^3_1)$, and $\alpha^2_1 \alpha^3_1 = \beta^2_1 \beta^3_1$.  Continuing, we find that the entire process occurs inside of $\Lambda_i$.  This contradicts the assumption that $\psi_i$ is non-isotropic.

Finally, we verify properties \eqref{l.finiteorbits_a} and \eqref{l.finiteorbits_b} of Lemma \ref{l.finiteorbits}.  For \eqref{l.finiteorbits_a}, let $E \subseteq \Lambda$ be infinite such that $\psi$ is constant on $E$.  Suppose that every pair in $E$ has a common extension.  Since $\psi$ is nondegenerate, no element of $E$ can extend another.  Thus by Lemma \ref{l.amalgamationextension}, every pair $\alpha$, $\beta \in E$ has normal forms $\alpha = [(\alpha_1, \ldots, \alpha_m, \gamma)]$, $\beta = [(\alpha_1, \ldots, \alpha_m, \delta)]$ such that $\gamma \Cap \delta$ in some $\Lambda_i$.  Thus the final terms of the normal forms of the elements of $E$ give an infinite subset $E' \subseteq \Lambda_i$ such that $\psi_i$ is constant on $E'$.  By property \eqref{l.finiteorbits_a} of $\psi_i$, there is a finite subset $F' \subseteq E'$ with $\bigvee F' = \emptyset$.  Then the corresponding finite set $F \subseteq E$ satisfies $\bigvee F = \emptyset$.  For \eqref{l.finiteorbits_b}, let $S \subseteq \psi(\Lambda)$ be finite.  Let $\pi_i : Q \to Q_i$ be the projection, and let $S_i = \pi_i(S)$.  Since $S$ is finite, there are only finitely many $i$ such that $S_i \not= \{ 0 \}$.  For such $i$, choose $T_i \subseteq \psi_i(Q_i)$ as in (\ref{l.finiteorbits_b}$'$).  By our hypothesis, we may assume that if $t \in T_i$ can be written as $t = t_1 + \cdots + t_k$ with $t_1$, $\ldots$, $t_k \in Q_i^+$, then $t_1$, $\ldots$, $t_k \in T_i$.  Let $T = \sum_i T_i$, a finite subset of $\psi(Q)$.  Let $E \subseteq \Lambda$ with $\psi(E) \subseteq T$.  If $\bigvee E \not= \emptyset$, we must have that $E$ is as above, in the proof of \eqref{l.finiteorbits_a}:  there is $\alpha \in \Lambda$ and $i \in I$ such that each element of $E$ is of the form $\alpha \gamma$ for some $\gamma \in \Lambda_i$, and $\bigvee E = \alpha \cdot \bigvee \{ \gamma : \alpha \gamma \in E \}$.  But then $\{ \psi_i(\gamma) : \alpha \gamma \in E \} \subseteq T_i$, so that $\psi_i(\bigvee E) \subseteq T$.
\end{proof}

\end{document}